\documentclass[11pt]{amsart}
\usepackage{amsmath}
\usepackage[psamsfonts]{amssymb}
\usepackage[all]{xy}
\usepackage{graphicx}
\usepackage[T1]{fontenc}
\usepackage[bitstream-charter]{mathdesign}
\usepackage{hyperref}
\usepackage{verbatim}
\usepackage{color}
\theoremstyle{definition}
\newtheorem{theorem}{Theorem}[section]
\newtheorem{prop}[theorem]{Proposition}
\newtheorem{lemma}[theorem]{Lemma}
\newtheorem{cor}[theorem]{Corollary}

\newtheorem{conj}[theorem]{Conjecture}

\theoremstyle{remark}
\newtheorem{dfn}[theorem]{Definition}
\newtheorem{remark}[theorem]{Remark}

\numberwithin{equation}{section}

\def\co{\colon\thinspace}

\def\ep{\epsilon}

\def\Q{\mathbb{Q}}
\def\R{\mathbb{R}}
\def\Z{\mathbb{Z}}
\def\N{\mathbb{N}}
\def\C{\mathbb{C}}

\title[Symplectic embeddings of ellipsoids into polydisks]{Infinite staircases in the symplectic embedding problem for four-dimensional ellipsoids into polydisks}
\author{Michael Usher}
\address{Department of Mathematics\\University of Georgia\\Athens, GA 30602}
\email{usher@uga.edu}

\begin{document}
\begin{abstract}
We study the symplectic embedding capacity function $C_{\beta}$  for ellipsoids $E(1,\alpha)\subset \R^4$ into dilates of polydisks $P(1,\beta)$ as both $\alpha$ and $\beta$ vary through $[1,\infty)$.  For $\beta=1$ results of \cite{FM} show that $C_{\beta}$ has an infinite staircase accumulating at $\alpha=3+2\sqrt{2}$, while for \emph{integer} $\beta\geq 2$ \cite{CFS} found that no infinite staircase arises.  We show that, for arbitrary $\beta\in (1,\infty)$, the restriction of $C_{\beta}$ to $[1,3+2\sqrt{2}]$ is determined entirely by the obstructions from \cite{FM}, leading $C_{\beta}$ on this interval to have a finite staircase with the number of steps tending to $\infty$ as $\beta\to 1$.  On the other hand, in contrast to \cite{CFS}, for a certain doubly-indexed sequence of irrational numbers $L_{n,k}$ we find that $C_{L_{n,k}}$ has an infinite staircase; these $L_{n,k}$ include both numbers that are arbitrarily large and numbers that are arbitrarily close to $1$, with the corresponding accumulation points respectively arbitrarily large and arbitrarily close to $3+2\sqrt{2}$.  
\end{abstract} 

\maketitle
\section{Introduction}

It is now understood that questions about when one domain in $\mathbb{R}^{2n}$ symplectically embeds into another often have quite intricate answers.  The best known example of this is the characterization from \cite{MS} of when one four-dimensional ellipsoid embeds into a ball. Writing \[ E(a,b)=\left\{(w,z)\in \C^2\left|\frac{\pi |w|^2}{a}+\frac{\pi |z|^2}{b}\leq 1\right.\right\}, \] McDuff and Schlenk completely describe the embedding capacity function \[ C^{\mathrm{ball}}(\alpha)=\inf\{\lambda|(\exists\mbox{ a symplectic embedding }E(1,\alpha)\hookrightarrow E(\lambda,\lambda))\}, \] showing that, on the interval $[1,\tau^4)$ where $\tau$ is the golden ratio, $C^{\mathrm{ball}}$ is given by an ``infinite staircase'' made up of piecewise linear steps; for $\alpha>\tau^4$, $C^{\mathrm{ball}}(\alpha)$ is given by either by the volume bound $\sqrt{\alpha}$ or one of a finite list of piecewise linear functions.  

In this paper we consider instead embeddings of four-dimensional ellipsoids into four-dimensional \emph{polydisks} \[ P(a,b)=\left\{(w,z)\in\C^2|\pi|w|^2\leq a,\,\pi|z|^2\leq b\right\}.\]  For any given $\beta\geq 1$, we consider the embedding capacity function $C_{\beta}\co [1,\infty)\to \R$ defined by \begin{equation}
\label{cbdef} 
C_{\beta}(\alpha)=\inf\{\lambda|(\exists\mbox{ a symplectic embedding }E(1,\alpha)\hookrightarrow P(\lambda,\lambda\beta))\}. 
\end{equation}  So the fact that symplectic embeddings are volume-preserving implies the ``volume bound'' $C_{\beta}(\alpha)\geq \sqrt{\frac{\alpha}{2\beta}}$.    The function $C_1$ was completely described in \cite{FM}, and was found to be qualitatively similar to the McDuff-Schlenk function $C^{\mathrm{ball}}$, with an infinite staircase followed by a finite alternating sequence of piecewise linear steps and intervals on which it coincides with the volume bound; in this case the infinite staircase occupies the interval $[1,3+2\sqrt{2})$.  More recently, for all integer $\beta\geq 2$ the function $C_{\beta}$ was found to have a rather simpler description in \cite{CFS}:  in this case there is no infinite staircase and the function coincides with the volume bound on all but finitely many intervals where it is piecewise linear, with the piecewise linear steps fitting into a fairly simple pattern as $\beta$ varies.  

The contrast between the complexity of the function $C_1$ and the simplicity of $C_{\beta}$ for integer $\beta\geq 2$ raises a number of questions, some of which we answer here.  First, we determine how the infinite staircase that describes $C_1|_{[1,3+2\sqrt{2})}$ disappears as the parameter $\beta$ is adjusted away from $1$.  In fact, we show that for \emph{all} real $\beta$, the restriction of $C_{\beta}$ to $[1,3+2\sqrt{2}]$ is in a sense ``as simple as possible'' given the results of \cite{FM} concerning $C_1$: the obstructions to symplectic embeddings (arising from a specific sequence of exceptional spheres in blowups of $S^2\times S^2$) that give rise to the Frenkel-M\"uller staircase are the only obstructions needed to understand $C_{\beta}(\alpha)$ for any real $\beta$ and any $\alpha\in [1,3+2\sqrt{2}]$ (see Theorem \ref{fmsup}).  By directly inspecting these obstructions one can see that, for any given $\beta>1$, only finitely many of them will actually be relevant, and indeed we find a sequence $b_m\searrow 1$ such that the graph of $C_{\beta}|_{[1,3+2\sqrt{2}]}$ consists of exactly $m$ steps whenever $\beta\in [b_m,b_{m-1})$.  

Complementing this, we show that once $\alpha$ becomes larger than $3+2\sqrt{2}$ the obstructions from \cite{FM} and \cite{CFS} are quite far from being sufficient to describe $C_{\beta}|_{[1,\alpha]}$ for all $\beta$.  The main ingredient in this is a triply-indexed family of exceptional spheres $A_{i,n}^{(k)}$ in blowups of $S^2\times S^2$; for very small values of $i$ these have some overlap with the classes from \cite{FM} and \cite{CFS}, but otherwise they are new.  If one fixes integers $n\geq 2,k\geq 0$ and varies $i$, the resulting classes can be used to show that, for certain irrational numbers $L_{n,k}$ (see (\ref{Lnkdef})), the function $C_{L_{n,k}}$ has an infinite staircase, accumulating at the value $\alpha=S_{n,k}>1$ characterized by the identity $\frac{(1+S_{n,k})^2}{S_{n,k}}=\frac{2(1+L_{n,k})^2}{L_{n,k}}$.  Fixing $n$, it holds that $L_{n,k}\searrow 1$ as $k\to\infty$, and hence that $S_{n,k}\searrow 3+2\sqrt{2}$ as $k\to\infty$.  On the other hand, setting $k=0$ we have $L_{n,0}=\sqrt{n^2-1}$, so there are arbitrarily large $\beta$ (which even become arbitrarily close to integers) for which $C_{\beta}$ has an infinite staircase, a counterpoint to the result of \cite{CFS} that $C_{\beta}$ never has an infinite staircase for integer $\beta\geq 2$.

For $i\geq 2$, the obstructions from our classes $A_{i,n}^{(0)}$ give larger lower bounds for $C_{L_{n,0}}(\alpha)$ for $\alpha=\frac{c_{i,n}}{d_{i,n}}$ with notation as in (\ref{abcd}) than do any of the classes denoted $E_m,F_m$ in \cite{CFS}.  Since $L_{n,0}=\sqrt{n^2-1}>2$ for $n\geq 3$, this gives many counterexamples to \cite[Conjecture 1.5]{CFS}.

\subsection{Initial background and notation}\label{init}
Before stating our results more explicitly let us recall some of the facts that are the basis of our analysis; these will largely be familiar to readers of \cite{MS},\cite{FM},\cite{CFS}.  The first main point is that, if $\frac{b}{a}\in\Q$, the existence of a symplectic embedding $E(a,b)^{\circ}\hookrightarrow P(c,d)^{\circ}$ from the interior of an ellipsoid into the interior of a polydisk is equivalent to the existence of a certain ball packing, dictated in part by the so-called \textbf{weight sequence} $\mathcal{W}(a,b)$ of $E(a,b)$. Here $\mathcal{W}(a,b)$ is determined recursively by setting $\mathcal{W}(x,0)=\mathcal{W}(0,x)$ equal to the empty sequence and, if $x\leq y$, setting $\mathcal{W}(x,y)$ and $\mathcal{W}(y,x)$ both equal to the result of prepending $x$ to the sequence $\mathcal{W}(x,y-x)$. (The recursion terminates because we assume $\frac{b}{a}\in\Q.)$ For any $a\in \Q_{\geq 0}$ we let $w(a)=\mathcal{W}(1,a)$.  So for instance $\mathcal{W}(8,3)=(3,3,2,1,1)$ and $w(\frac{8}{3})=(1,1,\frac{2}{3},\frac{1}{3},\frac{1}{3})$.    Then \cite[Proposition 1.4]{FM}, which is based on the analysis in \cite{Mell}, asserts that $E(a,b)^{\circ}$ symplectically embeds into $P(c,d)^{\circ}$ if and only if there is a symplectic embedding of a disjoint union of balls:
 \[ B(c)^{\circ}\sqcup B(d)^{\circ}\sqcup\left(\coprod_{w\in \mathcal{W}(a,b)}B(w)^{\circ}\right) \hookrightarrow B(c+d)^{\circ}.\]  Here $B(x)$ denotes the four-dimensional ball of capacity $x$ (\emph{i.e.}  $B(x)=E(x,x)$).  

In turn, as is explained at the end of the introduction to \cite{Mell} based on \cite{MP},\cite{Liu}, \cite{Bi}, a disjoint union $B(a_0)\sqcup\cdots\sqcup B(a_N)$ of closed balls symplectically embeds into $B(\lambda)^{\circ}$ if and only if there is a symplectic form $\omega$ on the complex $(N+1)$-fold blowup $X_{N+1}$ of $\mathbb{C}P^2$ whose associated first Chern class agrees with the one induced by the standard complex structure on $X_{N+1}$ (namely $3L-\sum_iE_i$ where $L$ is Poincar\'e dual to the hyperplane class and the $E_i$ are the Poincar\'e duals of the exceptional divisors), and which endows the the standard hyperplane class with area $\lambda$ and the respective exceptional divisors with areas $a_0,\ldots,a_N$.  

Let us write $\mathcal{C}_K(X_{N+1})$ for the set of cohomology classes of symplectic forms on $X_{N+1}$ having  associated first Chern class $3L-\sum_iE_i$, and denote the closure of this set by $\bar{\mathcal{C}}_K(X_{N+1})$.  Also write a general element of $H^2(X_{N+1};\R)$ as \[ dL-\sum_{i=0}^{N}t_iE_i=\langle d;t_0,\ldots,t_N\rangle.\]  In this notation it follows easily from the above facts that:

\begin{prop}[\cite{FM},\cite{Mell}] \label{background} Let $\alpha\in \Q,\beta,\lambda\in \R$ with $\alpha,\beta\geq 1$ and $\lambda>0$, and write the weight sequence $w(\alpha)=\mathcal{W}(1,\alpha)$ as $w(\alpha)=(x_2,\ldots,x_{N})$.  Then the following are equivalent:
\begin{itemize}
\item[(i)] $\lambda\geq C_{\beta}(\alpha)$.
\item[(ii)] $\langle \lambda(\beta+1);\lambda \beta,\lambda,x_2,\ldots,x_N\rangle\in \bar{\mathcal{C}}_K(X_{N+1})$.  
\end{itemize}
\end{prop}

Moreover, by \cite[Theorem 3]{Liu}, we have \begin{equation}\label{poscrit} \bar{\mathcal{C}}_K(X_{N+1})=\left\{c\in H^2(X_{N+1};\R)|c^2\geq 0,\,c\cdot E\geq 0\mbox{ for all }E\in \mathcal{E}_{N+1}\right\},\end{equation}  where $\mathcal{E}_{N+1}$ denotes the set of \emph{exceptional classes} in $X_{N+1}$, \emph{i.e.} the classes Poincar\'e dual to symplectically embedded spheres of self-intersection $-1$.    (Applying \cite[Theorem 3]{Liu} directly we would also need to check that $c\cdot L\geq 0$, but since  $L=(L-E_0-E_1)+E_0+E_1$ is a sum of elements of $\mathcal{E}_{N+1}$ this follows from the other conditions.)

To study embeddings into polydisks it is often helpful to use different coordinates on $H^2(X_{N+1};\R)$, as described in \cite[Remark 3.7]{FM}.  Recall that, for $N\geq 1$, our $(N+1)$-fold blowup $X_{N+1}$ of $\mathbb{C}P^2$ can also be viewed as an $N$-fold blowup of $S^2\times S^2$ (say with exceptional divisors $E'_{1},\ldots,E'_N$), with the Poincar\'e duals $S_1$ and $S_2$ of the $S^2$ factors corresponding respectively to $L-E_0$ and $L-E_1$ and with the $E'_i$ corresponding to $L-E_0-E_1$ for $i=1$ and to $E_i$ for $i\geq 2$. Let us accordingly write \[ \left(d,e;m_1,\ldots,m_N\right)=dS_1+eS_2-\sum_{i=1}^{N}m_iE'_i\in H^2(X_{N+1};\R) \] (note that we are using angle brackets when we use ``$\C P^2$ coordinates'' and parentheses when we use ``$S^2\times S^2$ coordinates'').   Hence the representations in our two bases are related by: \begin{equation}\label{convert1}
(d,e;m_1,m_2\ldots,m_N)=\langle d+e-m_1; d-m_1, e-m_1,m_2,\ldots,m_N\rangle;
\end{equation}
\begin{equation}\label{convert2}
\langle r; s_0,s_1,s_2,\ldots,s_N\rangle = (r-s_1,r-s_0;r-s_0-s_1,s_2,\ldots,s_N).
\end{equation}

Condition (ii) in Proposition \ref{background} can then be rephrased as \begin{equation}\label{s2s2crit} \left(\lambda \beta,\lambda;0,w(\alpha)\right)\in \bar{\mathcal{C}}_K(X_{N+1};\R) \end{equation} where here and throughout the rest of the paper we abuse notation slightly by writing the weight sequence $w(\alpha)=(x_2,\ldots,x_N)$ as though it were a single entry in the coordinate expression of our cohomology class, so that $(\lambda \beta,\lambda;0,w(\alpha))$ is shorthand for $(\lambda \beta,\lambda;0,x_2,\ldots,x_N)$.  By considering small-weight blowups and taking a limit it is easy to see that (\ref{s2s2crit}) is equivalent to \[ (\lambda \beta,\lambda;w(\alpha))\in \bar{\mathcal{C}}_K(X_N).\] 

Now if $w(\alpha)=(x_2,\ldots,x_N)$ then $\sum_{i} x_{i}^{2}=a$; conceptually this is because one can obtain the weight sequence by subdividing a $1$-by-$a$ rectangle into squares of sidelength $x_i$.  Thus the self-intersection of the class $(\lambda \beta,\lambda;w(\alpha))$ is equal to $2\lambda^2 \beta-\alpha$ and so is nonnegative if and only if $\lambda$ obeys the volume bound $\lambda\geq \sqrt{\frac{\alpha}{2\beta}}$ alluded to earlier.  Now suppose that $E=(d,e;\vec{m})\in \mathcal{E}_N$ where $\vec{m}\in \mathbb{Z}^n$.  One example of such an element $E$ is $E'_i=(0,0;0,\ldots,-1,\ldots,0)$, which has nonnegative intersection number with $(\lambda \beta,\lambda;w(\alpha))$ since all entries of $w(\alpha)$ are nonnegative.  All other elements of $\mathcal{E}_N$ have $d,e\geq 0$ (by positivity of intersections with embedded holomorphic spheres Poincar\'e dual to $S_1$ and $S_2$), all $m_i\geq 0$ (by positivity of intersections with $E'_i$) and $d+e> 0$ (given that $m_i\geq 0$, this follows from the Chern number of $E$ being $1$).  The intersection number of such a class with $(\lambda \beta,\lambda;w(\alpha))$ is equal to $\lambda(d+\beta e)-w(\alpha)\cdot \vec{m}$ and so is nonnegative if and only if $\lambda\geq \frac{w(\alpha)\cdot \vec{m}}{d+\beta e}$.  Recalling that elements of $\mathcal{E}_N$ have Chern number $1$ and self-intersection $-1$, we accordingly make the following definition:

\begin{dfn}\label{obsdef}
Let $E=(d,e;\vec{m})\in H^2(X_{N};\Z)$ be either equal to some $E'_i$ or have the properties that such that $c_1(TX_N)\cdot E=1$, $E\cdot E=-1$,  and all $m_i\geq 0$ (and hence\footnote{indeed the condition on the Chern number shows that $2(d+e)\geq 2(d+e)-\sum m_i>0$, and then if either of $d$ and $e$ were negative we would have $E\cdot E\leq 2de\leq -2$ which is not the case} $d,e\geq 0$ with $d+e>0$).  Let $\alpha\in \Q$ have weight sequence $w(\alpha)$ of length $N-1$ and let $\beta\in [1,\infty)$.  The \textbf{obstruction from $E$ at $(\alpha,\beta)$} is \[ \mu_{\alpha,\beta}(E)=\left\{\begin{array}{ll} 0 & \mbox{ if }E=E'_i \\ \frac{w(\alpha)\cdot \vec{m}}{d+\beta e} & \mbox{ otherwise}\end{array}\right..\]
\end{dfn}

Proposition \ref{background} and (\ref{poscrit}) therefore imply that:

\begin{cor}\label{cbe}
For any $\beta\geq 1$ and any $\alpha$ whose weight sequence has length $N-1$ we have: \[ C_{\beta}(\alpha)=\max\left\{\sqrt{\frac{\alpha}{2\beta}},\sup_{E\in \mathcal{E}_N}\mu_{\alpha,\beta}(E)\right\}.\]
\end{cor}

In fact, it follows from \cite[Section 6.1]{Bi} that if we let $\tilde{\mathcal{E}}_N$ be the set of classes  $E\in H^2(X_N;\Z)$ obeying the assumptions in the first sentence of Definition \ref{obsdef}, then we continue to have \begin{equation}\label{biggerE}  C_{\beta}(\alpha)=\max\left\{\sqrt{\frac{\alpha}{2\beta}},\sup_{E\in \tilde{\mathcal{E}}_N}\mu_{\alpha,\beta}(E)\right\}.\end{equation}  Thus enlargng the set $\mathcal{E}_N$ to $\tilde{\mathcal{E}}_N$ does not affect the supremum on the right-hand side above.  This sometimes will save us the trouble of checking that certain families of classes that are easily seen to lie in $\tilde{\mathcal{E}}_N$ in fact lie in $\mathcal{E}_N$.  That said, it is sometimes important to know that a class lies in $\mathcal{E}_N$, because the fact that distinct elements of $\mathcal{E}_N$ have nonnegative intersection number often provides useful constraints.

As $\alpha$ varies through $\mathbb{Q}$, the length of its weight sequence also varies, so the value $N$ appearing in Corollary \ref{cbe} and in (\ref{biggerE}) depends on $\alpha$.  To avoid keeping track of this dependence, it is better to work in the union of all of the $H^2(X_N;\R)$, with two elements in this union regarded as equivalent if one can be obtained from the other by pullback under the map $X_{N'}\to X_N$ given by blowing down the last $N'-N$ exceptional divisors when $N'>N$.  Let $\mathcal{H}^2$ denote this union (more formally, \[ \mathcal{H}^2:=\varinjlim_{N\to\infty} H^2(X_N;\R) \] for the directed system whose structure maps $H^2(X_N;\R)\to H^2(X_{N'};\R)$ are the pullbacks associated to the blowdowns $X_{N'}\to X_N$).  So any element of $\mathcal{H}^2$ can be expressed as $(d,e;m_1,\ldots,m_N)$ (or, if one prefers, as $\langle r;s_0,\ldots,s_N\rangle$) for some finite collection of real numbers $d,e,m_1,\ldots,m_N$, and $(d,e;m_1,\ldots,m_N)$ and $(d,e;m_1,\ldots,m_N,0)$ are expressions of the same element of $\mathcal{H}^2$.  The Chern number of such an element is $2(d+e)-\sum m_i$ and its self-intersection is $2de-\sum m_{i}^{2}$; in particular these are both independent of the choice of representative of the equivalence class.  

It is easy to see that if $E\in\mathcal{E}_N$ (resp. $E\in\tilde{\mathcal{E}}_N$) then the image of $E$ under the blowdown-induced map $H^2(X_N;\R)\to H^2(X_{N'};\R)$ for $N'>N$ belongs to $\mathcal{E}_{N'}$ (resp. to $\tilde{\mathcal{E}}_{N'}$).   Let $\mathcal{E}$ and $\tilde{\mathcal{E}}$, respectively, be the unions of the images under the canonical map $H^2(X_N;\R)\to\mathcal{H}^2$ of the various $\mathcal{E}_N$ and $\tilde{\mathcal{E}}_N$, so an element $E\in\mathcal{E}$ can be regarded as Poincar\'e dual to an embedded symplectic sphere of self-intersection $-1$ for all sufficiently large $N$.    

Definition \ref{obsdef} extends to arbitrary $\alpha\in\Q\cap [1,\infty)$ and arbitrary $E\in \tilde{\mathcal{E}}$: we simply need to interpret the dot product $w(a)\cdot \vec{m}$ when $E=(d,e;\vec{m})$, and if $w(a)=(x_2,\ldots,x_N)$ and $\cdot{\vec{m}}=(m_2,\ldots,m_{N'})$ (where we can assume $N'\geq N$ by appending zeros to $\vec{m}$ which does not change the corresponding element of $\tilde{\mathcal{E}}$) we use the obvious convention that $w(a)\cdot\vec{m}=\sum_{i=2}^{N}x_im_i$.  With this definition of $\mu_{\alpha,\beta}(E)$ for arbitrary $E\in\tilde{\mathcal{E}}$ and $\alpha\in \Q\cap [1,\infty)$ it follows easily from Corollary \ref{cbe} and from (\ref{biggerE}) that \begin{equation}\label{dirlim}C_{\beta}(\alpha)=\max\left\{\sqrt{\frac{\alpha}{2\beta}},\sup_{E\in \mathcal{E}}\mu_{\alpha,\beta}(E)\right\}=\max\left\{\sqrt{\frac{\alpha}{2\beta}},\sup_{E\in \tilde{\mathcal{E}}}\mu_{\alpha,\beta}(E)\right\}.
\end{equation}
Since $C_{\beta}$ is easily seen to be continuous this is enough to characterize $C_{\beta}(\alpha)$ for all real $\alpha\geq 1$.

The great majority of elements of $\mathcal{E}$ or $\tilde{\mathcal{E}}$ that we will consider in this paper have a rather special form:

\begin{dfn}\label{qpdef}
An element $E\in \mathcal{H}^2$ is said to be \textbf{quasi-perfect} if both $E\in\tilde{\mathcal{E}}$ and there are nonnegative integers $a,b,c,d$ such that \[ E=(a,b;\mathcal{W}(c,d)).\]
 Such an element is said to be \textbf{perfect}  if additionally $E\in\mathcal{E}$.  \end{dfn}

There are quasi-perfect classes that are not perfect, such as $(31,14;\mathcal{W}(79,11))=(31,14;11^{\times 7},2^{\times 5},1^{\times 2})$,\footnote{Throughout the paper we use the usual convention that $z^{\times \ell}$ means that $z$ is repeated $\ell$ times, so $(31,14;11^{\times 7},2^{\times 5},1^{\times 2})=(31,14;11,11,11,11,11,11,11,2,2,2,2,2,1,1)$.} which cannot lie in $\mathcal{E}$ because it has negative intersection number with the element $(3,1;1^{\times 7})$ of  $\mathcal{E}$.

The discussion in \cite[Section 2.1]{MS} shows that for each $E\in\tilde{\mathcal{E}}$ and $\beta\geq 1$ the function $\alpha\mapsto \mu_{\alpha,\beta}(E)$ is piecewise linear, though in some cases it can be somewhat complicated.  For quasi-perfect classes $E=(a,b;\mathcal{W}(c,d))$, it will often be sufficient for us to consider a simpler piecewise linear function $\Gamma_{\cdot,\beta}(E)$ which, like $\mu_{\cdot,\beta}(E)$, provides a lower bound for $C_{\beta}$:

\begin{prop} \label{Gammadef}
If $E=(a,b;\mathcal{W}(c,d))\in\tilde{\mathcal{E}}$ is quasi-perfect and $\alpha,\beta\geq 1$ define \[ \Gamma_{\alpha,\beta}(E)=\left\{\begin{array}{ll} \frac{d\alpha}{a+\beta b} & \mbox{ if }\alpha\leq\frac{c}{d} \\ \frac{c}{a+\beta b} & \mbox{ if }\alpha\geq \frac{c}{d}\end{array}\right..\] Then $C_{\beta}(\alpha)\geq \Gamma_{\alpha,\beta}(E)$.
\end{prop}

\begin{proof} Observe first that $w(\frac{c}{d})=\frac{1}{d}\mathcal{W}(c,d)$ and $w(\frac{c}{d})\cdot w(\frac{c}{d})=\frac{c}{d}$, so that \[ \mu_{\frac{c}{d},\beta}\left((a,b;\mathcal{W}(c,d))\right)= \frac{c}{a+\beta b}.\]  Hence \begin{equation}\label{vertex} C_{\beta}\left(\frac{c}{d}\right)\geq \frac{c}{a+\beta b}.\end{equation}  But $C_{\beta}$ is trivially a monotone increasing function (increasing $\alpha$ enlarges the codomain of the desired embedding) while $C_{\beta}$ also satisfies the sublinearity property $C_{\beta}(t\alpha)\leq tC_{\beta}(\alpha)$ for $t\geq 1$, because if there is a symplectic embedding $E(1,\alpha)\hookrightarrow P(\lambda,\lambda \beta)$ then by scaling we obtain a composition of symplectic embeddings $E(1,t\alpha)\hookrightarrow E(t,t\alpha)\hookrightarrow P(t\lambda,t\lambda \beta)$ (cf. \cite[Lemma 1.1.1]{MS}).  The proposition then follows from (\ref{vertex}) and these monotonicity and sublinearity properties. 
\end{proof}

\subsection{The disappearing Frenkel-M\"uller staircase}\label{fmintro}

In \cite{FM} the authors introduce a sequence of perfect classes that we denote $\{FM_{n}\}_{n=-1}^{\infty}$ (these are the classes called $E(\beta_n)$ in \cite[Section 5.1]{FM}; we recall the formula in (\ref{fmclass}) below), and \cite[Theorem 1.3 (i)]{FM} can be expressed in our notation as stating that \[ C_1(\alpha)=\sup_n\{\Gamma_{\alpha,1}(FM_n)|n\geq -1\} \quad\mbox{for }\alpha\in [1,3+2\sqrt{2}].\] See Figure \ref{fmfig}.

\begin{center}
\begin{figure}
\includegraphics[width=6 in]{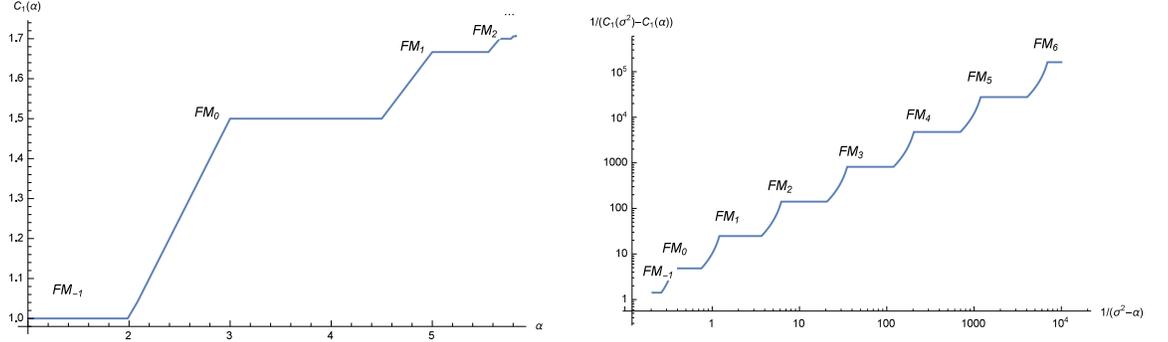}
\caption{A plot of the Frenkel-M\"uller infinite staircase $C_1|_{[1,\sigma^2]}$ where $\sigma^2=3+2\sqrt{2}$, together with a log-log plot which makes more visible some of the steps that accumulate at $\sigma^2$. Each  step in each of the plots is labeled by the Frenkel-M\"uller class $FM_n$ having the property that $C_1$ coincides with $\Gamma_{\cdot,1}(FM_n)$ on that step.}
\label{fmfig}
\end{figure}
\end{center}

 Our first main result is that the analogous statement 
continues to hold for $C_{\beta}$ with $\beta>1$:

\begin{theorem}\label{fmsup} For any $\beta\geq 1$ and any $\alpha\in [1,3+2\sqrt{2}]$ we have \begin{equation}\label{fmeq} C_{\beta}(\alpha)=\sup_n\{\Gamma_{\alpha,\beta}(FM_n)|n\geq -1\}.\end{equation}
\end{theorem}

The right-hand side of (\ref{fmeq}) can be computed explicitly, and its behavior when $\beta=1$ is different from its behavior when $\beta>1$.  The perfect classes $FM_n$ have the form $(x_n,y_n;\mathcal{W}(c_n,d_n))$ where the $\frac{c_n}{d_n}$ form an increasing sequence.  It is also true that the $\Gamma_{\frac{c_n}{d_n},1}(FM_n)=\frac{c_n}{x_n+y_n}$ form an increasing sequence, in view of which the graph of $\alpha\mapsto \sup_n\{\Gamma_{\alpha,1}(FM_n)|n\geq -1\} $ forms an infinite staircase as described in \cite{FM}.  However for any $\beta>1$ there is a value of $n$ (depending on $\beta$, and always odd) for which $\Gamma_{\frac{c_n}{d_n},\beta}(FM_n)$ is maximal, as a result of which the right-hand side of (\ref{fmeq}) reduces to a maximum over a finite set.  

A bit more specifically, let $\{P_n\}_{n=0}^{\infty}$ and $\{H_n\}_{n=0}^{\infty}$ be the Pell numbers and the half-companion Pell numbers respectively (see Section \ref{pellprelim}) and for $n\geq -1$ let \[ b_n=\left\{\begin{array}{ll} \frac{P_{n+2}+1}{P_{n+2}-1} & \mbox{if $n$ is even}\\ \frac{H_{n+1}+1}{H_{n+1}-1} &\mbox{if $n$ is odd}\end{array}\right.. \]  The $b_n$ form a decreasing sequence that converges to $1$, with the first few values being given by $b_{-1}=\infty$, $b_0=3$, $b_1=2$, $b_2=\frac{13}{11}$, $b_3=\frac{9}{8}$, $b_4=\frac{71}{69}$.  We show in Proposition \ref{supreduce} that, for all $\alpha$, 
\[ \sup_n\{\Gamma_{\alpha,\beta}(FM_n)|n\geq -1\}=\max\{\Gamma_{\alpha,\beta}(FM_{-1}),\Gamma_{\alpha,\beta}(FM_0),\ldots,\Gamma_{\alpha,\beta}(FM_{2k-1})\}\mbox{ for }b\in [b_{2k},b_{2k-1}],\] and  \begin{align*}& \sup_n\{\Gamma_{\alpha,\beta}(FM_n)|n\geq -1\}=\max\{\Gamma_{\alpha,\beta}(FM_{-1}),\Gamma_{\alpha,\beta}(FM_0),\ldots,\Gamma_{\alpha,\beta}(FM_{2k-1}),\Gamma_{\alpha,\beta}(FM_{2k+1})\}\\ & \qquad\qquad\qquad \qquad\mbox{ for }b\in [b_{2k+1},b_{2k}].\end{align*}
As $\beta$ increases within the interval $[b_{2k+1},b_{2k}]$, the $\Gamma_{\alpha,\beta}(FM_n)$ all become smaller since our  codomain $P(1,\beta)$ is expanding, but the maximal value of $\Gamma_{\alpha,\beta}(FM_{2k-1})$ decreases more slowly than does the maximal value of $\Gamma_{\alpha,\beta}(FM_{2k+1})$, matching it precisely when $\beta=b_{2k}$. In particular the step in the graph of $C_{\beta}$ corresponding to $FM_{2k+1}$ disappears as $\beta\nearrow b_{2k}$, being overtaken by the step corresponding to $FM_{2k-1}$.  Similarly, as $\beta\nearrow b_{2k-1}$, the step corresponding to $FM_{2k-2}$ is overtaken by the step corresponding to $FM_{2k-3}$ (and the step corresponding to $FM_{2k-1}$ remains as the final step, surviving until $\beta$ reaches $b_{2k-2}$).  See Figure \ref{fmgraphs}. Once $\beta$ rises above $b_0=3$, only the ``step'' (more accurately described as a floor) corresponding to $FM_{-1}$ remains. In fact one has $FM_{-1}=(1,0;1)$ so that $\Gamma_{\alpha,\beta}(FM_{-1})=1$ for all $\alpha,\beta$; thus for $\beta\geq 3$ Theorem \ref{fmsup} just says that $C_{\beta}(\alpha)=1$ for $\alpha\leq 3+2\sqrt{2}$, \emph{i.e.} that the bound given by the non-squeezing theorem is sharp for all such $\alpha$.  (This latter fact is easily deduced from well-known results, since $3+2\sqrt{2}<6$ and for $\beta\geq 3$ there is a symplectic embedding of $E(1,6)^{\circ}$ into $P(1,\beta)$,  see \emph{e.g.} \cite[Remark 1.2.1]{CFS}.)

\begin{center}
\begin{figure}
\includegraphics[width=6 in]{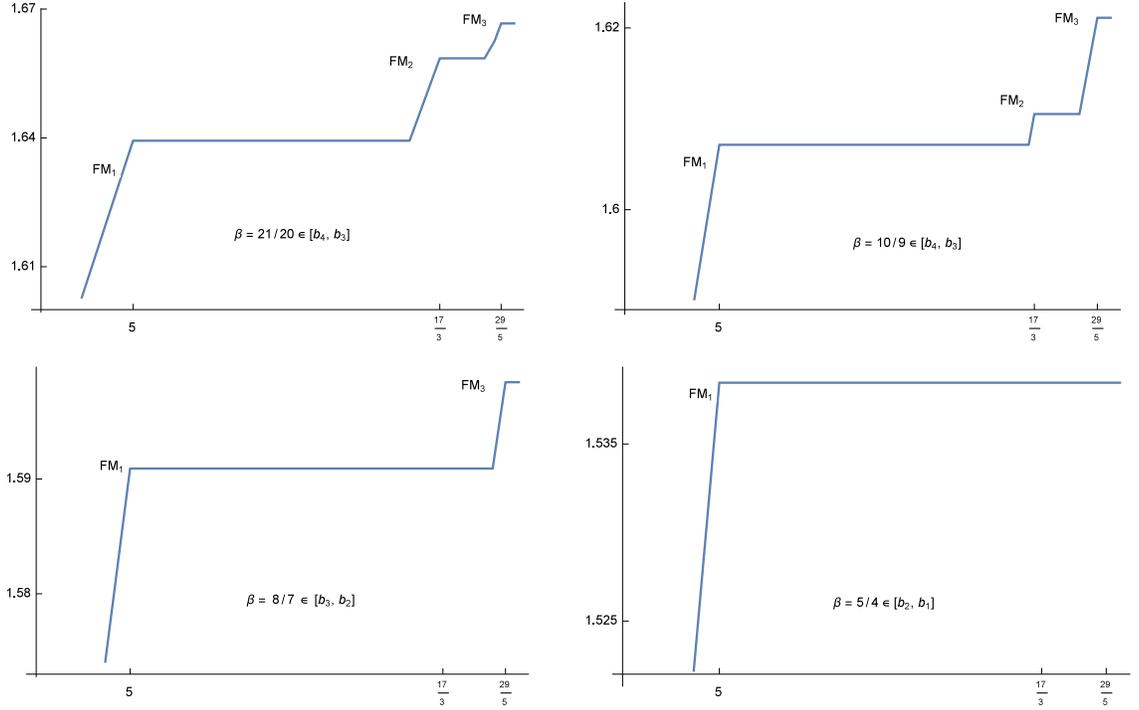}
\caption{Plots of the functions $C_{\beta}|_{[4.8,3+2\sqrt{2}]}$ for selected values of $\beta$. For $b_4<\beta<b_3$, $C_{\beta}|_{[1,\sigma^2]}$ is given as the maximum of obstructions arising from the Frenkel-M\"uller classes $FM_{-1},FM_{0},FM_{1},FM_{2},FM_{3}$ (the first two of which are relevant only for values of $\alpha$ outside the domain of these plots).  The obstruction from $FM_1$ approaches the obstruction from $FM_2$ as $\beta$ approaches $b_3=\frac{9}{8}$, and these obstructions cross once $\beta>b_3$ so that $FM_2$ is no longer relevant. Once $\beta>b_2=\frac{13}{11}$ the obstruction from $FM_1$ likewise overtakes the obstruction from $FM_{3}$.  Increasing $\beta$ still further would lead to the obstruction from $FM_{-1}$ overtaking that from $FM_0$ when $\beta$ crosses $b_1=2$ and overtaking that from $FM_1$ when $\beta$ crosses $b_0=3$.}
\label{fmgraphs}
\end{figure}
\end{center}

\begin{remark}
In fact, our proof shows that, when $\beta>1$, the equality $C_{\beta}(\alpha)=\sup_n\{\Gamma_{\alpha,\beta}(FM_n)|n\geq -1\}$ continues to hold
for $\alpha<\alpha_0(\beta)$ for an upper bound $\alpha_0(\beta)$ that is somewhat larger than $3+2\sqrt{2}$, and converges to $3+2\sqrt{2}$ as $\beta\to 1$.  (Explicit, though typically not optimal, values for $\alpha_0(\beta)$ can be read off from Propositions \ref{p2k} and \ref{lastplat}.)  It follows from our other main results that such a bound $\alpha_0(\beta)$ must depend on $\beta$, as there are pairs  $(\alpha,\beta)$ arbitrarily close to $(3+2\sqrt{2},1)$ for which (\ref{fmeq}) is false. (For instance, in the notation used elsewhere in the paper, one can take $(\alpha,\beta)=(S_{n,k},L_{n,k})$ for large $k$ and any $n\geq 2$.)\end{remark}

\subsection{New infinite staircases}\label{intronew}

The analysis in \cite{MS},\cite{FM} shows that, for any $\alpha,\beta\geq 1$ such that $C_{\beta}(\alpha)$ exceeds the volume bound $\sqrt{\frac{\alpha}{2\beta}}$, there is a neighborhood of $\alpha$ on which $C_{\beta}$ is piecewise linear.  Let $\mathcal{S}_{\beta}$ denote the collection of affine functions $f\co \R\to\R$ having the property that there is an nonempty open set on which $C_{\beta}$ coincides with $f$.  Thus the graph of $C_{\beta}$ consists of segments which coincide with the graph of one of the functions from $\mathcal{S}_{\beta}$, collectively forming a sort of staircase, and other segments which coincide with the volume bound.  We say that $C_{\beta}$ \textbf{has an infinite staircase} if $\mathcal{S}_{\beta}$ is an infinite set.  In this case, we say that $\alpha\in \R$ is an \textbf{accumulation point} of the infinite staircase if for every neighborhood $U$ of $\alpha$ there are infinitely many $f\in \mathcal{S}_{\beta}$ such that $C_{\beta}$ coincides with $f$ on some nonempty open subset of $U$.

\begin{remark}
Arguing as in \cite[Corollary 1.2.4]{MS}, for any $E=(x,y;\vec{m})\in\mathcal{E}\setminus\cup_i\{E'_i\}$ we have $1=c_1(E)=2(x+y)-\sum_i m_i$, so that for any $\alpha\in\Q$ \[ \vec{m}\cdot w(\alpha)\leq \sum_i m_i<2(x+y)\] and so (for any $\beta\geq 1$) \[ \mu_{\alpha,\beta}(E)=\frac{\vec{m}\cdot w(\alpha)}{x+\beta y}\leq \frac{2(x+y)}{x+\beta y}\leq 2.\]   Thus if the volume bound $\sqrt{\frac{\alpha}{2\beta}}$ is at least $2$, \emph{i.e.} if $\alpha\geq 8\beta$, then $C_{\beta}(\alpha)$ is equal to the volume bound. It follows easily from this that if $C_{\beta}$ has an infinite staircase then this infinite staircase must have an accumulation point.   An unpublished argument communicated to the author by Cristofaro-Gardiner appears to imply that the only possible accumulation point for any such infinite staircase is the value $\alpha>1$ determined by the equation $\frac{(1+\alpha)^2}{\alpha}=\frac{2(1+\beta)^2}{\beta}$.
\end{remark}

Again denoting by $P_m$ and $H_m$ the Pell and half-companion Pell numbers respectively, for any integers $n\geq 2$ and $k\geq 0$ let 
  \[ L_{n,k}=\frac{H_{2k}(\sqrt{n^2-1}+1)+2nP_{2k}+(\sqrt{n^2-1}-1)}{H_{2k}(\sqrt{n^2-1}+1)+2nP_{2k}-(\sqrt{n^2-1}-1)} \] and \[S_{n,k}=\frac{(\sqrt{n^2-1}+1)P_{2k+1}+nH_{2k+1}}{(\sqrt{n^2-1}+1)P_{2k-1}+nH_{2k-1}} .\] In particular \[ L_{n,0}=\sqrt{n^2-1},\qquad S_{n,0}=\frac{\sqrt{n^2-1}+1+n}{\sqrt{n^2-1}+1-n}.\] Our second main result is the following, proven as part of Corollary \ref{nkstair}:
\begin{theorem}\label{stairmain} For any $n\geq 2$ and $k\geq 0$ the function $C_{L_{n,k}}$ has an infinite staircase, with accumulation point at $S_{n,k}$.
\end{theorem}

We also show in Corollary \ref{nkstair} that $\frac{(1+S_{n,k})^2}{S_{n,k}}=\frac{2(1+L_{n,k})^2}{L_{n,k}}$, consistently with Cristofaro-Gardiner's work.  

The proof of Theorem \ref{stairmain} makes use of a collection of perfect classes $A_{i,n}^{(k)}=(a_{i,n,k},b_{i,n,k};\mathcal{W}(c_{i,n,k},d_{i,n,k}))$ for $i,k\geq 0$ and $n\geq 2$.  (See (\ref{abcd}) and (\ref{ink}) for explicit formulas.)  For fixed $n$ and $k$ and varying $i$, the numbers $\frac{c_{i,n,k}}{d_{i,n,k}}$ form a strictly increasing sequence, and we see in (\ref{stairnhd}) that $C_{L_{n,k}}$ coincides on a neighborhood of each $\frac{c_{i,n,k}}{d_{i,n,k}}$ with the function $\Gamma_{\cdot,L_{n,k}}(A_{i,n}^{(k)})$.  This is enough to show that $C_{L_{n,k}}$ has an infinite staircase, though it does not determine the structure of the staircase in detail since does not address the behavior of $C_{L_{n,k}}$ outside of these neighborhoods.

At least for $k=0$, the infinite staircase of Theorem \ref{stairmain} does not consist only of steps given by the $\Gamma_{\cdot,L_{n,k}}(A_{i,n}^{(k)})$; indeed Proposition \ref{undervol} shows that $\sup_i\Gamma_{\alpha,L_{n,0}}(A_{i,n}^{(0)})$ is below the volume bound at certain points $\alpha\in \left(\frac{c_{i,n,0}}{d_{i,n,0}},\frac{c_{i+1,n,0}}{d_{i+1,n,0}}\right)$.  (We expect the analogous statement to be true for arbitrary $k$, and have confirmed this with computer calculations for all $n,k\leq 100$.)  In Section \ref{ahatsect} we introduce a different set of quasi-perfect classes\footnote{We expect these classes to all be perfect, and have checked this for many examples, but we do not have a general argument.} $\hat{A}_{i,n}^{(k)}$.  The fact that $A_{i,n}^{(k)}$ and $\hat{A}_{i,n}^{(k)}$ are all quasi-perfect implies that we have a lower bound \begin{equation}\label{lowerbound} C_{L_{n,k}}(\alpha)\geq \sup\left\{\Gamma_{\alpha,L_{n,k}}(A)|A\in \cup_{i=0}^{\infty}\{A_{i,n}^{(k)},\hat{A}_{i,n}^{(k)}\}\right\},\end{equation} and Conjecture \ref{fillconj} asserts that this inequality is in fact an equality for all $\alpha$ in the region $\left[\frac{c_{0,n,k}}{d_{0,n,k}},S_{n,k}\right]$ occupied by the staircase.  

Setting $k=0$, Proposition \ref{hatbeatsvol} shows that the right-hand side of (\ref{lowerbound}) is strictly greater than the volume bound for all $\alpha\in [\frac{c_{0,n,0}}{d_{0,n,0}},S_{n,0}]$.  Computer calculations show that the analogous statement continues to hold for all $n,k\leq 100$.  In particular, while it is in principle possible for an infinite staircase for some $C_{\beta}$ to include intervals on which $C_{\beta}$ coincides with the volume bound, our infinite staircases do not have this property at least when $k=0$ and $n$ is arbitrary, or when $n,k\leq 100$.   Indeed in contrast to the Frenkel-M\"uller staircase these infinite staircases do not even touch the volume bound at isolated points prior to the accumulation point $S_{n,k}$.

One can show that the interval of $\alpha$ on which $\Gamma_{\alpha,L_{n,k}}(A_{i,n}^{(k)})$ realizes the supremum on the right-hand side of (\ref{lowerbound}) has length with the same order of magnitude as $\frac{1}{P_{2k}^{2}\omega_{n}^{2i}}$ where $\omega_n=n+\sqrt{n^2-1}$, while the corresponding interval for $\Gamma_{\alpha,L_{n,k}}(\hat{A}_{i,n}^{(k)})$ is contained in $\left[\frac{c_{i,n,k}}{d_{i,n,k}},\frac{c_{i+1,n,k}}{d_{i+1,n,k}}\right]$ and has length bounded by a constant times $\frac{1}{P_{2k}^{2}\omega_{n}^{4i}}$.  Thus the steps corresponding to the $\hat{A}_{i,n}^{(k)}$ are between those corresponding to $A_{i,n}^{(k)}$ and $A_{i+1,n}^{(k)}$, and decay in size at a faster rate.  See Figure \ref{stairfig}.

\begin{center}
\begin{figure}
\includegraphics[width=5 in]{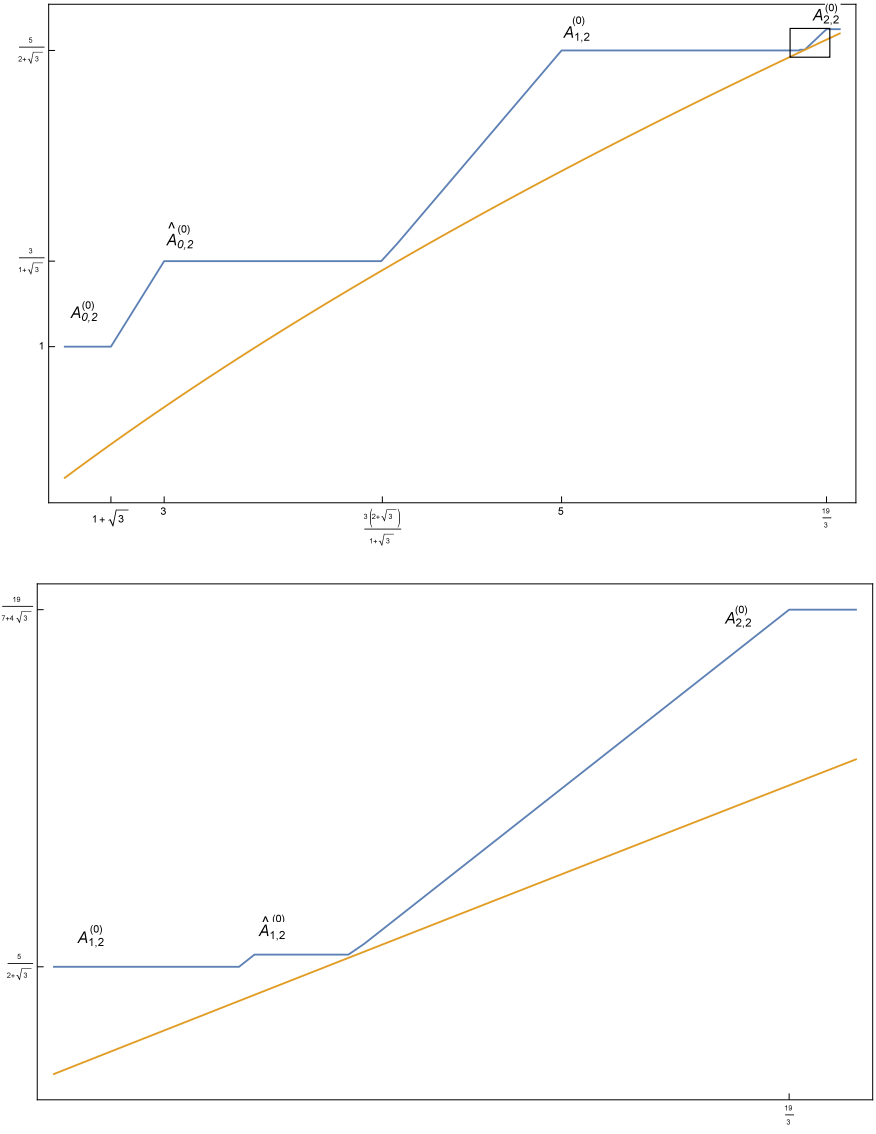}
\caption{Partial plots of our lower bound (\ref{lowerbound}) for $C_{L_{2,0}}$ (which we conjecture to be equal to $C_{L_{2,0}}$), with steps labelled by their corresponding elements of $\mathcal{E}$, together with the volume bound in orange.  The bottom plot is a magnification of the box in the upper right of the top plot; evidently such a magnification is needed in order to make the step corresponding to $\hat{A}_{1,2}^{(0)}=(11,7;\mathcal{W}(31,5))$ visible.}
\label{stairfig}
\end{figure}
\end{center}

Our infinite staircases for $C_{L_{n,k}}$ join together nicely with the picture in Section \ref{fmintro}.  As we see in Section \ref{connect}, for all $n$ it holds that $A_{0,n}^{(k)}=FM_{2k-1}$, so that the initial obstruction in our staircase coincides with one of the Frenkel-M\"uller obstructions.  Moreover when $n\geq 4$ Proposition \ref{Lsize} shows that $L_{n,k}$ lies in the interval $(b_{2k},b_{2k-1})$, and so $FM_{2k-1}$ is the last surviving obstruction in the Frenkel-M\"uller staircase for $C_{L_{n,k}}$.  For the remaining values of $n$, as we explain in Section \ref{connect}, Proposition \ref{Lsize} shows that $A_{0,3}^{(k)}=FM_{2k-1}$ is the penultimate surviving obstruction in the Frenkel-M\"uller staircase for $C_{L_{3,k}}$ (and the last one is $\hat{A}_{0,3}^{(k)}=FM_{2k+1}$), and that $A_{0,2}^{(k)}$ is the antepenultimate surviving obstruction in the Frenkel-M\"uller staircase for $C_{L_{2,k}}$, with the last two being $\hat{A}_{0,2}^{(k)}=FM_{2k}$ and $A_{1,2}^{(k)}=FM_{2k+1}$.  So in all cases our staircases overlap with the remnants of the Frenkel-M\"uller staircase.  

As for the accumulation points $S_{n,k}$, we see in Proposition \ref{snkint} that each $S_{n,k}$ lies in the interval $\left(\frac{P_{2k+4}}{P_{2k+2}},\frac{P_{2k+2}}{P_{2k}}\right)$ (which for $k=0$ is to be interpreted as $(6,\infty)$). Since $\frac{P_{2k+2}}{P_{2k}}\to 3+2\sqrt{2}$ as $k\to\infty$, this shows that $S_{n,k}\to 3+2\sqrt{2}$ as $k\to\infty$, uniformly in $n$.  In the other limit as $n\to\infty$ with $k$ fixed, one has  $S_{n,k}\nearrow \frac{P_{2k+2}}{P_{2k}}$ and $L_{n,k}\nearrow \frac{H_{2k+1}+1}{H_{2k+1}-1}$  as $n\to\infty$.   In this limit all of the steps in our staircases have length tending to zero except for the step corresponding to $A_{0,n}^{(k)}$ (which as mentioned in the previous paragraph is equal to $FM_{2k-1}$ independently of $n$), and indeed a special case of Proposition \ref{p2k} shows that when $\beta=\frac{H_{2k+1}+1}{H_{2k+1}-1}$ the final step that remains in the Frenkel--M\"uller staircase for $C_{\beta}$ extends all the way to $\alpha=\frac{P_{2k+2}}{P_{2k}}$, at which point it can be seen to coincide with the volume bound.

The existence of an infinite staircase for $C_{L_{n,k}}$ appears to depend quite delicately on the specific values $L_{n,k}$. In particular it follows from Corollary \ref{moveL} that all but finitely many of the $\Gamma_{\cdot,\beta}(A_{i,n}^{(k)})$ cease to be relevant to $C_{\beta}$ when $\beta$ is arbitrarily close to but not equal to $L_{n,k}$.    For typical values of $\beta$ that are close to some of the $L_{n,k}$ we expect $C_{\beta}(\alpha)$ for $\alpha$ slightly larger than $3+2\sqrt{2}\approx 5.828$ to be given by the maximum of a small collection of the $\Gamma_{\alpha,\beta}(A_{i,n}^{(k)})$ for various values of $n$.  For example one can show (for instance using the program at \cite{U}) that for $\beta=5/4$ (which lies between $L_{6,1}$ and $L_{7,1}$), $C_{\beta}$ is given on $[3+2\sqrt{2},\frac{1000}{169}]$ by the obstruction coming from the exceptional class $FM_1=A_{0,n}^{(1)}=(2,1;\mathcal{W}(5,1))$, on $[\frac{1000}{169},\frac{5929}{1000}]$ by the obstruction coming from $A_{1,6}^{(1)}=(25,20;\mathcal{W}(77,13))$, and on $[\frac{5929}{1000},\frac{457}{77}]$ by the obstruction coming from $A_{1,7}^{(1)}=(29,23;\mathcal{W}(89,15))$, after which it is given on a somewhat longer interval by the obstruction coming from the non-quasi-perfect class $(2,2;2,1^{\times 5})$, which readers of \cite{FM} will recognize as the first class to appear after the infinite staircase for $C_1$.  

The $A_{i,n}^{(k)}$ and $\hat{A}_{i,n}^{(k)}$ are not the only perfect classes to contribute to some of the functions $C_{\beta}$ in the region following the Frenkel-M\"uller staircase; for instance $(15,10;\mathcal{W}(43,7))$ is the first class after $FM_1$ to contribute to $C_{3/2}$, and cannot be found among the $A_{i,n}^{(k)}$ or $\hat{A}_{i,n}^{(k)}$.  Preliminary computer experiments suggest that classes such as $(15,10;\mathcal{W}(43,7))$ may fit into different families that are structurally similar to the $A_{i,n}^{(k)}$, perhaps leading to infinite staircases for other irrational values of $\beta$ besides the $L_{n,k}$.  (To give a concrete family of examples, the author suspects that $C_{\beta}$ has an infinite staircase for $\beta=\frac{2n-1+2\sqrt{n^2-1}}{2n-2+\sqrt{n^2-1}}$ for all integers $n\geq 2$.  For $n=2$ this is equal to $\sqrt{3}=L_{2,0}$, but for $n\geq 3$ it is distinct from all of the $L_{n,k}$ since it lies strictly between $\frac{4}{3}=\sup_{k\geq 1,n\geq 2}L_{n,k}$ and $\sqrt{3}=\min_{n\geq 2}L_{n,0}$.) However analogous methods would not seem to be capable of producing infinite staircases for $C_{\beta}$ when $\beta$ is rational, consistently with the conjecture of Cristofaro-Gardiner, Holm, Mandini, and Pires (see \cite[p. 13]{P}) that would imply that the only rational $\beta$ for which any such staircase exists is the value $\beta=1$ considered in \cite{FM}.

\subsection{Organization of the paper}

The following section collects a variety of tools that are used at various places in our analysis.  It seems unavoidable that many of our proofs will involve extensive manipulations of Pell numbers $P_n$ and $H_n$, and some relevant facts about these appear in Section \ref{pellprelim}.  As will be familiar to experts, the subsets of $H^2(X_{N+1};\R)$ appearing in (\ref{poscrit}), namely $\overline{\mathcal{C}}_K(X_{N+1})$ and $\mathcal{E}_{N+1}$, are acted upon \emph{Cremona moves}. In Section \ref{cremintro} we recall this and set up relevant notation, after which we identify a very useful composition of Cremona moves, labeled $\Xi$ in Proposition \ref{xiaction}, and compute its action on various kinds of classes that appear in the rest of the paper.  Restricting attention to classes of the form $(a,b;\mathcal{W}(c,d))$ such as those that appear in Definition \ref{qpdef}, we then consider the question of when such a class is (quasi-)perfect.  Some simple algebra shows that the quasi-perfect classes of this form having $\gcd(c,d)=1$ correspond after a change of variables to solutions to a certain (generalized) Pell equation.  We can then exploit the construction from \cite{B} of infinite families of such solutions to define (Definition \ref{bmovedef}) the ``$k$th-order Brahmagupta move'' $C\mapsto C^{(k)}$ acting on classes $(a,b;\mathcal{W}(c,d))$.  By construction this move preserves the property of being quasi-perfect provided that $\gcd(c,d)=1$, and in Proposition \ref{pellcrem} we use the aforementioned composition of Cremona moves $\Xi$ to show that it also preserves the property of being perfect.  Section \ref{toolsect} concludes with a brief discussion of what we call the ``tiling criterion,'' which gives a sufficient criterion for a class to belong to $\bar{\mathcal{C}}(X_{N+1})$, expressed in terms of partial tilings of a large square by several rectangles.  The roots of this go back to \cite[Section 5]{T} and something similar is used in \cite[Section 3]{GU}, but we give a more systematic and straightforwardly-applicable formulation here.    

Section \ref{fmsect} contains the  proof of Theorem \ref{fmsup}.  First we rewrite more explicitly, for any given $\beta>1$, the supremum $\sup_n\Gamma_{\alpha,\beta}(FM_n)$, identifying it as a supremum over a finite set depending on $\beta$ and not on $\alpha$.  Using the monotonicity and sublinearity of $C_{\beta}$, the statement that the lower bound $C_{\beta}(\alpha)\geq \sup_n\Gamma_{\alpha,\beta}(FM_n)$ is sharp for all $\alpha\in [1,3+2\sqrt{2}]$ is easily seen to be equivalent to sharpness just for a finite subset of $\alpha$ (depending on $\beta>1$), namely the points where the ``steps'' in the finite staircase determining $\sup_n\Gamma_{\alpha,\beta}(FM_n)$ come together (as well as one point at the end of the staircase).  In each case this is equivalent to a certain class belonging to $\bar{\mathcal{C}}_K(X_{N+1})$ which we show (in Section \ref{fmsupproof}) to hold using the techniques of Section \ref{toolsect}.  A bit more specifically, 
our preferred approach to showing that a general class belongs to $\bar{\mathcal{C}}_K(X_{N+1})$ is to apply repeated Cremona moves to the class, often iteratively using $\Xi$, until it satisfies the tiling criterion.  Roughly speaking the move $\Xi$ transforms the problem for the $k$th class to the problem for the $(k-2)$nd class.
 
Section \ref{findsect} contains the proof of Theorem \ref{stairmain} and discusses some of the properties of our infinite staircases.  In Section \ref{critsect} we provide a general criterion for a sequence of perfect classes $(a_i,b_i;\mathcal{W}(c_i,d_i))$ to give rise to an infinite staircase.  We then construct our key collection of perfect classes $A_{i,n}^{(k)}$ in Section \ref{ainksect}, and show in Section \ref{versect} that, for fixed $n,k$, the sequence $\{A_{i,n}^{(k)}\}_{i=0}^{\infty}$ satisfies our general criterion.  This suffices to prove Theorem \ref{stairmain}, though it does not provide a complete description of the staircases.  We explain in Section \ref{undervolsect} that, at least for $k=0$, the lower bound $\sup_i\Gamma_{\alpha,L_{n,k}}(A_{i,n}^{(k)})$ for $C_{\beta}$ provided by the $A_{i,n}^{(k)}$ falls under the volume constraint at some values of $\alpha$ lying within the region occupied by the staircase, so the staircase is not completely described by the obstructions from $A_{i,n}^{(k)}$.  
Section \ref{lnksnk} carries out a few elementary calculations that help make sense of the values $L_{n,k}$ and $S_{n,k}$ from Theorem \ref{stairmain}, and then makes progress toward understanding how the function of two variables $(\alpha,\beta)\mapsto C_{\beta}(\alpha)$ behaves near $S_{n,k}$ and $L_{n,k}$ by finding two classes which are not among those contributing to the infinite staircase and whose obstructions at $(S_{n,k},L_{n,k})$ exactly match the volume.  We use this to give some indication of how  our infinite staircases disappear as $\beta$ is varied away from $L_{n,k}$ in Corollary \ref{moveL}.  Section \ref{ahatsect} introduces the classes $\hat{A}_{i,n}^{(k)}$ which appear to be necessary to completely describe our staircases, leading to the conjectural formula for $C_{L_{n,k}}(\alpha)$ in Conjecture \ref{fillconj}.  Finally, in Section \ref{connect} we work out the interface between our infinite staircases and the remnants of the Frenkel-M\"uller staircase that are determined in Section \ref{fmsect}.  With the exception of Section \ref{connect}, Sections \ref{fmsect} and \ref{findsect} are completely independent of each other.

\subsection*{Acknowledgements}  I am grateful to D. Cristofaro-Gardiner for useful comments and encouragement and for explaining his results about possible accumulation points of infinite staircases.  A crucial hint for the discovery of the classes $A_{i,n}^{(k)}$ was provided by OEIS entry A130282.  This work was partially supported by the NSF through the grant DMS-1509213.
\newpage
\section{Some tools}\label{toolsect}
\subsection{Preliminaries on Pell numbers}\label{pellprelim}
The Pell numbers $P_n$ and the half-companion Pell numbers $H_n$ appear frequently throughout the paper, and here we collect some facts concerning them. The sequences $\{P_n\}$ and $\{H_n\}$, by definition, obey the same recurrence relation: \[ P_{n+2}=2P_{n+1}+P_n\qquad\qquad H_{n+2}=2H_{n+1}+H_n \] with different initial conditions \[ P_0=0,\,\,P_1=1 \qquad\qquad H_0=1,\,\,H_1=1.\]  Denote by $\sigma=1+\sqrt{2}$ the ``silver ratio,'' so that $\sigma$ is the larger solution to the equation $x^2=2x+1$, the smaller solution being $-1/\sigma=1-\sqrt{2}$.  Note that $\sigma^2=3+2\sqrt{2}$, which is the quantity appearing in Theorem \ref{fmsup}. It is easy to check the following closed-form expressions for $P_n$ and $H_n$. \begin{equation}\label{closedform}
P_n=\frac{\sigma^n-(-\sigma)^{-n}}{2\sqrt{2}}\qquad\qquad H_n=\frac{\sigma^n+(-\sigma)^{-n}}{2}.\end{equation}

From these expressions it is not hard to check the identities, for $n,j\in \N$ with $j\leq n$, \begin{equation}\label{pp} P_{n+j}P_{n-j}=P_{n}^{2}+(-1)^{n+j+1}P_{j}^{2}, \end{equation}
\begin{equation} \label{ph} P_{n\pm j}H_{n\mp j}=P_nH_n\pm (-1)^{n+j}P_jH_j,\end{equation}
\begin{equation} \label{hh} H_nH_{n+j}=2P_nP_{n+j}+(-1)^nH_j. \end{equation}

Given the initial conditions $P_1=H_1=H_0=1$, we immediately see some useful special cases of these identities:
\begin{equation}\label{phloc} P_{n\pm 1}H_{n\mp 1}=P_nH_n\pm (-1)^{n+1},\end{equation} \begin{equation}\label{psquared} P_{n+1}P_{n-1}=P_{n}^{2}+(-1)^n,\end{equation}
\begin{equation} \label{htop}  H_{n}^{2}=2P_{n}^{2}+(-1)^n=2P_{n+1}P_{n-1}-(-1)^n,\end{equation}
\begin{equation} \label{consec} H_{n}H_{n+1}=2P_nP_{n+1}+(-1)^n.\end{equation}

Also, the fact that the $P_n$ and $H_n$ obey the same linear recurrence relation makes the following an easy consequence of their initial conditions: \begin{equation}\label{hpadd} H_{n}=P_{n}+P_{n-1}=P_{n+1}-P_{n},\qquad \qquad P_n=\frac{H_n+H_{n-1}}{2}=\frac{H_{n+1}-H_n}{2}.\end{equation}  Furthermore we have the identities
\begin{equation}\label{4hp} H_{n+2}+H_{n}=2H_{n+1}+2H_n=4P_{n+1}\end{equation} and similarly \begin{equation}\label{2hp} P_{n+2}+P_n=2P_{n+1}+2P_n=2H_{n+1}.\end{equation}
Moreover, 
\begin{equation}\label{4gapP} P_{n+2}+P_{n-2}=2P_{n+1}+P_n+P_{n-2}=5P_n+2P_{n-1}+P_{n-2}=6P_n \end{equation} and similarly \begin{equation}\label{4gapH} H_{n+2}+H_{n-2}=6H_n.\end{equation}

Although the conventional definition is that $P_n,H_n$ are defined only for $n\in \N$, we will occasionally (and without comment) use the convention that $P_{-1}=1,\,H_{-1}=-1,\,P_{-2}=-2$; evidently this is consistent with both the recurrence relations and the closed forms given above.

\begin{prop}\label{orderratios} For $k\geq 0$ the following inequalities hold: \[ \frac{P_{2k+1}}{P_{2k-1}}<\frac{H_{2k+2}}{H_{2k}}<\frac{P_{2k+3}}{P_{2k+1}}<\sigma^2<\frac{P_{2k+4}}{P_{2k+2}}<\frac{H_{2k+3}}{H_{2k+1}}<\frac{P_{2k+2}}{P_{2k}}.\]
\end{prop}
(Strictly speaking $\frac{P_{2k+2}}{P_{2k}}$ is not defined if $k=0$ since $P_0=0$, but in this case we interpret $\frac{P_{2k+2}}{P_{2k}}$ as $\infty$. A similar remark applies to Proposition \ref{28649})
\begin{proof}
We have \begin{align*} P_{n+1}H_n-P_{n-1}H_{n+2}&=(2P_n+P_{n-1})H_n-P_{n-1}(2H_{n+1}+H_n)
\\ &= 2(P_nH_n-P_{n-1}H_{n+1})=-2(-1)^n \end{align*} where the last equality uses (\ref{phloc}), and this immediately implies  that \[ \frac{P_{2k+1}}{P_{2k-1}}<\frac{H_{2k+2}}{H_{2k}}\qquad\mbox{and}\qquad \frac{P_{2k+2}}{P_{2k}}>\frac{H_{2k+3}}{H_{2k+1}}.\]  Similarly, using the other case of (\ref{phloc}) we have \[ P_nH_{n+1}-P_{n+2}H_{n-1}=P_n(2H_n+H_{n-1})-(2P_{n+1}+P_n)H_{n-1}=2(P_nH_n-P_{n+1}H_{n-1})=2(-1)^n,\]
so that \[ \frac{H_{2k+3}}{H_{2k+1}}>\frac{P_{2k+4}}{P_{2k+2}}\qquad \mbox{and}\qquad \frac{H_{2k+2}}{H_{2k}}<\frac{P_{2k+3}}{P_{2k+1}}.\]
So it remains only to show that $\frac{P_{2k+1}}{P_{2k-1}}<\sigma^2<\frac{P_{2k+2}}{P_{2k}}$ for all $k$.  By (\ref{closedform}), we see that \[ \frac{P_{2k+1}}{P_{2k-1}}=\frac{\sigma^{2k+1}+\sigma^{-2k-1}}{\sigma^{2k-1}+\sigma^{-2k+1}}=\sigma^2\left(\frac{1+\sigma^{-4k-2}}{1+\sigma^{-4k+2}   } \right)<\sigma^2\] since $\sigma>1$.  Similarly \[ \frac{P_{2k+2}}{P_{2k}}=\sigma^2\left(\frac{1-\sigma^{-4k-4}}{1-\sigma^{-4k}}\right)>\sigma^2.\]
\end{proof}

It so happens that the sequence $\left\{\frac{2(P_{2k+2}^{2}-1)}{H_{2k+1}^{2}}\right\}_{k=0}^{\infty}=\{6,\frac{286}{49},\frac{9798}{1681},\ldots\}$ will play a role in the proof of Theorem \ref{fmsup} (specifically in Proposition \ref{lastplat}), and the following estimate will be relevant:
\begin{prop}\label{28649}
For $k\geq 0$ we have \[ \sigma^2<\frac{2(P_{2k+2}^{2}-1)}{H_{2k+1}^{2}}<\frac{P_{2k+2}}{P_{2k}}.\]  
\end{prop}

\begin{proof}
First notice that (\ref{htop}) gives $H_{2k+1}^{2}=2P_{2k}P_{2k+2}+1$, and so \[ P_{2k+2}H_{2k+1}^{2}=2P_{2k+2}^{2}P_{2k}+P_{2k+2}>2P_{2k}(P_{2k+2}^{2}-1),\] from which the second inequality is immediate.  As for the first inequality, based on (\ref{closedform}) we have \begin{align*}
\frac{2(P_{2k+2}^{2}-1)}{H_{2k+1}^{2}}&=\frac{\sigma^{4k+4}-10+\sigma^{-4k-4}}{\sigma^{4k+2}-2+\sigma^{-4k-2}}
\\ &=\sigma^2\left(\frac{1-10\sigma^{-4k-4}+\sigma^{-8k-8}}{1-2\sigma^{-4k-2}+\sigma^{-8k-4}}\right).
\end{align*} So the desired inequality is equivalent to the statement that $-10\sigma^{-4k-4}+\sigma^{-8k-8}>-2\sigma^{-4k-2}+\sigma^{-8k-4}$, \emph{i.e.} (multiplying both sides by $\sigma^{4k+4}$ and rearranging) \begin{equation}\label{286need} 2\sigma^2-10>\sigma^{-4k}(1-\sigma^{-4}).\end{equation}  Of course since $\sigma>1$, (\ref{286need}) holds for all $k\geq 0$ if and only if it holds for $k=0$, \emph{i.e.} if and only if $2\sigma^2+\sigma^{-4}>11$. But this is indeed true: we have $2\sigma^2=2(3+2\sqrt{2})>11$ since $\sqrt{2}>\frac{5}{4}$.  This proves (\ref{286need}) and hence the proposition.
\end{proof}

We now discuss some connections between weight sequences and Pell numbers.  As in the introduction, for any pair of nonnegative, rationally dependent real numbers $x,y$, the \textbf{weight sequence} $\mathcal{W}(x,y)$ associated to the ellipsoid $E(x,y)$ is determined recursively by setting $\mathcal{W}(x,0)=\mathcal{W}(0,x)$ equal to the empty sequence and, if $x\leq y$, setting $\mathcal{W}(x,y)$ and $\mathcal{W}(y,x)$ both equal to (abusing notation slightly) $\left(x,\mathcal{W}(x,y-x)\right)$, \emph{i.e.} to the sequence that results from prepending $x$ to the sequence $\mathcal{W}(x,y-x)$.   

More geometrically, the weight sequence $\mathcal{W}(x,y)=(w_1,\ldots,w_k)$ is obtained by beginning with a $x$-by-$y$ rectangle and inductively removing as large a square as possible (of side length $w_i$ at the $i$th stage), leaving a smaller rectangle, until the final stage when the rectangle that remains is a square of side length $w_k$.  Thus the statement that $\mathcal{W}(x,y)=(w_1,\ldots,w_k)$ in particular implies that an $a$-by-$b$ rectangle can be tiled by a set of squares of side lengths $w_1,\ldots,w_k$; 

First we compute a certain specific family of weight sequences.

\begin{prop}\label{pellweight}
For any positive real number $x$ and any $m\in \N$ we have \[ \mathcal{W}(2P_{2m+1}x,P_{2m}x)=\left((P_{2m}x)^{\times 4},2P_{2m-1}x,\ldots,(P_{2j}x)^{\times 4},2P_{2j-1}x,\ldots,(P_2x)^{\times 4},2P_1x\right).\]
\end{prop}

(Of course, since $P_2=2P_1=2$ we could equally well write the last five terms in the sequence as $(2x)^{\times 5}$. See Figure \ref{pellpic} for the corresponding tiling in the case that $m=3$.)

\begin{proof}
This is a straightforward induction on $m$: if $m=0$ then since $P_{2m}=0$ both sides of the equation are equal to the empty sequence, while for $m>0$ the fact that \[ 0<2P_{2m+1}-4P_{2m}=2P_{2m-1}=P_{2m}-P_{2m-2}\leq P_{2m} \] implies that \begin{align*} \mathcal{W}(2P_{2m+1}x,P_{2m}x)&=\left(\left((P_{2m}x)^{\times 4}\right), \mathcal{W}(2P_{2m-1}x,P_{2m}x)\right)\\&=\left(\left((P_{2m}x)^{\times 4},2P_{2m-1}x\right), \mathcal{W}(2P_{2m-1}x,P_{2m-2}x)\right).\end{align*}  Thus the validity of the result for $m-1$ implies it for $m$. \end{proof}

\begin{center}
\begin{figure}
\includegraphics[height=1.2 in]{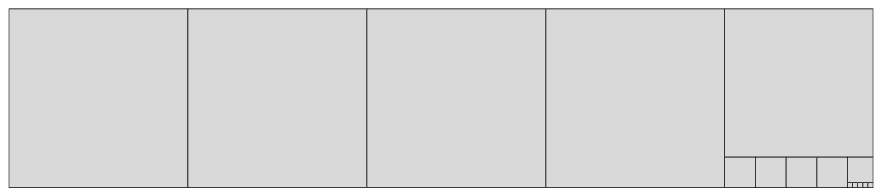}
\caption{The square tiling of a $338$-by-$70$ rectangle corresponding to the fact that \[\mathcal{W}(2P_{7},P_6)=(P_{6}^{\times 4},2P_5,P_{4}^{\times 4},2P_3,P_{2}^{\times 4},2P_1).\]}
\label{pellpic}
\end{figure}
\end{center}

The following computation gives the first part of the weight expansion $w(\alpha)=\mathcal{W}(1,\alpha)$ of an arbitrary rational number $\alpha\geq 1$, with more information when $\alpha$ is close to $\sigma^2$.

\begin{prop}\label{genexp}
Assume that $\alpha\in \left[\frac{P_{2k+1}}{P_{2k-1}},\frac{P_{2k+2}}{P_{2k}}\right]\cap\Q$ where $k\geq 0$.  Then \begin{align*} w(\alpha)&=\left(1,\left(\frac{P_2}{2}-\frac{P_0}{2}\alpha\right)^{\times 4},P_1\alpha-P_3,\ldots,\left(\frac{P_{2k}}{2}-\frac{P_{2k-2}}{2}\alpha\right)^{\times 4},P_{2k-1}\alpha-P_{2k+1},\right. \\ & \qquad \qquad \qquad \left.\mathcal{W}\left(\frac{P_{2k+2}}{2}-\frac{P_{2k}}{2}\alpha,P_{2k-1}\alpha-P_{2k+1}\right)\right).\end{align*}
\end{prop} (If $k=0$, in which case the condition $\alpha\in \left[\frac{P_{2k+1}}{P_{2k-1}},\frac{P_{2k+2}}{P_{2k}}\right]$  just says that $\alpha\in [1,\infty)$, then  the  sequence $\left(\frac{P_2}{2}-\frac{P_0}{2}\alpha\right)^{\times 4},\ldots,P_{2k-1}\alpha-P_{2k+1}$ should be interpreted as empty, so this just simplifies to $w(\alpha)=\left(1,\mathcal{W}(1,\alpha-1)\right)$.)

\begin{proof} We proceed by induction; for $k=0$ the statement is trivial.  Let $\alpha\in \left[\frac{P_{2k+1}}{P_{2k-1}},\frac{P_{2k+2}}{P_{2k}}\right]$ where $k\geq 1$ and assume the statement proven for all $j<k$.  Note that Proposition \ref{orderratios} shows that \[ \left[\frac{P_{2k+1}}{P_{2k-1}},\frac{P_{2k+2}}{P_{2k}}\right]\subset \left[\frac{P_{2j+1}}{P_{2j-1}},\frac{P_{2j+2}}{P_{2j}}\right] \mbox{ for }j<k,\] so the inductive hypothesis applies to our particular $\alpha$.  The special case $j=k-1$ of the inductive hypothesis leads us to consider $\mathcal{W}\left(\frac{P_{2k}}{2}-\frac{P_{2k-2}}{2}\alpha,P_{2k-3}\alpha-P_{2k-1}\right)$.  We simply observe that \[  \left(P_{2k-3}\alpha-P_{2k-1}\right)-4\left(\frac{P_{2k}}{2}-\frac{P_{2k-2}}{2}\alpha\right)=P_{2k-1}\alpha-P_{2k+1}\geq 0 \] since we assume that $\alpha\geq \frac{P_{2k+1}}{P_{2k-1}}$, and then that \[ \left(\frac{P_{2k}}{2}-\frac{P_{2k-2}}{2}\alpha\right)-(P_{2k-1}\alpha-P_{2k+1})=\frac{P_{2k+2}}{2}-\frac{P_{2k}}{2}\alpha\geq 0\] since we assume that $\alpha\leq \frac{P_{2k+2}}{P_{2k}}$.  Thus 
\begin{align*} \mathcal{W}\left(\frac{P_{2k}}{2}-\frac{P_{2k-2}}{2}\alpha,P_{2k-3}\alpha-P_{2k-1}\right)=&\left(\left(\frac{P_{2k}}{2}-\frac{P_{2k-2}}{2}\alpha\right)^{\times 4},P_{2k-1}\alpha-P_{2k+1}\right)\\ &\sqcup\mathcal{W}\left(\frac{P_{2k+2}}{2}-\frac{P_{2k}}{2}\alpha,P_{2k-1}\alpha-P_{2k+1}\right).\end{align*}  The result then follows immediately by induction.
\end{proof}

\subsection{Cremona moves}\label{cremintro}

As in the introduction let $X_{N+1}$ denote the $(N+1)$-point blowup of $\mathbb{C}P^2$, with $L,E_0,\ldots,E_N\in H^2(X_{N+1},\Z)$ the Poincar\'e duals of a complex projective line  and of the $N+1$ exceptional divisors of the blowups, respectively.  If $x,y,z\in \{0,\ldots, N\}$ then $X_{N+1}$ contains a smoothly embedded sphere of self-intersection $-2$ that is Poincar\'e dual to the class $L-E_x-E_y-E_z$, and the \textbf{Cremona move} $\frak{c}_{xyz}\co H^2(X_{N+1};\R)\to H^2(X_{N+1};\R)$ is defined to be the cohomological action of the generalized Dehn twist along this sphere.  Likewise if $x,y\in\{0,\ldots,N\}$ then $X_{N+1}$ contains a smoothly embedded sphere of self-intersection $-2$ Poincar\'e dual to $E_x-E_y$, and we let $\mathfrak{c}_{xy}$ denote the action on $H^2$ of the generalized Dehn twist along this sphere.   
In terms of the basis $\{L,E_0,\ldots,E_N\}$ we have \[ \frak{c}_{xyz}\left(dL-\sum_i a_iE_i\right)=(d+\delta_{xyz})L-\sum_{i\in\{x,y,z\}}(a_i+\delta_{xyz})E_i-\sum_{i\notin\{x,y,z\}}a_iE_i \] where \begin{equation}\label{deltadef} \delta_{xyz}=d-a_x-a_y-a_z\end{equation} and \[ \frak{c}_{xy}\left(dL-\sum_ia_iE_i\right)=dL-a_yE_x-a_xE_y-\sum_{i\notin\{x,y\}}a_iE_i.\]  (So in terms of the coordinates $\langle d;a_0,\ldots,a_N\rangle$ from Section \ref{init}, $\frak{c}_{xyz}$ adds $\delta_{xyz}$ to the coordinates $d,a_x,a_y,a_z$ and $\frak{c}_{xy}$ simply swaps $a_x$ with $a_y$.)  Note that Cremona moves preserve the standard first Chern class $c_1(TX_N)=3L-\sum_i E_i$.  We say that two classes $A,B\in H^2(X_{N+1};\R)$ are \textbf{Cremona equivalent} if there is a sequence of Cremona moves mapping $A$ to $B$.  The operations $\frak{c}_{xyz},\frak{c}_{xy}$ obviously give rise to corresponding operations on the direct limit $\mathcal{H}^2=\varinjlim H^2(X_{N};\R)$, which we denote by the same symbols.

A crucial fact for our purposes will be that Cremona moves $\frak{c}_{xyz}$ and $\frak{c}_{xy}$ preserve both the closure of the symplectic cone $\bar{\mathcal{C}}_K(X_{N+1})$ and the set of exceptional classes $\mathcal{E}_{N+1}$.  Indeed, as shown in \cite[Proposition 1.2.12(iii)]{MS}, one has $E\in \mathcal{E}_{N+1}$ if and only if $E$ can be mapped to $E_0$ by a sequence of Cremona moves; since Cremona moves are induced by orientation-preserving diffeomorphisms this implies that they likewise preserve $\bar{\mathcal{C}}_K(X_{N+1})$ by (\ref{poscrit}).  Thus to verify condition (ii) in Proposition \ref{background} holds (and thus to show that $\lambda\geq C_{\beta}(\alpha)$) it suffices to find a sequence of Cremona moves which sends the class $\langle \lambda(\beta+1);\lambda \beta,\lambda,x_2,\ldots,x_N\rangle$ to a class that can be directly verified to lie in $\bar{\mathcal{C}}_K(X_{N+1})$.  Likewise to show that classes such as the $A_{i,n}^{(k)}$ that we use to prove Theorem \ref{stairmain} belong to $\mathcal{E}$ it suffices to show that they are Cremona equivalent to $E_0=\langle 0;-1,0\rangle=(1,0;1)$.

 There is a particular composition of Cremona moves whose repeated application underlies both the proof of Theorem \ref{fmsup} and the construction of many of the classes involved in our infinite staircases. Specifically, let \[ \Xi=\frak{c}_{36}\circ \frak{c}_{456}\circ \frak{c}_{236}\circ \frak{c}_{012}\circ \frak{c}_{345}\in Aut\left( H^2(X_7;\R) \right).\]  

\begin{prop}\label{xiaction} Given any $Z,A,B,C,\epsilon\in \R$, we have \[ \Xi\left(\langle Z;A+\epsilon,A-\epsilon,B^{\times 4},C\rangle\right)=\langle Z';A'+\epsilon,A'-\epsilon,B'^{\times 4},C'\rangle \] where $A',B',C',Z'$ are computed as follows.  Let \[\zeta=Z-2B.\]  Then \begin{align*}
A' &= 2\zeta-A, \\ C' &= 2\zeta-C, \\ B' &= C'+Z-2A-B, \\ Z' &= 2B'+\zeta.
\end{align*}
\end{prop}

\begin{proof}
This is a straightforward computation which we leave to the reader.
\end{proof}

Repeated application of the following proposition will be helpful in the proof of Theorem \ref{fmsup}.

\begin{prop}\label{indxi}
For any $j\in\N$, $\gamma,\alpha,\beta\in\R$ we have \begin{align*} &\Xi\left(\left\langle P_{2j-1}(\gamma(\beta+1)-1)-P_{2j-2}\alpha;\frac{H_{2j-2}}{2}\gamma(\beta+1)-P_{2j-1}+\frac{\gamma(\beta-1)}{2},  \right.\right.\\ &\qquad\qquad \qquad  \frac{H_{2j-2}}{2}\gamma(\beta+1)-P_{2j-1}-\frac{\gamma(\beta-1)}{2},  \left.\left.\left(\frac{P_{2j}}{2}-\frac{P_{2j-2}}{2}\alpha\right)^{\times 4},P_{2j-1}\alpha-P_{2j+1} \right\rangle \right) 
\\ &
 =\left\langle P_{2j+1}(\gamma(\beta+1)-1)-P_{2j}\alpha;\frac{H_{2j}}{2}\gamma(\beta+1)-P_{2j+1}+\frac{\gamma(\beta-1)}{2},\right. 
\\  &\qquad\qquad \qquad \left. \frac{H_{2j}}{2}\gamma(\beta+1)-P_{2j+1}-\frac{\gamma(\beta-1)}{2}, \left(\frac{P_{2j}}{2}(2\gamma(\beta+1)-\alpha-1)\right)^{\times 4},\right.\\ & \qquad\qquad \qquad \qquad P_{2j-1}(2\gamma(\beta+1)-\alpha-1)   \Bigg\rangle. \end{align*}
\end{prop}

\begin{proof}
We follow the notation of Proposition \ref{xiaction}, so $Z=P_{2j-1}(\gamma(\beta+1)-1)-P_{2j-2}\alpha$, $A=\frac{H_{2j-2}}{2}\gamma(\beta+1)-P_{2j-1}$, $B=\frac{P_{2j}}{2}-\frac{P_{2j-2}}{2}\alpha$, and $C=P_{2j-1}\alpha-P_{2j+1}$.    We then find \begin{align*} \zeta&=\left(P_{2j-1}(\gamma(\beta+1)-1)-P_{2j-2}\alpha\right)-2\left(\frac{P_{2j}}{2}-\frac{P_{2j-2}}{2}\alpha\right)
\\ &=  P_{2j-1}\gamma(\beta+1)-P_{2j-1}-P_{2j}=P_{2j-1}\gamma(\beta+1)-H_{2j} \end{align*}
and \begin{align*} Z-2A-B&=\gamma(\beta+1)(P_{2j-1}-H_{2j-2})-\frac{P_{2j-2}}{2}\alpha+\left(-P_{2j-1}+2P_{2j-1}-\frac{P_{2j}}{2}\right)
\\ &= P_{2j-2}\left(\gamma(\beta+1)-\frac{1}{2}\alpha-\frac{1}{2}\right).
\end{align*} Thus \[ C'=2\zeta-C=P_{2j-1}(2\gamma(\beta+1)-\alpha)-(2H_{2j}-P_{2j+1})=P_{2j-1}(2\gamma(\beta+1)-\alpha-1) \] where we have used that $2H_{2j}-P_{2j+1}=P_{2j-1}$ by (\ref{2hp}). Also, \[ B'=C'+(Z-2A-B)=\left(P_{2j-1}+\frac{P_{2j-2}}{2}\right)(2\gamma(\beta+1)-\alpha-1)=\frac{P_{2j}}{2}(2\gamma(\beta+1)-\alpha-1),\] and then \begin{align*} Z'&=2B'+\zeta=(P_{2j-1}+2P_{2j})\gamma(\beta+1)-(H_{2j}+P_{2j})-P_{2j}\alpha
\\ &=P_{2j+1}(\gamma(\beta+1)-1)-P_{2j}\alpha.\end{align*}  Finally \begin{align*}A'&=2\zeta-A=\left(2P_{2j-1}-\frac{H_{2j-2}}{2}\right)\gamma(\beta+1)-(2H_{2j}-P_{2j-1})\\ &=\frac{H_{2j}}{2}\gamma(\beta+1)-P_{2j+1} \end{align*} since (\ref{4hp}) shows that $\frac{H_{2j}}{2}+\frac{H_{2j-2}}{2}=2P_{2j-1}$ and (\ref{2hp}) shows that $P_{2j+1}+P_{2j-1}=2H_{2j}$. In view of Proposition \ref{xiaction}, this completes the proof. 
\end{proof}

The definition of $\Xi$ given above presents it as an automorphism of $H_2(X_7;\R)$; we will use variants $\Xi^{(n)}$ of $\Xi$ which are automorphisms of $H_2(X_N;\R)$ where $N\geq n+5\geq 7$.  Specifically, we define \begin{equation}\label{xindef} \Xi^{(n)}=\frak{c}_{n+1,n+4}\circ \frak{c}_{n+2,n+3,n+4}\circ \frak{c}_{n,n+1,n+4}\circ \frak{c}_{01n}\circ \frak{c}_{n+1,n+2,n+3}.\end{equation}  Equivalently, \begin{align*} \Xi^{(n)}&\left(\left\langle r;s_0,s_1,s_2,\ldots,s_{n-1},s_n,s_{n+1}s_{n+2},s_{n+3},s_{n+4},s_{n+5},\ldots,s_{N-1}\right\rangle\right)
\\ &\quad = \left\langle r';s'_0,s'_1,s_2,\ldots,s_{n-1},s'_n,s'_{n+1}s'_{n+2},s'_{n+3},s'_{n+4},s_{n+5},\ldots,s_{N-1}\right\rangle \end{align*} where $r',s'_0,s'_1,s'_n,\ldots,s'_{n+4}$ are defined by the property that $\Xi(\langle r;s_0,s_1,s_{n},\ldots,s_{n+4}\rangle)=\Xi(\langle r';s'_0,s'_1,s'_{n},\ldots,s'_{n+4}\rangle)$.  (In practice we will have $s_n=s_{n+1}=s_{n+2}=s_{n+3}$ so that we can apply Proposition \ref{xiaction}.)

We will now use the $\Xi^{(n)}$ together with Propositions \ref{genexp} and \ref{indxi} to modify via Cremona moves the classes that are relevant to the embedding problems that arise in the proof of Theorem \ref{fmsup}. Recall that $E(1,\alpha)^{\circ}$ symplectically embeds into $\gamma P(1,\beta)$ if and only if the $\left(\gamma b,\gamma;\mathcal{W}(1,a)\right)\in \bar{\mathcal{C}}_K(X_{N+1})$ where $N$ is the length of the weight sequence $w(\alpha)$.

\begin{prop} \label{bigreduce} Assume that $\gamma,\beta\geq 1$ and that $\alpha\in \left[\frac{P_{2k+1}}{P_{2k-1}},\frac{P_{2k+2}}{P_{2k}}\right]\cap \Q$ where $k\geq 1$.  If $2\gamma(\beta+1)-a-1<0$ then $\left(\gamma b,\gamma;\mathcal{W}(1,a)\right)\notin \bar{\mathcal{C}}_K(X_{N+1})$ where $N$ is the length of $w(\alpha)$.  Otherwise, $\left(\gamma \beta,\gamma;\mathcal{W}(1,\alpha)\right)$ is Cremona equivalent to the class \begin{align*} \Sigma_{\alpha,\beta,\gamma}^{k}&=\left\langle P_{2k+1}\left(\gamma(\beta+1)-1\right)-P_{2k}\alpha;\frac{H_{2k}}{2}\gamma(\beta+1)-P_{2k+1}+\gamma\left(\frac{\beta-1}{2}\right),\right.\\ &\qquad  \frac{H_{2k}}{2}\gamma(\beta+1)-P_{2k+1}-\gamma\left(\frac{\beta-1}{2}\right),\mathcal{W}\left(\frac{P_{2k+2}}{2}-\frac{P_{2k}}{2}\alpha,P_{2k-1}\alpha-P_{2k+1}\right)    ,\\ & \qquad \left.\mathcal{W}\left(\frac{P_{2k}}{2}\left(2\gamma(\beta+1)-\alpha-1\right) , P_{2k+1}\left(2\gamma(\beta+1)-\alpha-1\right)\right)\right\rangle.\end{align*}
\end{prop}

\begin{proof}
Combining (\ref{convert1}) with Proposition \ref{genexp}, we see that our class $\left(\gamma\beta,\gamma;\mathcal{W}(1,\alpha)\right)$ is equal to \begin{align*}
&\left\langle \gamma(\beta+1)-1;\gamma \beta-1,\gamma-1,\left(\frac{P_2}{2}-P_0\alpha\right)^{\times 4},P_1\alpha-P_3,\ldots \right.
\\ & \qquad \left. \left(\frac{P_{2k}}{2}-\frac{P_{2k-2}}{2}\alpha\right)^{\times 4},P_{2k-1}\alpha-P_{2k+1},\mathcal{W}\left(\frac{P_{2k+2}}{2}-\frac{P_{2k}}{2}\alpha,P_{2k-1}\alpha-P_{2k+1}\right)\right\rangle.
\end{align*}
With  a view toward Proposition \ref{indxi}, note that the first three terms above can be rewritten as \[ \gamma(\beta+1)-1=P_1(\gamma(\beta+1)-1)-P_0\alpha,\] \[ \gamma \beta-1 = \frac{H_0}{2}\gamma(\beta+1)-P_1+\frac{\gamma(\beta-1)}{2},\] \[ \gamma-1 = \frac{H_0}{2}\gamma(\beta+1)-P_1-\frac{\gamma(\beta-1)}{2}.\]
So we can apply Proposition \ref{indxi} successively with $j=1,\ldots,k$ to find that the image of $\left(\gamma\beta,\gamma;\mathcal{W}(1,\alpha)\right)$ under the composition of Cremona moves $\Xi^{(5k-3)}\circ\cdots\circ\Xi^{(7)}\circ\Xi^{(2)}$ is equal to 
\begin{align*} &\left\langle P_{2k+1}\left(\gamma(\beta+1)-1\right)-P_{2k}\alpha;\frac{H_{2k}}{2}\gamma(\beta+1)-P_{2k+1}+\gamma\left(\frac{\beta+1}{2}\right),\right.\\ &\qquad  \frac{H_{2k}}{2}\gamma(\beta+1)-P_{2k+1}-\gamma\left(\frac{\beta-1}{2}\right),\left(\frac{P_2}{2}(2\gamma(\beta+1)-\alpha-1)\right)^{\times 4},\\ &\qquad  P_{1}\left(2\gamma(\beta+1)-\alpha-1\right),\ldots,\left(\frac{P_{2k}}{2}(2\gamma(\beta+1)-\alpha-1)\right)^{\times 4},P_{2k-1}(2\gamma(\beta+1)-\alpha-1),
\\ & \qquad \left. \mathcal{W}\left(\frac{P_{2k+2}}{2}-\frac{P_{2k}}{2}\alpha,P_{2k-1}\alpha-P_{2k+1}\right)\right\rangle. 
\end{align*}

If $2\gamma(\beta+1)-\alpha-1<0$ then above expression has some negative entries and so our class pairs negatively with some of the $E'_i$ and thus cannot belong to $\bar{\mathcal{C}}_K(X_{N+1})$.  Otherwise, we can use Proposition \ref{pellweight} to group the entries beginning with $\frac{P_2}{2}(2\gamma(\beta+1)-\alpha-1)$ and ending with $P_{2k-1}(2\gamma(\beta+1)-\alpha-1)$ together as $\mathcal{W}\left(\frac{P_{2k}}{2}(2\gamma(\beta+1)-\alpha-1),P_{2k+1}(2\gamma(\beta+1)-\alpha-1)\right)$ and so (modulo reordering, which can be carried out by Cremona moves $\frak{c}_{xy}$) the above class is precisely the class $\Sigma_{a,b,\gamma}^{k}$ given in the proposition.
\end{proof}

\begin{remark} The values of $\alpha$ such that there exist $k$ for which Proposition \ref{bigreduce} is applicable to $\alpha$  are precisely those $\alpha$ in the interval $\left[\frac{P_3}{P_1},\frac{P_4}{P_2}\right]=[5,6]$.  For any such $\alpha$, we have \[ w(\alpha)=\left(1^{\times 5},\alpha-5,\mathcal{W}(\alpha-5,6-\alpha)\right).\]  Consider the class $E=(2,2;2,1^{\times 5})$, which lies in $\mathcal{E}$.  We find that, for $\alpha\in [5,6]$, \[ \mu_{\alpha,\beta}(E)=\frac{(2,1^{\times 5})\cdot (1^{\times 5},\alpha-5)}{2+2\beta}=\frac{\alpha+1}{2(\beta+1)}.\]  Thus the condition that $2\gamma(\beta+1)-\alpha-1\geq  0$ in Proposition \ref{bigreduce} is equivalent to the condition that the class $E$ does not obstruct the embedding $E(1,\alpha)^{\circ}\hookrightarrow \gamma P(1,\beta)$.  

The class $E$ was identified in \cite{FM} as giving a sharp obstruction to this embedding when $\beta=1$ and $\alpha\in [\sigma^2,6]$.  (For $\beta=1$ and $1\leq \alpha<\sigma^2$, on the other hand, $\mu_{\alpha,\beta}(E)$ is less than the volume bound.)  Results such as Theorem \ref{stairmain} show that the situation is more complicated for $\alpha\in (\sigma^2,6]$ and $\beta$ arbitrarily close but not equal to $1$.
\end{remark}

\begin{remark}\label{k0} Since $P_0=0$ and $\frac{P_2}{2}=P_1=H_0=P_{-1}=1$, the $k=0$ version of the class $\Sigma_{\alpha,\beta,\gamma}^{k}$ would degenerate to 
$\langle \gamma(\beta+1)-1;\gamma \beta-1,\gamma-1,\mathcal{W}(1,\alpha-1)\rangle$, which by (\ref{convert1}) is equal to $(\gamma\beta,\gamma;\mathcal{W}(1,\alpha))$. So the appropriate---and trivially true---variant of Proposition \ref{bigreduce} for $k=0$ (which would allow $\alpha$ to be an arbitrary value in $[1,\infty)$) is simply that $(\gamma \beta,\gamma;\mathcal{W}(1,\alpha))\in \bar{\mathcal{C}}_K(X_{N+1})$ if and only if $\Sigma_{\alpha,\beta,\gamma}^{0}\in \bar{\mathcal{C}}_K(X_{N+1})$ (with no condition on $2\gamma(\beta+1)-\alpha-1$).
\end{remark}

\subsubsection{The Brahmagupta move on perfect classes} \label{bmove}  We now use the move $\Xi$ from Proposition \ref{xiaction} to construct an action on classes of the form $\left(a,b;\mathcal{W}(c,d)\right)$ that will be important in our proof of the existence of some of the infinite staircases from Theorem \ref{stairmain}. (Specifically, the obstructions producing the infinite staircase for $C_{L_{n,k}}$ will be obtained from those producing the infinite staircase for $C_{L_{n,0}}$ by the $k$th-order Brahmagupta move, defined below.)  

To motivate this, let us consider the question of whether a class $C=\left(a,b;\mathcal{W}(c,d)\right)$ where $a,b,c,d\in\N$ belongs to the sets $\tilde{\mathcal{E}}$ or $\mathcal{E}$ from the introduction.  Assume for simplicity that $\gcd(c,d)=1$ and that $a\geq b$ and $c\geq d$.  Since the entries of $\mathcal{W}(c,d)$ are all nonnegative, we will have $C\in\tilde{\mathcal{E}}$ if and only if $C$ has Chern number $1$ and self-intersection $-1$.  Writing $\mathcal{W}(c,d)=(m_1,\ldots,m_N)$, the Chern number of $C$ is $2(a+b)-\sum_{i}m_i=2(a+b)-(c+d-1)$ where we have used \cite[Lemma 1.2.6(iii)]{MS} and the assumption that $c$ and $d$ are relatively prime.  Thus $C$ has the correct Chern number for membership in $\tilde{\mathcal{E}}$ precisely if $2(a+b)=c+d$.  This holds if and only if we can express $C$ in the form \begin{equation}\label{xdeform} C=\left(\frac{x+\ep}{2},\frac{x-\ep}{2};\mathcal{W}(x+\delta,x-\delta)\right) \end{equation} where $x,\delta,\ep\in \N$ with $\delta\leq x$.  Since we will have $\sum m_{i}^{2}=(x+\delta)(x-\delta)$ (as is obvious from the interpretation of $\mathcal{W}(x+\delta,x-\delta)$ in terms of 
tiling a rectangle by squares, as in \cite[Lemma 1.2.6(ii)]{MS}), the self-intersection number of $C$ is $2\left(\frac{x+\ep}{2}\right)\left(\frac{x-\ep}{2}\right)-(x+\delta)(x-\delta)=-\frac{1}{2}\left(x^2-2\delta^2+\ep^2\right)$.  Thus a class of the form (\ref{xdeform}) belongs to $\tilde{\mathcal{E}}$ if and only if the triple $(x,\delta,\ep)$ obeys \begin{equation}\label{pelleqn} x^2-2\delta^2=2-\ep^2.\end{equation}

Now if we temporarily regard $\ep$ as fixed and $x$ and $\delta$ as variables (\ref{pelleqn}) is a (generalized) \emph{Pell equation} $x^2-D\delta^2=N$ with $D=2$ and $N=2-\ep^2$.  A basic feature of such equations, observed in \cite[XVIII 64-65, p. 246]{B}, is that their integer solutions come in infinite families.  Indeed the equation asks for $x+\delta\sqrt{D}$ to be an element of norm $N$ in $\Z[\sqrt{D}]$, and the norm is multiplicative, so if $u+v\sqrt{D}\in \Z[\sqrt{D}]$ has norm one then $(u+v\sqrt{D})(x+\delta\sqrt{D})$ will have norm $N$, \emph{i.e.} $(ux+Dv\delta,vx+u\delta)$ will again be a solution to the equation.  In our case where $D=2$, (\ref{htop}) shows that, for any $k\geq 0$, $H_{2k}+P_{2k}\sqrt{2}$ has norm one in $\Z[\sqrt{2}]$.  Thus, given $\ep\in \Z$, if $(x,\delta)$ is one solution to (\ref{pelleqn}) then, for all $k$, $(H_{2k}x+2P_{2k}\delta,P_{2k}x+H_{2k}\delta)$ is also a solution.

Accordingly we make the following definition:

\begin{dfn}\label{bmovedef} For $k\in\N$, the \textbf{$k$th-order Brahmagupta move} is the operation which sends a class $C\in\mathcal{H}^2$ having the form (\ref{xdeform}) where $x,\delta,\ep\in \N$ and $\delta\leq x$ to the class \[ C^{(k)}=\left(\frac{x_k+\ep}{2},\frac{x_k-\ep}{2};\mathcal{W}(x_k+\delta_k,x_k-\delta_k)\right) \] where \[ x_k=H_{2k}x+2P_{2k}\delta,\qquad \delta_k=P_{2k}x+H_{2k}\delta.\]
\end{dfn}

Recalling that a quasi-perfect class is by definition one having the form $(a,b;\mathcal{W}(c,d))$ that belongs to $\tilde{\mathcal{E}}$, the preceding discussion almost immediately implies that:

\begin{cor}\label{bqperfect}
If $C=(a,b;\mathcal{W}(c,d))$ with $\gcd(c,d)=1$ is a quasi-perfect class, then for all $k\in \N$ the class $C^{(k)}$ is also quasi-perfect.
\end{cor}

\begin{proof}
By construction, the operation $C\mapsto C^{(k)}$ preserves the self-intersection number.  Writing $C^{(k)}=\left(\frac{x_k+\ep}{2},\frac{x_k-\ep}{2};\mathcal{W}(x_k+\delta_k,x_k-\delta_k)\right)$ with $x_k,\delta_k$ as in Definition \ref{bmovedef}, the argument preceding Definition \ref{bmovedef} shows that  $C^{(k)}$ will have Chern number $1$ provided that $\gcd(x_k+\delta_k,x_k-\delta_k)=1$.  We have $x_k+\delta_k=P_{2k+1}x+H_{2k+1}\delta$ and $x_k-\delta_k=P_{2k-1}x+H_{2k-1}\delta$ by repeated use of (\ref{hpadd}).  In particular \begin{equation}\label{ind-gcd} x_k-\delta_k=x_{k-1}+\delta_{k-1}.\end{equation}  Also, \begin{align*}(x_k+\delta_k)+(x_{k-1}-\delta_{k-1})&=(P_{2k+1}+P_{2k-3})x+(H_{2k+1}+H_{2k-3})\delta\\ &=6(P_{2k-1}x+H_{2k-1}\delta)=6(x_{k-1}+\delta_{k-1})\end{align*} where we have used (\ref{4gapP}) and (\ref{4gapH}).  Together with (\ref{ind-gcd}) this shows that the ideal in $\Z$ generated by $x_k+\delta_k$ and $x_{k}-\delta_k$ is the same as that generated by $x_{k-1}+\delta_{k-1}$ and $x_{k-1}-\delta_{k-1}$.  So by induction on $k$ we will have $\gcd(x_k+\delta_k,x_k-\delta_k)=1$ if and only if $\gcd(x_0+\delta_0,x_0-\delta_0)=1$.  But $(x_0+\delta_0,x_0-\delta_0)=(c,d)$ so this follows from the hypothesis of the corollary.
\end{proof}

It turns out that the $k$th-order Brahmagupta move also preserves the set of perfect classes, not just the set of quasi-perfect classes:

\begin{prop}\label{pellcrem}
Let $x,\ep,\delta,k\geq 0$ with $x\geq \delta$.  Let $x_k=H_{2k}x+2P_{2k}\delta$ and $\delta_k=P_{2k}x+H_{2k}\delta$. Then \[ \left(\frac{x+\ep}{2},\frac{x-\ep}{2};\mathcal{W}(x+\delta,x-\delta)\right)\mbox{ and } \left(\frac{x_k+\ep}{2},\frac{x_k-\ep}{2};\mathcal{W}(x_k+\delta_k,x_k-\delta_k)\right)\] are Cremona equivalent. In particular if $C=(a,b;\mathcal{W}(c,d))\in\mathcal{E}$ with $\gcd(c,d)=1$ then also $C^{(k)}\in\mathcal{E}$.
\end{prop}

\begin{proof}  
First observe that, using (\ref{hpadd}), \[ 3H_{2k}+4P_{2k}=3H_{2k}+2(H_{2k+1}-H_{2k})=2H_{2k+1}+H_{2k}=H_{2k+2} \] and \[ 2H_{2k}+3P_{2k}=2(P_{2k+1}-P_{2k})+3P_{2k}=2P_{2k+1}+P_{2k}=P_{2k+2},\] and so (since $H_2=3$ and $P_2=2$) \begin{align*} \left(H_2x_k+2P_{2}\delta_k,P_2x_k+H_2\delta_k\right)&= \left(3(H_{2k}x+2P_{2k}\delta)+4(P_{2k}x+H_{2k}\delta),2(H_{2k}x+2P_{2k}\delta) +3(P_{2k}x+H_{2k}\delta)\right)
\\ &= \left(H_{2k+2}x+2P_{2k+2}\delta,P_{2k+2}x+H_{2k+2}\delta\right).\end{align*}
Thus the map $(x,\delta)\mapsto (x_{k+1},\delta_{k+1})$ can be factored as the composition of the maps $(x,\delta)\mapsto (x_k,\delta_k)$ and $(x,\delta)\mapsto (x_1,\delta_1)$, and so by induction $(x,\delta)\mapsto (x_k,\delta_k)$ is just the $k$-fold composition of the map $(x,\delta)\mapsto (x_1,\delta_1)$.  This implies that it suffices to prove the proposition when $k=1$, in which case $(x_1,\delta_1)=(3x+4\delta,2x+3\delta)$.

We then have (since we assume $x\geq \delta\geq 0$) \begin{align*} \mathcal{W}(x_1+\delta_1,x_1-\delta_1)&=\mathcal{W}(5x+7\delta,x+\delta)
\\ \\&=\left((x+\delta)^{\times 5}\right)\sqcup\mathcal{W}(2\delta,x+\delta)=\left((x+\delta)^{\times 5},2\delta\right)\sqcup\mathcal{W}(2\delta,x-\delta).\end{align*}

So \begin{align*} & \left(\frac{x_1+\ep}{2},\frac{x_1-\ep}{2};\mathcal{W}(x_1+\delta_1,x_1-\delta_1)\right)=\left(\frac{3x+4\delta+\ep}{2},\frac{3x+4\delta-\ep}{2};\mathcal{W}(x_1+\delta_1,x_1-\delta_1)\right)
\\ \qquad & =\left\langle 2x+3\delta;\frac{x+2\delta+\ep}{2},\frac{x+2\delta-\ep}{2},(x+\delta)^{\times 4},2\delta,\mathcal{W}(2\delta,x-\delta)\right\rangle.\end{align*}
Applying Proposition \ref{xiaction} shows that this class is Cremona equivalent (via $\Xi^{(2)}=\frak{c}_{36}\circ \frak{c}_{456}\circ \frak{c}_{236}\circ \frak{c}_{012}\circ \frak{c}_{345}$) to \[ \left\langle \delta;\frac{-x+2\delta+\ep}{2},\frac{-x+2\delta-\ep}{2},0^{\times 5},\mathcal{W}(2\delta,x-\delta)\right\rangle,\] which after the usual change of basis (\ref{convert1}) is equal to \[ \left(\frac{x+\ep}{2},\frac{x-\ep}{2};x-\delta,0^{\times 5},\mathcal{W}(2\delta,x-\delta)\right).\]  But $(x-\delta)\sqcup \mathcal{W}(2\delta,x-\delta)=\mathcal{W}(x+\delta,x-\delta)$, so after deleting zeros the above class simply becomes 
$\left(\frac{x+\ep}{2},\frac{x-\ep}{2};\mathcal{W}(x+\delta,x-\delta)\right)$.
\end{proof}

\subsection{Tiling}

We note here a simple criterion for a class to belong to the set  $\bar{\mathcal{C}}_K(X_{N+1})$ that appears in Proposition \ref{background}.  Throughout this section a ``$p$-by-$q$ rectangle'' means a subset of $\mathbb{R}^2$ that is given as a product of intervals having lengths $p$ and $q$, not necessarily in that order. A ``square of sidelength $p$'' is a $p$-by-$p$ rectangle.

\begin{prop}\label{squaretile}
Suppose that $r,s_0,\ldots,s_N\in\R_{>0}$ have the property that there are squares $R,S_0,\ldots,S_N$ of respective sidelengths $r,s_1,\ldots,s_N$ such that the interiors $S_{i}^{\circ}$ of the $S_i$ are disjoint and $\cup_{i=0}^{N}S_{i}^{\circ}\subset R$.  Then $\langle r;s_0,\ldots,s_N\rangle\in\bar{\mathcal{C}}_{K}(X_{N+1})$. 
\end{prop}

\begin{proof}
By \cite[Proposition 2.1.B]{MP} it suffices to show that there is a symplectic embedding $\coprod_{i=0}^{N}B^4(s_i)^{\circ}\hookrightarrow B(r)^{\circ}$.  

For any $v>0$ let us write $\square(v)=(0,v)\times (0,v)$ and $\triangle(v)=\{(x_1,x_2)\in (0,\infty)^2|x_1+x_2<v\}$. Also for any subsets $A,B\subset \R^2$ let us denote by $A\times_L B$ the ``Lagrangian product'' \[ A\times_L B=\{(x_1,y_1,x_2,y_2)|(x_1,x_2)\in A,(y_1,y_2)\in B\}. \]  (We of course use the symplectic form $dx_1\wedge dy_1+dx_2\wedge dy_2$ on $\R^4$.)

Now \cite[Proposition 5.2]{T} states\footnote{Note that what Traynor denotes $B^4(s)$ is what we would denote $B(\pi s)$.} that, for any $s>0$, there is a symplectomorphism $B(s)^{\circ}\cong \square(\pi)\times_L \triangle(s/\pi)$.  The map \[ (x_1,y_1,x_2,y_2)\mapsto \left(\frac{s}{\pi}x_1,\frac{\pi}{s}y_1,\frac{s}{\pi}x_2,\frac{\pi}{s}y_2\right) \] is a symplectomorphism of $\R^4$ which maps $\square(\pi)\times_L \triangle(s/\pi)$ to $\square(s)\times_L \triangle(1)$.  Thus the existence of a symplectic embedding $\coprod_{i=0}^{N}B(s_i)^{\circ}\hookrightarrow B(r)^{\circ}$ is equivalent to the existence of a symplectic embedding \begin{equation}\label{squareembed} \coprod_{i=0}^{N}\left(\square(s_i)\times_L \triangle(1)\right)\hookrightarrow \square(r)\times_L \triangle(1).\end{equation} But the hypothesis of the proposition implies that the squares $\square(s_1),\ldots,\square(s_k)$ can be arranged by translations to be disjoint and still contained in $\square(r)$;  applying these translations in the $x_1,x_2$ directions in $\R^4$ then gives the desired embedding (\ref{squareembed}).
\end{proof}

\begin{cor}\label{tilecrit} Let $a_1,b_1,\ldots,a_m,b_m,r>0$ and suppose that there are $a_i$-by-$b_i$ rectangles $T_i$ ($i=1,\ldots m$) such that the $T_{i}^{\circ}$ are disjoint and such that $\cup_{i=1}^{m}T_{i}^{\circ}$ is contained in a square of sidelength $r$.  Then the class \[ \langle
r;\mathcal{W}(a_1,b_1),\ldots,\mathcal{W}(a_m,b_m)\rangle \] belongs to $\bar{\mathcal{C}}(X_{N})$ where $N$ is the sum of the lengths of the weight sequences $\mathcal{W}(a_i,b_i)$.
\end{cor}

\begin{proof}
Indeed if we write $\mathcal{W}(a_i,b_i)=(s_{i1},\ldots,s_{ik})$ then, as noted earlier, an $a_i$-by-$b_i$ rectangle can be divided into squares of sidelength $s_{i1},\ldots,s_{ik_i}$ with disjoint interiors.  So we can simply apply Proposition \ref{squaretile} to a collection of squares of sidelengths $s_{ij}$ ($1\leq i\leq m,1\leq j\leq k_i$).
\end{proof}

We will say that a class $\langle
r;\mathcal{W}(a_1,b_1),\ldots,\mathcal{W}(a_m,b_m)\rangle$ ``\textbf{satisfies the tiling criterion}'' if it obeys the hypothesis of Corollary \ref{tilecrit}.  Our most common method of showing that a general class belongs to $\bar{\mathcal{C}}_K(X_{N})$ will be to find a sequence of Cremona moves which converts it to a class that satisfies the tiling criterion.

\section{Proof of Theorem \ref{fmsup}} \label{fmsect}

\subsection{The Frenkel-M\"uller classes}

Theorem \ref{fmsup} asserts that, for any $\beta\geq 1$ and $\alpha\leq \sigma^2=3+2\sqrt{2}$, $C_{\beta}(\alpha)$ is given as the supremum of the obstructions $\Gamma_{\alpha,\beta}(FM_n)$ induced by the Frenkel-M\"uller classes $FM_n$.  We now recall the definition of the classes $FM_n$ and describe more explicitly $\sup_n\Gamma_{\alpha,\beta}(FM_n)$; in particular we will see that for any given $\beta>1$ this  reduces to a supremum over a finite set that is independent of $\alpha$.

The classes $FM_n$, which are shown to belong to $\mathcal{E}$ in \cite[Theorem 5.1]{FM}, are given by \begin{equation}\label{fmclass} FM_n=\left\{\begin{array}{ll} \left(P_{n+1},P_{n+1};\mathcal{W}(H_{n+2},H_{n})\right) & n\mbox{ even}\\
\left(\frac{H_{n+1}+1}{2},\frac{H_{n+1}-1}{2};\mathcal{W}(P_{n+2},P_{n})\right) & n\mbox{ odd}\end{array}\right..
\end{equation}

(The slightly more complicated formula in \cite{FM} is equivalent to this by (\ref{4hp}) and (\ref{2hp}).) While the definition in \cite{FM} assumes that $n\geq 0$, we can also set $n=-1$; this yields the class $FM_{-1}=(1,0;1)$ which also belongs to $\mathcal{E}$.

\begin{remark}
In terms of Definition \ref{bmovedef} we have $FM_{2k-1}=FM_{-1}^{(k)}$ and $FM_{2k}=FM_{0}^{(k)}$.  So since it is easy to check that $FM_{-1}=(1,0;\mathcal{W}(1,1))\in\mathcal{E}$ and that $FM_{0}=(1,1;\mathcal{W}(3,1))\in\mathcal{E}$, Proposition \ref{pellcrem} leads to an easy proof that all of the $FM_n$ are perfect classes.
\end{remark}

Based on the definition in Proposition \ref{Gammadef}, we have \[ \Gamma_{\alpha,\beta}(FM_n)=\left\{\begin{array}{ll} \frac{H_n\alpha}{(\beta+1)P_{n+1}} & \alpha\leq \frac{H_{n+2}}{H_n}\\ \frac{H_{n+2}}{(\beta+1)P_{n+1}} & \alpha\geq \frac{H_{n+2}}{H_{n}}\end{array}\right.\mbox{ for $n$ even},\] and \[\quad \Gamma_{\alpha,\beta}(FM_n)=\left\{\begin{array}{ll} \frac{2P_n\alpha}{(\beta+1)H_{n+1}-(\beta-1)} & \alpha\leq \frac{P_{n+2}}{P_n}\\ \frac{2P_{n+2}}{(\beta+1)H_{n+1}-(\beta-1)} & \alpha\geq \frac{P_{n+2}}{P_{n}}\end{array}\right.\mbox{ for $n$ odd.}\]

For $n\geq -1$ let \begin{equation}\label{adef} \alpha_n=\left\{\begin{array}{ll}\frac{H_{n+2}}{H_n} & n\mbox{ even}\\ \frac{P_{n+2}}{P_n} & n\mbox{ odd}\end{array}\right.,\end{equation} and
let us abbreviate \begin{equation}\label{gammadef} \gamma_{n,\beta}=\Gamma_{\alpha_n,\beta}(FM_{n})=\left\{\begin{array}{ll} \frac{H_{n+2}}{(\beta+1)P_{n+1}} & n\mbox{ even} \\ \frac{2P_{n+2}}{(\beta+1)H_{n+1}-(\beta-1)} & n\mbox{ odd}.\end{array}\right.\end{equation} for the value taken by $\Gamma_{\alpha,\beta}(FM_n)$ for all $\alpha\geq \alpha_n$.  Note that $\gamma_{-1,\beta}=1$ for all $\beta$ (corresponding to the non-squeezing bound), but all other $\gamma_{n,\beta}$ depend nontrivially on $\beta$. 


For the remainder of this subsection we will examine more closely the quantity \begin{equation}\label{fmbound} \sup_n\Gamma_{\alpha,\beta}(FM_n),\end{equation} leading to the more explicit formula in Proposition \ref{supformula}.  First of all notice that the $\alpha_n$ from (\ref{adef}) form a strictly increasing sequence by the first two inequalities in Proposition \ref{orderratios}; this sequence converges to $\sigma^2$ by (\ref{closedform}).
When $\beta=1$ it is also true that the numbers $\gamma_{n,\beta}$ form a strictly increasing sequence as $n$ varies, but we will see presently that this is \textbf{not} the case when $\beta>1$, as a result of which Theorem \ref{fmsup} implies qualitatively different behavior for $\beta>1$ than for $\beta=1$. 

\begin{prop}\label{ordergamma} The various $\gamma_{n,\beta}$ satisfy the following relationships, for $k\in \N$ and $\beta\geq 1$: 
\begin{itemize}
\item[(i)] $\gamma_{2k+2,\beta}>\gamma_{2k,\beta}$ for all $\beta$.
\item[(ii)] $\gamma_{2k+1,\beta}>\gamma_{2k,\beta}$ for all $\beta$.
\item[(iii)] $\gamma_{2k,\beta}>\gamma_{2k-1,\beta}$ if and only if $\beta<1+\frac{2}{H_{2k+2}-1}$.
\item[(iv)] $\gamma_{2k+1,\beta}>\gamma_{2k-1,\beta}$ if and only if $\beta<1+\frac{2}{P_{2k+2}-1}$.
\end{itemize}
The statements (iii) and (iv) also hold with their strict inequalities replaced by non-strict inequalities.
\end{prop}

\begin{proof} First we use (\ref{phloc}) to see that \begin{align*} \frac{\gamma_{2k+2,\beta}}{\gamma_{2k,\beta}}&=\frac{H_{2k+4}/((1+\beta)P_{2k+3})}{H_{2k+2}/((1+\beta)P_{2k+1})}=\frac{2H_{2k+3}P_{2k+1}+H_{2k+2}P_{2k+1}}{2H_{2k+2}P_{2k+2}+H_{2k+2}P_{2k+1}}
\\ &=  \frac{2H_{2k+2}P_{2k+2}+2+H_{2k+2}P_{2k+1}}{2H_{2k+2}P_{2k+2}+H_{2k+2}P_{2k+1}}>1,\end{align*} proving (i).

Also, using (\ref{htop}), we see that 
\begin{align*} \frac{\gamma_{2k,\beta}}{\gamma_{2k+1,\beta}}&=\frac{H_{2k+2}/((\beta+1)P_{2k+1})}{2P_{2k+3}/((\beta+1)H_{2k+2}-(\beta-1))} = \frac{H_{2k+2}^{2}}{2P_{2k+3}P_{2k+1}}-\frac{\beta-1}{\beta+1}\frac{H_{2k+2}}{2P_{2k+1}P_{2k+3}} \\ &= \frac{2P_{2k+3}P_{2k+1}-1}{2P_{2k+3}P_{2k+1}}-\frac{\beta-1}{\beta+1}\frac{H_{2k+2}}{2P_{2k+1}P_{2k+3}}<1 ,\end{align*} proving (ii).

 On the other hand, we calculate using (\ref{hh}) and (\ref{pp}) that \begin{align*} H_{2k}H_{2k+2}=2P_{2k}P_{2k+2}+3=2P_{2k+1}^{2}+1\end{align*} and hence that \begin{align*} 
\frac{\gamma_{2k,\beta}}{\gamma_{2k-1,\beta}}&= \frac{H_{2k+2}/((\beta+1)P_{2k+1})}{2P_{2k+1}/((\beta+1)H_{2k}-(\beta-1))} = \frac{H_{2k}H_{2k+2}}{2P_{2k+1}^{2}}-\frac{\beta-1}{\beta+1}\frac{H_{2k+2}}{2P_{2k+1}^{2}}
\\&= 1+\frac{1}{2P_{2k+1}^{2}}\left(1-\frac{\beta-1}{\beta+1}H_{2k+2}\right).
\end{align*}  Thus $\gamma_{2k,\beta}> \gamma_{2k-1,\beta}$ if and only if $1-\frac{\beta-1}{\beta+1}H_{2k+2}>0$, \emph{i.e.} if and only if $\beta<1+\frac{2}{H_{2k+2}-1}$, as stated in (iii).

Finally, we see from (\ref{phloc}) that \begin{align*} P_{2k+3}H_{2k}-P_{2k+1}H_{2k+2}&= (2P_{2k+2}+P_{2k+1})H_{2k}-P_{2k+1}(2H_{2k+1}+H_{2k})
\\ &= 2(P_{2k+2}H_{2k}-P_{2k+1}H_{2k+1})=2 \end{align*} and hence that \begin{align*} 
\frac{\gamma_{2k+1,\beta}}{\gamma_{2k-1,\beta}}&=\frac{2P_{2k+3}/((\beta+1)H_{2k+2}-(\beta-1))}{2P_{2k+1}/((\beta+1)H_{2k}-(\beta-1))}=\frac{(\beta+1)P_{2k+3}H_{2k}-(\beta-1)P_{2k+3}}{(\beta+1)P_{2k+1}H_{2k+2}-(\beta-1)P_{2k+1}} 
\\ &= \frac{(\beta+1)(P_{2k+1}H_{2k+2}+2)-(\beta-1)(P_{2k+1}+2P_{2k+2})}{(\beta+1)P_{2k+1}H_{2k+2}-(\beta-1)P_{2k+1}}=1+\frac{2(\beta+1)-2(\beta-1)P_{2k+2}}{(\beta+1)P_{2k+1}H_{2k+2}-(\beta-1)P_{2k+1}},
\end{align*} which is greater than one if and only if $(\beta-1)P_{2k+2}<\beta+1$, \emph{i.e.} if and only if $\beta<1+\frac{2}{P_{2k+2}-1}$, proving (iv).
\end{proof}

Let us write $b_{-1}=\infty$ and, for $n\in \N$, \begin{equation}\label{bdef} b_n=\left\{\begin{array}{ll} 1+\frac{2}{P_{n+2}-1} & n\mbox{ even}\\ 1+\frac{2}{H_{n+1}-1} & n\mbox{ odd} \end{array}\right.. \end{equation}  So $b_0=1+\frac{2}{2-1}=3$, $b_1=1+\frac{2}{3-1}=2$, $b_2=1+\frac{2}{12-1}=\frac{13}{11}, b_3=\frac{18}{16},\,b_4=\frac{71}{69},b_5=\frac{100}{98},\ldots$.  Since for $n\geq 2$ it holds that $P_n<H_n<P_{n+1}$ we have $b_0>b_1>\cdots>b_n>\cdots$.  Also evidently $\lim_{n\to\infty}b_n=1$.

The following can be derived directly from Proposition \ref{ordergamma}:

\begin{cor}\label{orderint}
If $b_{2k}<\beta<b_{2k-1}$ then $\gamma_{2k-1,\beta}\geq \gamma_{n,\beta}$ for all $n\in \N\cup\{-1\}$, and we have \[ \gamma_{-1,\beta}<\gamma_{0,\beta}<\cdots<\gamma_{2k-1,\beta}.\] 
On the other hand if $b_{2k+1}<\beta<b_{2k}$  then $\gamma_{2k+1,\beta}\geq \gamma_{n,\beta}$ for all $n\in \N\cup\{-1\}$, and we have $\gamma_{n,\beta}<\gamma_{n+1,\beta}$ for all $n\in \{-1,0,\ldots,2k-2\}$, while $\gamma_{2k-2,\beta}<\gamma_{2k,\beta}<\gamma_{2k-1,\beta}<\gamma_{2k+1,\beta}$.
\end{cor}

Recalling the definitions of $\Gamma_{\alpha,\beta}(FM_n)$  from  Proposition \ref{Gammadef} and $\gamma_{n,\beta}$ from (\ref{gammadef}), observe that if $\beta,m,n$ have the property that $m>n$ (so that $\alpha_m>\alpha_n$) and $\gamma_{m,\beta}<\gamma_{n,\beta}$ then we in fact have $\Gamma_{\alpha,\beta}(FM_n)>\Gamma_{\alpha,\beta}(FM_m)$ for all $\alpha$.  Hence (using continuity considerations at the endpoints of the intervals) Corollary \ref{orderint} implies that, for any given $\beta> 1$, the supremum on the right hand side of Theorem \ref{fmsup} becomes a maximum over a finite set as follows:

\begin{prop} \label{supreduce}
If $\beta\in [b_{2k},b_{2k-1}]$ then \[ \sup_{n\in \N\cup\{-1\}}\Gamma_{\alpha,\beta}(FM_n)=\max\{\Gamma_{\alpha,\beta}(FM_{-1}),\Gamma_{\alpha,\beta}(FM_0),\ldots,\Gamma_{\alpha,\beta}(FM_{2k-1})\},\] and if $\beta\in [b_{2k+1},b_{2k}]$ then \[ \sup_{n\in \N\cup\{-1\}}\Gamma_{\alpha,\beta}(FM_n)=\max\{\Gamma_{\alpha,\beta}(FM_{-1}),\Gamma_{\alpha,\beta}(FM_0),\ldots,\Gamma_{\alpha,\beta}(FM_{2k-1}),\Gamma_{\alpha,\beta}(FM_{2k+1})\}.\]
\end{prop}

For $n\in \N$, let \begin{equation}\label{sndef} s_n(\beta)=\frac{\gamma_{n-1,\beta} \alpha_n}{\gamma_{n,\beta}};\end{equation} thus $s_n(\beta)$ is the unique value of $\alpha$ at which the constant piece of $\Gamma_{n-1,\beta}$ coincides with the nonconstant piece of $\Gamma_{n,\beta}$. Obviously if $\gamma_{n-1,\beta}<\gamma_{n,\beta}$ then $s_{n}(\beta)<\alpha_n$.

\begin{prop} \label{asorder}
If $n$ is even then $\alpha_{n-1}<s_n(\beta)$ for all $\beta$, while if $n$ is odd then $\alpha_{n-1}<s_n(\beta)$ provided that $\beta<1+\frac{2}{H_{n-1}-1}=b_{n-2}$.
\end{prop}

\begin{proof}  Evidently $\alpha_{n-1}<s_n(\beta)$ if and only if $\frac{\gamma_{n,\beta}}{\alpha_n}<\frac{\gamma_{n-1,\beta}}{\alpha_{n-1}}$.  We see from the definitions that, for $k\in \N$, \begin{equation}\label{gammaovera} \frac{\gamma_{2k,\beta}}{\alpha_{2k}}=\frac{H_{2k}}{(\beta+1)P_{2k+1}}\qquad\qquad \frac{\gamma_{2k-1,\beta}}{\alpha_{2k-1}}=\frac{2P_{2k-1}}{(\beta+1)H_{2k}-(\beta-1)}.\end{equation}  So \begin{align*}
\frac{\frac{\gamma_{2k,\beta}}{\alpha_{2k}}}{\frac{\gamma_{2k-1,\beta}}{\alpha_{2k-1}}}&= \frac{(\beta+1)H_{2k}^{2}-(\beta-1)H_{2k}}{(\beta+1)2P_{2k+1}P_{2k-1}}=\frac{(\beta+1)(2P_{2k+1}P_{2k-1}-1)-(\beta-1)H_{2k}}{(\beta+1)2P_{2k+1}P_{2k-1}}<1
\end{align*} where we have used (\ref{htop}).   This proves the first clause of the proposition.

For the second clause, write $n=2k+1$ and note that (\ref{hh}) and (\ref{pp}) imply that $H_{2k}H_{2k+2}=2P_{2k}P_{2k+2}+3=2P_{2k+1}^{2}+1$, and so \begin{align*}
\frac{\frac{\gamma_{2k+1,\beta}}{\alpha_{2k+1}}}{\frac{\gamma_{2k,\beta}}{\alpha_{2k}}}&= \frac{\frac{2P_{2k+1}}{(\beta+1)H_{2k+2}-(\beta-1)}}{\frac{H_{2k}}{(\beta+1)P_{2k+1}}} = \frac{2P_{2k+1}^{2}}{H_{2k}H_{2k+2}-\frac{\beta-1}{\beta+1}H_{2k}}
\\&= \frac{2P_{2k+1}^{2}}{2P_{2k+1}^{2}+1-\frac{\beta-1}{\beta+1}H_{2k}},
\end{align*} which is smaller than one provided that $\frac{\beta-1}{\beta+1}H_{2k}<1$, \emph{i.e.} provided that $\beta<1+\frac{2}{H_{2k}-1}$.
\end{proof}

\begin{cor}\label{firstpiece}
If $\beta\in [b_{2k+1},b_{2k-1}]$ then for all $n\in \{0,\ldots,2k-1\}$ we have $\alpha_{n-1}<s_n(\beta)\leq \alpha_{n}$. Hence \[ \max\{\Gamma_{\alpha,\beta}(FM_{-1}),\ldots,\Gamma_{\alpha,\beta}(FM_{2k-1})\}=\left\{\begin{array}{ll} \Gamma_{\alpha,\beta}(FM_{-1}) & 1\leq \alpha\leq s_0(\beta) \\
\Gamma_{\alpha,\beta}(FM_n) & s_n(\beta)\leq \alpha\leq s_{n+1}(\beta)\,\, (0\leq n\leq 2k-2) \\ \Gamma_{\alpha,\beta}(FM_{2k-1}) & \alpha\geq s_{2k-1}(\beta) \end{array}\right. \]
\end{cor}

\begin{proof}  Throughout the proof we only consider values of $n$ in the set $\{0,\ldots,2k-1\}$.

Corollary \ref{orderint} shows that $\gamma_{n-1,\beta}\leq \gamma_{n,\beta}$ for all such $n$, which as noted earlier implies that $s_n(\beta)\leq \alpha_n$.  Since the $b_n$ form a strictly decreasing sequence, we have $b_{2k-1}<b_{n-2}$ for $n=0,\ldots,2k-1$, so Proposition \ref{asorder} applies to show that $\alpha_{n-1}<s_n(\beta)$ for all $\beta$ in the interval under consideration.  This proves the first sentence of the proposition.

Given this, since the $\Gamma_{\cdot,\beta}(FM_n)$ are all globally nondecreasing and, on the interval $[\alpha_{n-1},\alpha_n]$, $\Gamma_{\cdot,\beta}(FM_{n-1})$ is constant while $\Gamma_{\cdot,\beta}(FM_n)$ is strictly increasing, with $\Gamma_{s_n(\beta),\beta}(FM_{n-1})=\Gamma_{s_n(\beta),\beta}(FM_n)$, it follows that $\Gamma_{\alpha,\beta}(FM_{n-1})\geq \Gamma_{\alpha,\beta}(FM_n)$ for all $\alpha\in [\alpha_{n-1},s_n(\beta)]$ and $\Gamma_{\alpha,\beta}(FM_{n})\geq \Gamma_{\alpha,\beta}(FM_{n-1})$ on $[s_n(\beta),\alpha_n]$.  Moreover since both $\Gamma_{\cdot,\beta}(FM_
{n-1})$ and $\Gamma_{\cdot,\beta}(FM_n)$ are constant on $[\alpha_n,\infty)$ and linear on $[1,\alpha_{n-1}]$ these inequalities extend to $\Gamma_{\cdot,\beta}(FM_{n-1})\geq \Gamma_{\cdot,\beta}(FM_n)$ on $[1,s_n(\beta)]$ and $\Gamma_{\cdot,\beta}(FM_n)\geq \Gamma_{\cdot,\beta}(FM_{n-1})$ on $[s_n(\beta),\infty)$.  Since the $s_n(\beta)$ form an increasing sequence in $n$, applying this repeatedly shows that, for $j\geq n$, $\Gamma_{\cdot,\beta}(FM_{n-1})\geq \Gamma_{\cdot,\beta}(FM_j)$ on $[1,s_n(\beta)]$ and for $j< n$, $\Gamma_{\cdot,\beta}(FM_n)\geq \Gamma_{\cdot,\beta}(FM_j)$ on $[s_n(\beta),\infty)$.  Hence on each interval $[s_n(\beta),s_{n+1}(\beta)]$ the $\Gamma_{\cdot,\beta}(FM_j)$ are maximized by setting $j=n$, while $\Gamma_{\cdot,\beta}(FM_{-1})$ is maximal on $[1,s_0(\beta)]$ and $\Gamma_{\cdot,\beta}(FM_{2k-1})$ is maximal on $[s_{2k-1}(\beta),\infty)$.  This is precisely what is stated in the second sentence of the corollary.
\end{proof}

In view of Proposition \ref{supreduce}, Corollary \ref{firstpiece} gives an explicit piecewise formula for $\sup_{n\in \N\cup\{-1\}}\Gamma_{\alpha,\beta}(FM_n)$ in the case that $\beta$ lies in an interval of the form $[b_{2k},b_{2k-1}]$.  If instead $\beta\in (b_{2k+1},b_{2k})$ for some $k$ then we  have 
$\sup_{n\in \N\cup\{-1\}}\Gamma_{\alpha,\beta}(FM_n)=\max\{\Gamma_{\alpha,\beta}(FM_{-1}),\ldots,\Gamma_{\alpha,\beta}(FM_{2k-1}),\Gamma_{\alpha,\beta}(FM_{2k+1})\}$ and so we need to take into account the relationship of $\Gamma_{\cdot,\beta}(FM_{2k-1})$ and $\Gamma_{\cdot,\beta}(FM_{2k+1})$.  Accordingly let us write \begin{equation}\label{sprimedef} s'_{2k}(\beta)=\frac{\gamma_{2k-1,\beta}\alpha_{2k+1}}{\gamma_{2k+1,\beta}},\end{equation} so that $s'_{2k}(\beta)$ is the value of $\alpha$ at which the linear piece of $\Gamma_{\cdot,\beta}(FM_{2k+1})$ coincides with the constant piece of $\Gamma_{\cdot,\beta}(FM_{2k-1})$.  Since for $\beta\in [b_{2k+1},b_{2k}]$ we have $\gamma_{2k-1,\beta}\leq \gamma_{2k+1,\beta}$, it holds that $s'_{2k}(\beta)\leq \alpha_{2k+1}$.  To compare $s'_{2k}(\beta)$ to $\alpha_{2k-1}$, we first use (\ref{gammaovera}) to see that \begin{align}\nonumber \frac{\alpha_{2k-1}}{s'_{2k}(\beta)}&=\frac{\frac{\gamma_{2k+1,\beta}}{\alpha_{2k+1}}}{\frac{\gamma_{2k-1,\beta}}{\alpha_{2k-1}}}=\frac{\frac{P_{2k+1}}{(\beta+1)H_{2k+2}-(\beta-1)}}{\frac{P_{2k-1}}{(\beta+1)H_{2k}-(\beta-1)}}
\\&=\frac{(\beta+1)P_{2k+1}H_{2k}-(\beta-1)P_{2k+1}}{(\beta+1)P_{2k-1}H_{2k+2}-(\beta-1)P_{2k-1}}.\label{sprimeratio}\end{align} 
Now \begin{align*} P_{2k+1}H_{2k}-P_{2k-1}H_{2k+2}&=(2P_{2k}+P_{2k-1})H_{2k}-P_{2k-1}(2H_{2k+1}+H_{2k})\\&=2(P_{2k}H_{2k}-P_{2k-1}H_{2k+1})=-2,
\end{align*} in view of which the numerator of (\ref{sprimeratio}) is smaller than the denominator for every $\beta\geq 1$.  

So when $\beta\in [b_{2k+1},b_{2k}]$ we have \begin{equation}\label{sprimea} \alpha_{2k-1}<s'_{2k}(b)\leq \alpha_{2k+1},\end{equation} and \[ \Gamma_{\alpha,\beta}(FM_{2k-1})\geq \Gamma_{\alpha,\beta}(FM_{2k+1})\mbox{ for }\alpha \leq s'_{2k}(\beta),\quad \Gamma_{\alpha,\beta}(FM_{2k+1})\geq \Gamma_{\alpha,\beta}(FM_{2k-1})\mbox{ for }\alpha\geq s'_{2k}(\beta).\]  The fact that $s'_{2k}(\beta)>\alpha_{2k-1}$ implies that $s'_{2k}(\beta)>s_{2k-1}(\beta)$, and so these calculations together with Corollary \ref{firstpiece} imply that, for $\beta\in [b_{2k+1},b_{2k}]$, \begin{align*} &
\max\{\Gamma_{\alpha,\beta}(FM_{-1}),\ldots,\Gamma_{\alpha,\beta}(FM_{2k-1}),\Gamma_{\alpha,\beta}(FM_{2k+1})\}\\ &\qquad\qquad=\left\{\begin{array}{ll} \Gamma_{\alpha,\beta}(FM_{-1}) & 1\leq \alpha\leq s_0(\beta) \\
\Gamma_{\alpha,\beta}(FM_n) & s_n(\beta)\leq \alpha\leq s_{n+1}(\beta)\,\, (0\leq n\leq 2k-2) \\ \Gamma_{\alpha,\beta}(FM_{2k-1}) &  s_{2k-1}(\beta)\leq \alpha\leq s'_{2k}(\beta) \\ \Gamma_{\alpha,\beta}(FM_{2k+1}) & \alpha \geq s'_{2k}(\beta) \end{array}\right. 
\end{align*}

For future reference we rephrase this derivation as follows.

\begin{prop}\label{supformula} If $\beta\in [b_{2k},b_{2k-1}]$ we have \[ \sup_{n\in \N\cup\{-1\}}\Gamma_{\alpha,\beta}(FM_n)=\left\{\begin{array}{ll} \gamma_{n-1,\beta} & \alpha_{n-1}\leq \alpha\leq s_n(\beta),\,n\in\{0,\ldots,2k-1\} \\
\frac{\gamma_{n,\beta}\alpha}{\alpha_n} & s_n(\beta)\leq \alpha\leq \alpha_n,\,n\in\{0,\ldots,2k-1\}  \\ 
\gamma_{2k-1,\beta} & \alpha\geq \alpha_{2k-1}\end{array}\right., \] and if $\beta\in [b_{2k+1},b_{2k}]$ then \[ 
\sup_{n\in \N\cup\{-1\}}\Gamma_{\alpha,\beta}(FM_n)=\left\{\begin{array}{ll} \gamma_{n-1,\beta} & \alpha_{n-1}\leq \alpha\leq s_n(\beta),\,n\in\{0,\ldots,2k-1\} \\
\frac{\gamma_{n,\beta}\alpha}{\alpha_n} & s_n(\beta)\leq \alpha\leq \alpha_n,\,n\in\{0,\ldots,2k-1\} \\ 
\gamma_{2k-1,\beta} & \alpha_{2k-1}\leq \alpha\leq s'_{2k}(\beta) \\ \frac{\gamma_{2k+1,\beta}\alpha}{\alpha_{2k+1}} & s'_{2k}(\beta)\leq \alpha\leq \alpha_{2k+1} \\
\gamma_{2k+1,\beta} & \alpha\geq \alpha_{2k+1}\end{array}\right..\]
Here $\alpha_n,\gamma_{n,\beta},b_n,s_n(\beta),s'_{2k}(\beta)$ are defined respectively in (\ref{adef}),(\ref{gammadef}),(\ref{bdef}),(\ref{sndef}), and (\ref{sprimedef}).  
\end{prop} 

Provided that $\beta$ lies in the interior $(b_{m},b_{m-1})$ of an interval between consecutive $b_n$, the $m+1$ values $\gamma_{-1,\beta},\gamma_{0,\beta},\ldots,\gamma_{m-1,\beta}$ (if $m$ is even) or $\gamma_{-1,\beta},\gamma_{0,\beta},\ldots,\gamma_{m-2,\beta},\gamma_{m,\beta}$ (if $m$ is odd) form a strictly increasing sequence by Corollary \ref{orderint}, and so the graph of $\sup_{n\in \N\cup\{-1\}}\Gamma_{\cdot,\beta}(FM_n)$ consists of $m$ distinct nontrivial ``steps'' from $\gamma_{-1,\beta}=1$ to $\gamma_{0,\beta}$, $\gamma_{0,\beta}$ to $\gamma_{1,\beta}$, and so on, ending at a step from $\gamma_{m-2,\beta}$ to either $\gamma_{m-1,\beta}$ or $\gamma_{m,\beta}$ depending on the parity of $m$.  (If $m=0$, so that $\beta\in (3,\infty)$, then $\sup_n\Gamma_{\alpha,\beta}(FM_n)$ is just the constant function $1$, corresponding to the non-squeezing theorem.)  As $\beta$ approaches $b_{m-1}$ from below, two of the heights $\gamma_{i,b}$ approach each other and so one of the steps collapses to a constant.

Note that since $H_n,P_n$ are each asymptotic to constants times $\sigma^n$ where $\sigma=1+\sqrt{2}$, the formula (\ref{bdef}) makes clear that $b_m-1$ is asymptotic to a constant times $\sigma^{-m}$. Thus the number of steps in the graph of $\sup_n\Gamma_{\cdot,\beta}(FM_n)$ is comparable to $\log(1/(\beta-1))$, which of course diverges to infinity as $\beta\to 1$, but does so rather slowly.  For example the interval $[b_{10},b_9)$ of values of $\beta$ for which the graph has $10$ steps  is $\left[\frac{13861}{13859},\frac{3364}{3362}\right)$.

\subsection{Sharpness of the lower bound} \label{fmsupproof}

As noted earlier, the existence of the Frenkel-M\"uller classes $FM_n$ immediately implies an inequality $C_{\beta}(\alpha)\geq \sup_{n\in\N\cup\{-1\}}\Gamma_{\alpha,\beta}(FM_n)$ for all $\alpha$, so to prove Theorem \ref{fmsup} we just need to establish the reverse inequality for $\alpha\leq \sigma^2$.
In fact, consulting the formula in Proposition \ref{supformula}, we see that it is sufficient to establish the reverse inequality at the various points $s_n(\beta)$ and (when $\beta\in (b_{2k+1},b_{2k})$) $s'_{2k}(\beta)$ that appear in that formula, together with a single point $\alpha$ with $\alpha\geq \sigma^2$.  Indeed if we know that $C_{\beta}(s_n(\beta))\leq \sup_n\Gamma_{s_n(\beta),\beta}(s_n(\beta))$, then obvious inequality $C_{\beta}(\alpha)\leq C_{\beta}(\alpha')$ for $\alpha\leq \alpha'$ will then imply that $C_{\beta}(\alpha)\leq \sup_n\Gamma_{\alpha,\beta}(FM_n)$  for $\alpha\in [\alpha_{n-1},s_n(\beta)]$, and the sublinearity inequality $C_{\beta}(t\alpha)\leq tC_{\beta}(\alpha)$ for $t\geq 1$ noted in the proof of Proposition \ref{Gammadef} will imply that $C_{\beta}(\alpha)\leq \sup_n\Gamma_{\alpha,\beta}(\alpha)$ for $\alpha\in [s_{n}(\beta),\alpha_{n}]$; similar remarks apply to the other intervals in (\ref{supformula}).  

Thus we must show that:
\begin{itemize}
\item[(I)] $C_{\beta}(s_n(\beta))\leq \gamma_{n-1,\beta}$ for $\beta\leq b_n$ if $n$ is odd, and for $\beta\leq b_{n+1}$ if $n$ is even,
\item[(II)] $C_{\beta}(s'_{2k}(\beta))\leq \gamma_{2k-1,\beta}$ for $b_{2k+1}\leq \beta\leq b_{2k}$, and 
\item[(III)] $C_{\beta}(\alpha)\leq \gamma_{2k-1,\beta}$ for some $\alpha\geq \sigma^2$, whenever $b_{2k}\leq \beta\leq b_{2k-2}$.
\end{itemize}

Note that it is sufficient to prove (I), (II), and (III) when $\beta$ (and hence also $s_n(\beta)$ and $s'_{2k}(\beta)$) is rational, as this will be sufficient to prove the equality $C_{\beta}(\alpha)=\sup_n\Gamma_{\alpha,\beta}(FM_n)$ for $\beta\in [1,\infty)\cap\Q$ and $\alpha\leq 3+2\sqrt{2}$, and both sides of this equality are easily seen to vary continuously with $\beta$. The forthcoming discussion will prove statements (I), (II), and (III) for rational $\beta>1$.  Specifically: \begin{itemize} \item The case of (I) with $n$ odd follows from Proposition \ref{bigreduce} and Corollary \ref{evencor}, while the case of (I) with $n$ even follows from Proposition \ref{typodd}.
\item (II) follows from Proposition \ref{excodd}.
\item Since the condition $b_{2k}\leq \beta\leq b_{2k-2}$ is equivalent to the condition that $P_{2k}(\beta-1)\leq \beta+1\leq P_{2k+2}(\beta-1)$, (III) follows by combining Propositions \ref{p2k} and \ref{lastplat}.
\end{itemize}

The statements listed above all amount to showing that a certain class of the form $\left(\gamma\beta,\gamma;\mathcal{W}(1,\alpha)\right)$ belongs to the appropriate symplectic cone closure $\bar{\mathcal{C}}_K(X_N)$.  Proposition \ref{genexp} shows that, if $\alpha\in\left[\frac{P_{2k+1}}{P_{2k-1}},\frac{P_{2k+2}}{P_{2k}}\right]\cap\Q$ and $\gamma\geq \frac{\alpha+1}{2\beta+2}$, then $\left(\gamma\beta,\gamma;\mathcal{W}(1,\alpha)\right)$ is Cremona equivalent to the class denoted there by $\Sigma_{\alpha,\beta,\gamma}^{k}$, which may be written as \[ \Sigma_{\alpha,\beta,\gamma}^{k}=\langle Z;A,B,\mathcal{W}(C,D),\mathcal{W}(E,F)\rangle \] where:
\begin{equation}\label{zadef} Z=P_{2k+1}\left(\gamma(\beta+1)-1\right)-P_{2k}\alpha, \quad A=\frac{H_{2k}}{2}\gamma(\beta+1)-P_{2k+1}+\gamma\left(\frac{\beta-1}{2}\right),\end{equation} \[ B=\frac{H_{2k}}{2}\gamma(\beta+1)-P_{2k+1}-\gamma\left(\frac{\beta-1}{2}\right),\quad C=\frac{P_{2k+2}}{2}-\frac{P_{2k}}{2}\alpha,\quad D=P_{2k-1}\alpha-P_{2k+1},\] and  \[ E=\frac{P_{2k}}{2}\left(2\gamma(\beta+1)-\alpha-1\right) ,\quad  F= P_{2k+1}\left(2\gamma(\beta+1)-\alpha-1\right).\] (We use this notation even if $k=0$, although in that case $E$ (which is zero) and $F$ (which is typically nonzero) are not relevant to $\Sigma_{\alpha,\beta,\gamma}^{k}$. As noted in Remark \ref{k0}, when $k=0$ we do not need to assume that $\gamma\geq\frac{\alpha+1}{2\beta+2}$.) 

Throughout the rest of this section $Z,A,B,C,D,E,F$ will refer to the above quantities.

\subsubsection{The case $\gamma=\gamma_{2k,\beta}$}  The statements (I),(II),(III) in the beginning of Section \ref{fmsupproof} involve one case of embedding an ellipsoid $E(1,\alpha)^{\circ}$ into $\gamma_{2k,\beta}P(1,\beta)$ (with $\alpha=s_{2k+1}(\beta)$ and $\beta\leq b_{2k+1}$), and three cases of embedding an ellipsoid $E(1,\alpha)^{\circ}$ into $\gamma_{2k-1,\beta}P(1,\beta)$ (with $\alpha=s_{2k}(\beta)$ and $\beta\leq b_{2k+1}$, with $\alpha=s'_{2k}(\beta)$ and $b_{2k+1}\leq \beta\leq b_{2k}$, and with $\alpha$ equal to some value greater than $\sigma^2$ and $b_{2k}\leq \beta\leq b_{2k-2}$).  This subsection will
establish the one case involving $\gamma=\gamma_{2k,\beta}=\frac{H_{2k+2}}{(\beta+1)P_{2k+1}}$.

We accordingly assume that $\beta\leq b_{2k+1}=1+\frac{2}{H_{2k+2}-1}$ (equivalently, that $(\beta+1)\geq H_{2k+2}(\beta-1)$).  Proposition \ref{asorder} (and the fact that $b_{2k+1}<b_{2k-1}$) shows that $\alpha_{2k}\leq s_{2k+1}(\beta)$, while Proposition \ref{ordergamma} shows that $\gamma_{2k,\beta}\leq \gamma_{2k+1,\beta}$ and hence that $s_{2k+1}(\beta)=\frac{\gamma_{2k,\beta}}{\gamma_{2k+1,\beta}}\alpha_{2k+1}\leq \alpha_{2k+1}$.  In particular since $\alpha_{2k-1}=\frac{P_{2k+1}}{P_{2k-1}}<\alpha_{2k}$ and $\alpha_{2k+1}<\sigma^2<\frac{P_{2k+2}}{P_{2k}}$ we have $s_{2k+1}(\beta)\in \left[\frac{P_{2k+1}}{P_{2k-1}},\frac{P_{2k+2}}{P_{2k}}\right]$, and so Proposition \ref{bigreduce} (or, if $k=0$, Remark \ref{k0}) is applicable to the question of whether there is a symplectic embedding $E(1,s_{2k+1}(b))^{\circ}\hookrightarrow \gamma_{2k,b}P(1,b)$.


\begin{prop}\label{evenids}
Let $\gamma=\gamma_{2k,\beta}$ and $\alpha=s_{2k+1}(\beta)$.  Then for any $\beta$ we have: \begin{itemize}\item[(i)] $Z=2C$, \item[(ii)] $D+F=4C$, \item[(iii)] $A+B+E=C$, and \item[(iv)] $A=F=P_{2k+3}-P_{2k+1}\alpha$.\end{itemize} Under the additional assumption that $1\leq \beta\leq b_{2k+1}$, we have: \begin{itemize} \item[(v)] $A\leq C$, and \item[(vi)] $B\geq 0$.\end{itemize}
\end{prop}

\begin{proof}
By the definition of $\gamma_{2k,\beta}$ we have $P_{2k+1}\gamma(\beta+1)=H_{2k+2}$, so \[ Z=H_{2k+2}-P_{2k+1}-P_{2k}\alpha=P_{2k+2}-P_{2k}\alpha=2C,\] proving (i).  Next, notice that \begin{align*} F&=2P_{2k+1}\gamma(\beta+1)-P_{2k+1}\alpha-P_{2k+1}=(2H_{2k+2}-P_{2k+1})-P_{2k+1}\alpha
\\ &= (2P_{2k+2}+P_{2k+1})-P_{2k+1}\alpha=P_{2k+3}-P_{2k+1}\alpha,\end{align*} which proves the second inequality in (iv), and also implies that \[ F+D=(P_{2k+3}-P_{2k+1})-(P_{2k+1}-P_{2k-1})\alpha=2P_{2k+2}-2P_{2k}\alpha=4C,\] proving (ii).  Also, using again that $\gamma(\beta+1)=\frac{H_{2k+2}}{P_{2k+1}}$, we see that 
\begin{align}\nonumber A&=\frac{H_{2k}H_{2k+2}}{2P_{2k+1}}-P_{2k+1}+\frac{H_{2k+2}}{2P_{2k+1}}\frac{\beta-1}{\beta+1}
\\ &= \frac{1}{2P_{2k+1}}\left(H_{2k}H_{2k+2}-2P_{2k+1}^{2}+H_{2k+2}\frac{\beta-1}{\beta+1}\right) = \frac{1}{2P_{2k+1}}\left(1+H_{2k+2}\frac{\beta-1}{\beta+1}\right)\label{A}\end{align} (using (\ref{hh}) and (\ref{psquared})) and similarly \begin{equation}\label{B} B=\frac{1}{2P_{2k+1}}\left(1-H_{2k+2}\frac{\beta-1}{\beta+1}\right).\end{equation}  Since the condition that $\beta\leq b_{2k+1}$ is equivalent to the statement that $\frac{\beta-1}{\beta+1}\leq \frac{1}{H_{2k+2}}$ this last equation immediately implies (vi).  

Since $E=\frac{P_{2k}}{2P_{2k+1}}F$, we also find that \begin{align*} C-E&=\left(\frac{P_{2k+2}}{2}-\frac{P_{2k}}{2}\alpha\right)-\frac{P_{2k}}{2P_{2k+1}}(P_{2k+3}-P_{2k+1}\alpha)=\frac{1}{2P_{2k+1}}(P_{2k+1}P_{2k+2}-P_{2k}P_{2k+3})
\\ &=\frac{1}{2P_{2k+1}}(P_{2k+1}P_{2k+2}-2P_{2k}P_{2k+2}-P_{2k}P_{2k+1})=\frac{1}{2P_{2k+1}}\left(P_{2k+1}(P_{2k+2}-P_{2k})-2P_{2k}P_{2k+2}\right) \\ &= \frac{1}{P_{2k+1}}(P_{2k+1}^{2}-P_{2k}P_{2k+2})=\frac{1}{P_{2k+1}} \end{align*} where the last equation uses (\ref{psquared}).
But (\ref{A}) and (\ref{B}) clearly imply that $A+B=\frac{1}{P_{2k+1}}$ also, so $C-E=A+B$, which is equivalent to (iii).

So far we have not used the assumption that $\alpha=s_{2k+1}(\beta)$; however  this assumption will be relevant to the remaining two  statements. We have \[ \alpha=\gamma\frac{\alpha_{2k+1}}{\gamma_{2k+1,\beta}}=\gamma\frac{H_{2k+2}(\beta+1)-(\beta-1)}{2P_{2k+1}}.\]  We see then using (\ref{A}) that \begin{align*}
A+P_{2k+1}\alpha &= \frac{1}{2P_{2k+1}}\left(1+H_{2k+2}\frac{\beta-1}{\beta+1}\right)+\frac{\gamma}{2}\left(H_{2k+2}(\beta+1)-(\beta-1)\right)
\\ &= \frac{1}{2P_{2k+1}}\left(1+H_{2k+2}\frac{\beta-1}{\beta+1}+H_{2k+2}^{2}-H_{2k+2}\frac{\beta-1}{\beta+1}\right)=\frac{H_{2k+2}^{2}+1}{2P_{2k+1}}=P_{2k+3}
\end{align*} where the last equation uses (\ref{htop}). This proves (iv) since we have already seen that $F=P_{2k+3}-P_{2k+1}\alpha$.

It remains to prove (v). It is obvious from the definitions that $A\geq B$ and that $E$ and $F$ have the same sign.  So since $B\geq 0$ by (vi) and  $A=F$ by (iv) we deduce that also $E\geq 0$.  But then (iii) gives $A=C-B-E\leq C$, as desired.
\end{proof} 

\begin{cor}\label{evencor}
For $\gamma=\gamma_{2k,\beta},\alpha=s_{2k+1}(\beta),1\leq \beta\leq b_{2k+1}$, we have $2\gamma(\beta+1)-\alpha-1\geq 0$, and the class $\Sigma_{\alpha,\beta,\gamma}^{k}$ belongs to $\bar{\mathcal{C}}_K(X_N)$ for appropriate $N$.
\end{cor}
\begin{proof}
We noted earlier that $\alpha<\alpha_{2k+1}$ (as a consequence of Proposition \ref{ordergamma}), so since $\alpha_{2k+1}=\frac{P_{2k+3}}{P_{2k+1}}$ Proposition \ref{evenids} (iv) shows that $F>0$. But $F$ has the same sign as $2\gamma(\beta+1)-\alpha-1$, proving the first statement of the corollary.  

Proposition \ref{evenids} (ii),(iv), and (v) together show that $D=4C-A>3C$.  (Note also that $C,D$ are nonnegative since, as noted earlier, $\alpha\in \left[\frac{P_{2k+1}}{P_{2k-1}},\frac{P_{2k+2}}{P_{2k}}\right]$.)  Thus $\Sigma_{\alpha,\beta,\gamma}^{k}$ can be rewritten (also using Proposition \ref{evenids} (i)) as \[ \langle 2C;A,B,C^{\times 3},\mathcal{W}(C,D-3C),\mathcal{W}(E,F)\rangle.\]  We will see that this class satisfies the tiling criterion of Corollary \ref{tilecrit}.   Note that obviously $A\geq B$, and $B\geq 0$ by Proposition \ref{evenids} (vi); we have also already noted that $C$, $D-3C$, and $F$ are positive, and so $E$ is also nonnegative since $E$ is a nonnegative multiple of $F$.

We must show that a square of sidelength $2C$ contains, disjointly, the interiors of $3$ squares of sidelength $C$, squares of sidelengths $A$ and $B$, a $C$-by-$(D-3C)$ rectangle, and a $E$-by-$F$ rectangle.  We can place the $3$ squares of sidelengths $C$ in three of the four quadrants of the square of sidelength $2C$, so it suffices to show that the remaining quadrant (also a square of sidelength $C$) can be tiled by squares of sidelengths $A$ and $B$ together with a $C$-by-$(D-3C)$ rectangle and a $E$-by-$F$ rectangle.  

Proposition \ref{evenids} (ii) and (iv) shows that $A+(D-3C)=C$. So by placing the $C$-by-$(D-3C)$ rectangle along one side of the remaining quadrant we see that it suffices to tile an $A$-by-$C$ rectangle by a square of sidelength $A$, a square of sidelength $B$, and a $E$-by-$F$ rectangle.  Since (using various statements in Proposition \ref{evenids}) $F=A\leq C$, $A+B+E=C$, and $A,B,E\geq 0$ (and hence $B\leq C$), this is straightforward to do by simply stacking the two squares and the rectangle on top of each other.  Thus $\Sigma_{\alpha,\beta,\gamma}^{k}$ satisfies the tiling criterion, so belongs to the appropriate $\bar{\mathcal{C}}_K(X_N)$ by Corollary \ref{tilecrit}.  See Figure \ref{evenfig}.
\end{proof}

The fact that $(\gamma_{2k,\beta}\beta,\gamma_{2k,\beta};\mathcal{W}(1,s_{2k+1}(\beta)))\in \bar{\mathcal{C}}_K(X_N)$  for $\beta\leq b_{2k+1}$, or equivalently that there is a symplectic embedding $E(1,s_{2k+1}(\beta))^{\circ}\hookrightarrow \gamma_{2k,\beta}P(1,\beta)$, now follows directly from Proposition \ref{bigreduce} (or Remark \ref{k0} in the case that $k=0$).

\begin{center}
\begin{figure}
\includegraphics[height=2 in]{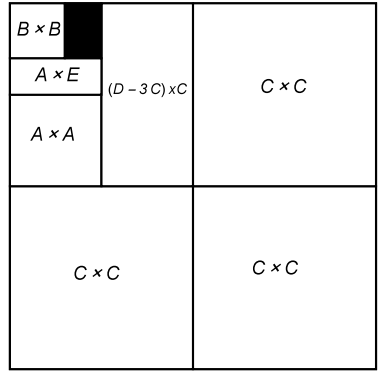}
\caption{The tiling used to prove Corollary \ref{evencor}.}
\label{evenfig}
\end{figure}
\end{center}

\subsubsection{The case $\gamma=\gamma_{2k-1,\beta}$}
We now turn to the various embeddings $E(1,\alpha)^{\circ}\hookrightarrow \gamma P(1,\beta)$ that we require when \[ \gamma=\gamma_{2k-1,\beta}=\frac{2P_{2k+1}}{H_{2k}(\beta+1)-(\beta-1)}.\]  Continue to denote by $Z,A,B,C,D,E,F$ the functions of $k,\alpha,\beta,\gamma$ given by the formulas starting with (\ref{zadef}).  

\begin{lemma}\label{genodd}
If $\gamma=\gamma_{2k-1,\beta}$, then for any $k,\alpha,\beta$ we have:
\begin{itemize} \item[(i)] $A=\gamma(\beta-1)$,
\item[(ii)] $B=0$, and
\item[(iii)] $A+C+E=Z$.
\end{itemize}
\end{lemma}

\begin{proof}
Indeed, we find that \[ \frac{H_{2k}}{2}\gamma(\beta+1)-P_{2k+1}=\frac{(2P_{2k+1}H_{2k}-2P_{2k+1}H_{2k})(\beta+1)+2P_{2k+1}(\beta-1)    }{2H_{2k}(\beta+1)-2(\beta-1)}=\gamma\frac{\beta-1}{2},\] which immediately implies (i) and (ii) by the definitions of $A$ and $B$.  

Also, \begin{align*} A+C+E &= \gamma(\beta-1)+\left(\frac{P_{2k+2}}{2}-\frac{P_{2k}}{2}\alpha\right)+\frac{P_{2k}}{2}(2\gamma(\beta+1)-\alpha-1)
\\ &= \gamma(\beta-1)+P_{2k}\gamma(\beta+1)-P_{2k}\alpha+\frac{P_{2k+2}-P_{2k}}{2}.\end{align*} Now we can rewrite $P_{2k}$ as $P_{2k+1}-H_{2k}$, and $\frac{P_{2k+2}-P_{2k}}{2}$ as $P_{2k+1}$, giving \begin{align*} A+C+E&=P_{2k+1}(\gamma(\beta+1)-1)-P_{2k}\alpha+2P_{2k+1}+\gamma(\beta-1)-H_{2k}\gamma(\beta+1)
\\ &=Z+2P_{2k+1}-\gamma(H_{2k}(\beta+1)-(\beta-1))=Z,\end{align*} proving (iii).
\end{proof}

\begin{prop}\label{efpos}
If $\gamma=\gamma_{2k-1,\beta}$ and $\alpha,\beta\geq 1$, then $2\gamma(\beta+1)-\alpha-1\geq 0$ provided that one of the following holds:
\begin{itemize} \item[(i)] $\alpha\leq \frac{H_{2k+2}}{H_{2k}}$, or
\item[(ii)] $\alpha=s'_{2k}(\beta)$, or 
\item[(iii)] $\alpha=\frac{P_{2k+2}}{P_{2k}}$ and $\beta+1\leq H_{2k+1}(\beta-1)$, or
\item[(iv)] $\alpha=\frac{2P_{2k+1}}{H_{2k+1}}\frac{P_{2k+2}(\beta+1)-(\beta-1)}{H_{2k}(\beta+1)-(\beta-1)}$ and $\beta+1\leq H_{2k+2}(\beta-1)$.\end{itemize}
\end{prop}

\begin{proof}
Using (\ref{4hp}), we have
\begin{equation}\label{gen2g} 2\gamma(\beta+1)-1=\frac{(4P_{2k+1}-H_{2k})(\beta+1)+(\beta-1)}{H_{2k}(\beta+1)-(\beta-1)}=\frac{H_{2k+2}(\beta+1)+(\beta-1)}{H_{2k}(\beta+1)-(\beta-1)}.
\end{equation} This is obviously greater than $\frac{H_{2k+2}}{H_{2k}}$, proving (i). 
As for (ii), we have \[ s'_{2k}(\beta)=\gamma\frac{\alpha_{2k+1}}{\gamma_{2k+1,\beta}}=\gamma\frac{H_{2k+2}(\beta+1)-(\beta-1)}{2P_{2k+1}},\] so if $\alpha=s'_{2k}(\beta)$ then (\ref{4hp}) gives \[ 2\gamma(\beta+1)-\alpha=\gamma\frac{(4P_{2k+1}-H_{2k+2})(\beta+1)+(\beta-1) }{2P_{2k+1}}=\gamma\frac{H_{2k}(\beta+1)+(\beta-1)}{H_{2k}(\beta+1)-(\beta-1)}\geq 1,\] which suffices to prove (ii).

Next, we see from (\ref{2hp}) and (\ref{consec}) that \begin{align*} 2\gamma(\beta+1)-\frac{P_{2k+2}}{P_{2k}}-1 &=2\gamma(\beta+1)-\frac{P_{2k+2}+P_{2k}}{P_{2k}} =2\gamma(\beta+1)-\frac{2H_{2k+1}}{P_{2k}} 
\\ &=\frac{2\left(2P_{2k}P_{2k+1}(\beta+1)-H_{2k+1}H_{2k}(\beta+1)+H_{2k+1}(\beta-1)\right)}{P_{2k}(H_{2k}(\beta+1)-(\beta-1))}
\\ &=\frac{2}{P_{2k}(H_{2k}(\beta+1)-(\beta-1))}\left(-(\beta+1)+H_{2k+1}(\beta-1)\right),\end{align*} which immediately implies (iii).

Finally let $\alpha=\frac{2P_{2k+1}}{H_{2k+1}}\frac{P_{2k+2}(\beta+1)-(\beta-1)}{H_{2k}(\beta+1)-(\beta-1)}$.  By (\ref{gen2g}), we have \begin{align*} 2\gamma(\beta+1)&-\alpha-1=\frac{H_{2k+1}H_{2k+2}(\beta+1)+H_{2k+1}(\beta-1)-2P_{2k+1}P_{2k+2}(\beta+1)+2P_{2k+1}(\beta-1)    }{H_{2k+1}(H_{2k}(\beta+1)-(\beta-1))}
\\ &=\frac{(H_{2k+1}H_{2k+2}-2P_{2k+1}P_{2k+2})(\beta+1)+(H_{2k+1}+2P_{2k+1})(\beta-1)}{H_{2k+1}(H_{2k}(\beta+1)-(\beta-1))} \\&= \frac{-(\beta+1)+H_{2k+2}(\beta-1)}{H_{2k+1}(H_{2k}(\beta+1)-(\beta-1))},\end{align*} where the last equation follows from (\ref{consec}) and (\ref{hpadd}). Thus (iv) holds.
\end{proof}

\begin{prop} \label{typodd}
If $1\leq \beta\leq b_{2k+1}$, $\gamma=\gamma_{2k-1,\beta}$, and $\alpha=s_{2k}(\beta)$ with $\mathcal{W}(1,\alpha)$ having length $N-1$ then $(\gamma \beta,\gamma;\mathcal{W}(1,\alpha))\in\bar{\mathcal{C}}_K(X_N)$.
\end{prop}

\begin{proof} Corollary \ref{firstpiece} shows (under the assumption that $\beta\leq b_{2k+1}$) that $\alpha_{2k-1}\leq s_{2k}(\beta)\leq \alpha_{2k}$ for any $k$, \emph{i.e.} that $s_{2k}(\beta)\in \left[\frac{P_{2k+1}}{P_{2k-1}},\frac{H_{2k+2}}{H_{2k}}\right]$.  In particular Proposition \ref{bigreduce} applies, with the same value of $k$, when $\alpha=s_{2k}(\beta)$, and Case (i) of Proposition \ref{efpos} shows that $0\leq 2\gamma(\beta+1)-\frac{H_{2k+2}}{H_{2k}}-1\leq 2\gamma(\beta+1)-\alpha-1$.  So the parameters $E=\frac{P_{2k}}{2}(2\gamma(\beta+1)-\alpha-1)$ and $F=P_{2k+1}(2\gamma(\beta+1)-\alpha-1)$ are both nonnegative, and (using Proposition \ref{bigreduce} and Lemma  \ref{genodd}) the class $(\gamma \beta,\gamma;\mathcal{W}(1,\alpha))$ is Cremona equivalent to \begin{equation}\label{typoddeqn} \langle Z;A,0,\mathcal{W}(C,D),\mathcal{W}(E,F)\rangle,\end{equation} where as before $Z=P_{2k+1}(\gamma(\beta+1)-1)-P_{2k}\alpha$, $A=\gamma(\beta-1)$, $C=\frac{P_{2k+2}}{2}-\frac{P_{2k}}{2}\alpha$ and $D=P_{2k-1}\alpha-P_{2k+1}$, and where moreover $Z=A+C+E$.

Now by definition $\alpha=\frac{\alpha_{2k}}{\gamma_{2k,\beta}}\gamma=\frac{P_{2k+1}(\beta+1)\gamma}{H_{2k}}$.  So by (\ref{2hp}) \[ F-D=2P_{2k+1}\gamma(\beta+1)-(P_{2k+1}+P_{2k-1})\alpha=2P_{2k+1}\gamma(\beta+1)-2H_{2k}\alpha=0.\]  

Also \[ Z-F=-P_{2k+1}\gamma(\beta+1)+(P_{2k+1}-P_{2k})\alpha=-P_{2k+1}\gamma(\beta+1)+H_{2k}\alpha=0 .\] So (\ref{typoddeqn}) can be rewritten as\[ \langle F;A,0,\mathcal{W}(C,F),\mathcal{W}(E,F)\rangle.\]
We have already noted that $E,F\geq 0$, and clearly $A=\gamma(\beta-1)\geq 0$.  Also $C>0$ since $\alpha\leq \frac{H_{2k+2}}{H_{2k}}<\frac{P_{2k+2}}{P_{2k}}$.  Moreover by Lemma \ref{genodd} $F=A+C+E$ (which of course in particular implies that $A\leq F$).  Consequently the class satisfies the tiling criterion (see Figure \ref{tile2}), and so our original class $(\gamma \beta,\gamma;\mathcal{W}(1,\alpha))$ belongs to $\bar{\mathcal{C}}_K(X_N)$.
\end{proof}

\begin{center}
\begin{figure}
\includegraphics[height=2 in]{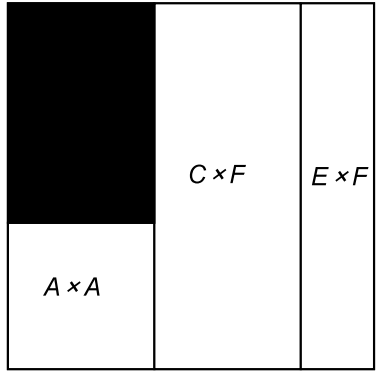}
\caption{The tiling used to prove Proposition \ref{typodd}.}
\label{tile2}
\end{figure}
\end{center}

\begin{prop} \label{excodd}
If $b_{2k+1}\leq  \beta\leq b_{2k}$, $\gamma=\gamma_{2k-1,\beta}$, and $\alpha=s'_{2k}(\beta)$ with $\mathcal{W}(1,\alpha)$ having length $N-1$  then $(\gamma \beta,\gamma;\mathcal{W}(1,\alpha))\in \bar{\mathcal{C}}_K(X_N)$.
\end{prop}

\begin{proof}
By (\ref{sprimea}), we have $\alpha_{2k-1}<\alpha\leq a_{2k+1}$, \emph{i.e.} $\frac{P_{2k+1}}{P_{2k-1}}<\alpha\leq \frac{P_{2k+3}}{P_{2k+1}}$.  So (since $\frac{P_{2k+3}}{P_{2k+1}}<\sigma^2<\frac{P_{2k+2}}{P_{2k}}$) Proposition \ref{bigreduce} applies, with the same value of $k$, when $\alpha=s'_{2k}(b)$; moreover Proposition \ref{efpos} (ii) shows that we have $2\gamma(\beta+1)-\alpha-1\geq 0$.  So our class $(\gamma \beta,\gamma;\mathcal{W}(1,\alpha))$ is Cremona equivalent to $\langle Z;A,0,\mathcal{W}(C,D),\mathcal{W}(E,F)\rangle$ where $Z,A,C,D,E,F$ are defined by the usual formulas starting with (\ref{zadef}).

Note that the definitions yield \[ s'_{2k}(\beta)=\frac{H_{2k+2}(\beta+1)-(\beta-1)}{H_{2k}(\beta+1)-(\beta-1)}.\]  In particular we have $\alpha=s'_{2k}(\beta)\geq \frac{H_{2k+2}}{H_{2k}}$.  Consequently \[ D-2C=(P_{2k-1}\alpha-P_{2k+1})-(P_{2k+2}-P_{2k}\alpha)=H_{2k}\alpha-H_{2k+2}\geq 0 \] and \[ \mathcal{W}(C,D)=(C^{\times 2})\sqcup \mathcal{W}(C,D-2C).\]  Also, 
\begin{align*} 
F &= P_{2k+1}(2\gamma(\beta+1)-\alpha-1)=P_{2k+1}\frac{(4P_{2k+1}-H_{2k+2}-H_{2k})(\beta+1)+2(\beta-1)   }{H_{2k}(\beta+1)-(\beta-1)}\\ &= \frac{2P_{2k+1}(\beta-1)}{H_{2k}(\beta+1)-(\beta-1)}=\gamma(\beta-1)=A 
\end{align*} where the third equality uses (\ref{4hp}) and the last uses Lemma \ref{genodd} (i).  Moreover, \begin{align*} Z+C+E&=(P_{2k+1}+P_{2k})\gamma(\beta+1)+\left(-P_{2k+1}+\frac{P_{2k+2}}{2}-\frac{P_{2k}}{2}\right)-2P_{2k}\alpha
\\ &= H_{2k+1}\gamma(\beta+1)-2P_{2k}\alpha,\end{align*} and so \begin{equation}\label{zced} Z+C+E-D  = H_{2k+1}\gamma(\beta+1)-(2P_{2k}+P_{2k-1})\alpha+P_{2k+1}=H_{2k+1}\gamma(\beta+1)-P_{2k+1}(\alpha-1). \end{equation} Now \[ \alpha-1=\frac{(H_{2k+2}(\beta+1)-(\beta-1))-(H_{2k}(\beta+1)-(\beta-1))}{H_{2k}(\beta+1)-(\beta-1)}=\frac{2H_{2k+1}(\beta+1)}{H_{2k}(\beta+1)-(\beta-1)}.\]  Thus (\ref{zced}) shows that \[ Z+C+E-D=\frac{2H_{2k+1}P_{2k+1}(\beta+1)-2P_{2k+1}H_{2k+1}(\beta+1)}{H_{2k}(\beta+1)-(\beta-1)}=0.\] So $D-2C=Z-C+E$.

Thus the class $\langle Z;A,0,\mathcal{W}(C,D),\mathcal{W}(E,F)\rangle$ can be rewritten as \[ \langle Z;A,0,C^{\times 2},\mathcal{W}(C,Z-C+E),\mathcal{W}(E,A)\rangle.\]  Applying the Cremona move $\frak{c}_{023}$ and recalling from Lemma \ref{genodd} that $Z=A+C+E$ and hence $Z-A-2C=E-C$, we conclude that $(\gamma \beta,\gamma;\mathcal{W}(1,\alpha))$ is Cremona equivalent to \begin{equation} \label{lastpre} \langle Z-C+E;A-C+E,0,E^{\times 2},\mathcal{W}(C,Z-C+E),\mathcal{W}(E,A)\rangle.\end{equation}  Now rearranging the equation $Z=A+C+E$ shows that $A+2E=Z-C+E$, and also that \[ (A-C+E)+E+C=(A+C+E)-C+E=Z-C+E.\]  So by combining the $E^{\times 2}$ with the $\mathcal{W}(E,A)$ into a $E\times (Z-C+E)$ rectangle and placing this in between a $C\times (Z-C+E)$ rectangle and a square of sidelength $A-C+E$ as in Figure \ref{tile3fig} shows that (\ref{lastpre}) satisfies the tiling criterion, provided that it holds that $A-C+E\geq 0$.

\begin{center}
\begin{figure}
\includegraphics[height=2 in]{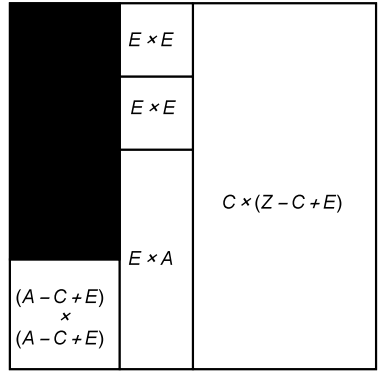}
\caption{The tiling used to prove Proposition \ref{excodd}.}
\label{tile3fig}
\end{figure}
\end{center}

Since $A+C+E=Z$, we have \begin{align*} 
A-C+E &=Z-2C = P_{2k+1}(\gamma(\beta+1)-1)-P_{2k+2}=P_{2k+1}\gamma(\beta+1)-H_{2k+2}
\\ &= \frac{(2P_{2k+1}^{2}-H_{2k}H_{2k+2})(\beta+1)+H_{2k+2}(\beta-1)}{H_{2k}(\beta+1)-(\beta-1)}=\frac{-(\beta+1)+H_{2k+2}(\beta-1)}{H_{2k}(\beta+1)-(\beta-1)} \end{align*} where the last equation follows from (\ref{hh}) and (\ref{phloc}).  Our assumption that $\beta\geq b_{2k+1}$ is equivalent to the statement that $\beta+1\leq H_{2k+2}(\beta-1)$, and so we indeed have $A-C+E\geq 0$.  So Figure \ref{tile3fig} shows that (\ref{lastpre}) satisfies the tiling criterion and hence that the class $(\gamma \beta,\gamma;\mathcal{W}(1,\alpha))$ (to which it is Cremona equivalent) belongs to $\bar{\mathcal{C}}_K(X_N)$.
\end{proof}

\begin{prop}\label{p2k}
Assume that $P_{2k}(\beta-1)\leq \beta+1\leq H_{2k+1}(\beta-1)$ where $k\geq 1$, and let $\gamma=\gamma_{2k-1,\beta}$ and $\alpha=\frac{P_{2k+2}}{P_{2k}}$.  Then $(\gamma \beta,\gamma;\mathcal{W}(1,\alpha))\in\bar{\mathcal{C}}_K(X_N)$ where $N-1$ is the length of $\mathcal{W}(1,\alpha)$. 
\end{prop}

\begin{proof}
By Proposition \ref{bigreduce} the statement is equivalent to the statement that $\langle Z;A,B,\mathcal{W}(C,D),\mathcal{W}(E,F)\rangle$  belongs to the appropriate $\bar{\mathcal{C}}_K(X_N)$ provided that $E,F\geq 0$, and Proposition \ref{efpos} (iii) shows that we indeed have $E,F\geq 0$ due to the assumption that $\beta+1\leq H_{2k+1}(\beta-1)$. Moreover the hypothesis that $\alpha=\frac{P_{2k+2}}{P_{2k}}$ means that $C=0$, so that $\mathcal{W}(C,D)$ is the empty sequence.  Also Lemma \ref{genodd} shows that $B=0$ and that (since $C=0$) $Z=A+E$, so we just need to consider the class \begin{equation}\label{p2kclass} \langle A+E;A,\mathcal{W}(E,F)\rangle \end{equation} where \[ A=\gamma(\beta-1), \] \[ E=\frac{P_{2k}}{2}\left(2\gamma(\beta+1)-\frac{P_{2k+2}}{P_{2k}}-1\right)=P_{2k}\gamma(\beta+1)-\frac{P_{2k+2}+P_{2k}}{2}=P_{2k}\gamma(\beta+1)-H_{2k+1} \] (using (\ref{2hp})), and \[ F=\frac{2P_{2k+1}}{P_{2k}}E.\] 

Now \[ F-4E=\frac{2P_{2k+1}-4P_{2k}}{P_{2k}}E=\frac{2P_{2k-1}}{P_{2k}}{E} \geq 0,\] so (\ref{p2kclass}) is equal to \begin{equation}\label{p2kclass2} \langle A+E;A,E^{\times 4},\mathcal{W}(E,F-4E)\rangle,\end{equation} where $F-4E=\frac{2P_{2k-1}}{P_{2k}}{E}=\left(1-\frac{P_{2k-2}}{P_{2k}}\right)E\leq E$.  Also \begin{align*} A-2E &= \gamma(\beta-1)-2P_{2k}\gamma(\beta+1)+2H_{2k+1}\\ &=\frac{2P_{2k+1}(\beta-1)-4P_{2k+1}P_{2k}(\beta+1)+2H_{2k+1}H_{2k}(\beta+1)-2H_{2k+1}(\beta-1)   }{H_{2k}(\beta+1)-(\beta-1)}
\\ &= \frac{2(\beta+1)-(2H_{2k+1}-2P_{2k+1})(\beta-1)}{H_{2k}(\beta+1)-(\beta-1)}=\frac{2\left((\beta+1)-P_{2k}(\beta-1)\right)}{H_{2k}(\beta+1)-(\beta-1)}\geq 0 \end{align*} by our assumption on $\beta$. So $A\geq 2E$.

The facts that $A\geq 2E\geq 0$ and that $0\leq F-4E\leq E$ imply that the class (\ref{p2kclass}) satisfies the tiling criterion, as can be seen for instance from Figure \ref{tile4fig}.  Since $(\gamma \beta,\gamma;\mathcal{W}(1,\alpha))$ is Cremona equivalent to (\ref{p2kclass2}) this proves the proposition.
\end{proof}

\begin{center}
\begin{figure}
\includegraphics[height=2 in]{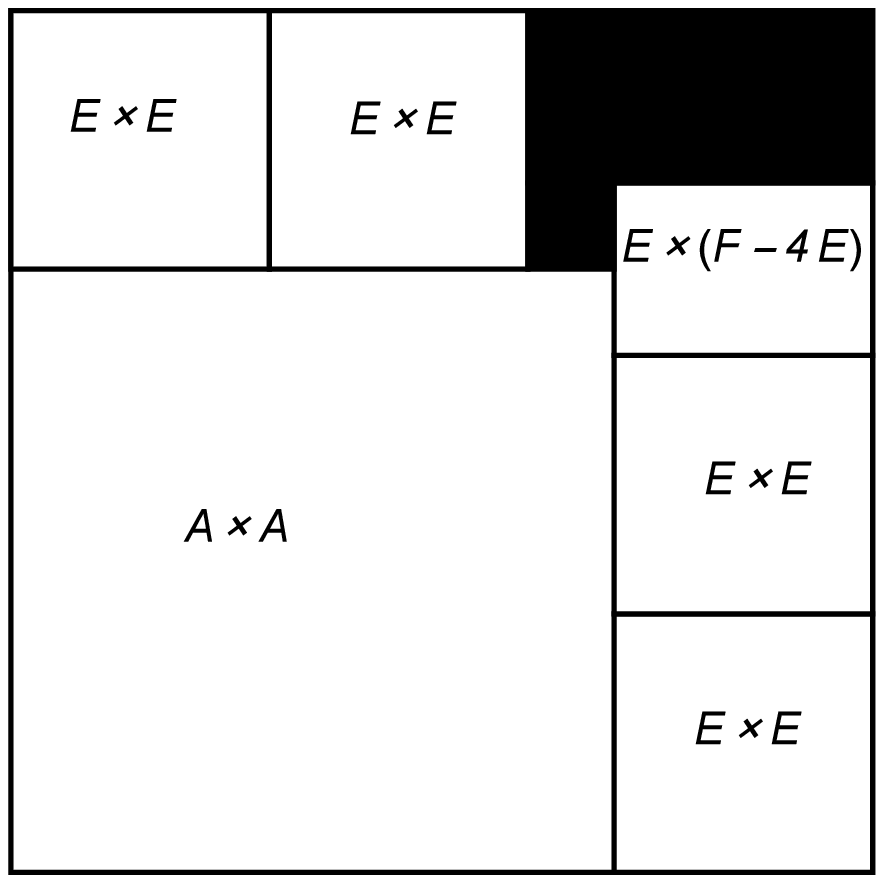}
\caption{The tiling used to prove Proposition \ref{p2k}.}
\label{tile4fig}
\end{figure}
\end{center}

\begin{prop}\label{lastplat} Assume that $H_{2k+1}(\beta-1)\leq \beta+1\leq P_{2k+2}(\beta-1)$ and let $\alpha=\frac{2P_{2k+1}}{H_{2k+1}}\frac{P_{2k+2}(\beta+1)-(\beta-1)}{H_{2k}(\beta+1)-(\beta-1)}$.  Then $\sigma^2<\alpha <\frac{P_{2k+2}}{P_{2k}}$. Furthermore, if $\gamma=\gamma_{2k-1,\beta}$ and $\mathcal{W}(1,\alpha)$ has length $N-1$ then $(\gamma \beta,\gamma;\mathcal{W}(1,\alpha))\in \bar{\mathcal{C}}_K(X_N)$.
\end{prop}

\begin{proof}
Note that \[ \frac{P_{2k+2}(\beta+1)-(\beta-1)}{H_{2k}(\beta+1)-(\beta-1)}=1+\frac{P_{2k+2}-H_{2k}}{H_{2k}-\frac{\beta-1}{\beta+1}} \] is an increasing function of $\beta$, so to prove the first sentence it suffices to check that $\alpha\in \left(\sigma^2,\frac{P_{2k+2}}{P_{2k}}\right)$ just when $\beta$ is one of the endpoints of the interval of possible values of $\beta$ under consideration, \emph{i.e.} when $\beta+1=P_{2k+2}(\beta-1)$ or when $\beta+1=H_{2k+1}(\beta-1)$.  If $\beta+1=P_{2k+2}(\beta-1)$ (\emph{i.e.}, $\beta=1+\frac{2}{P_{2k+2}-1}$) then, using (\ref{phloc}), \[ \alpha=\frac{2P_{2k+1}}{H_{2k+1}}\frac{P_{2k+2}^{2}-1}{P_{2k+2}H_{2k}-1}=\frac{2P_{2k+1}}{H_{2k+1}}\frac{P_{2k+2}^{2}-1}{P_{2k+1}H_{2k+1}}=\frac{2(P_{2k+2}^{2}-1)}{H_{2k+1}^{2}},\] which is greater than $\sigma^2$ by Proposition \ref{28649}.  On the other hand if $\beta+1=H_{2k+1}(\beta-1)$ (\emph{i.e.}, $\beta=1+\frac{2}{H_{2k+1}-1}$) then, using (\ref{consec}), \[ \alpha=\frac{2P_{2k+1}}{H_{2k+1}}\frac{P_{2k+2}H_{2k+1}-1}{H_{2k}H_{2k+1}-1}=\frac{2P_{2k+1}(P_{2k+2}H_{2k+1}-1)}{H_{2k+1}(2P_{2k}P_{2k+1})}=\frac{P_{2k+2}}{P_{2k}}-\frac{1}{P_{2k}H_{2k+1}}.\]  So since $\alpha$ is an increasing function of $\beta$, for any $\beta\in \left[1+\frac{2}{P_{2k+2}-1},1+\frac{2}{H_{2k+1}-1}\right]$ we will have inequalities \[ \sigma^2<\frac{2(P_{2k+2}^{2}-1)}{H_{2k+1}^{2}}\leq \alpha\leq \frac{P_{2k+2}}{P_{2k}}-\frac{1}{P_{2k}H_{2k+1}} <\frac{P_{2k+2}}{P_{2k}},\] proving the first sentence of the proposition.

Now since $H_{2k+2}>P_{2k+2}$, Proposition \ref{efpos} shows that $2\gamma(\beta+1)-\alpha-1\geq 0$.  By what we have already shown, we have \[ \frac{P_{2k+1}}{P_{2k-1}}<\frac{P_{2k+3}}{P_{2k+1}}<\sigma^2<\alpha<\frac{P_{2k+2}}{P_{2k}}.\]  So Proposition \ref{bigreduce} applies with the same value of $k$, so that $(\gamma \beta,\gamma;\mathcal{W}(1,\alpha))$ is Cremona equivalent to $\langle Z;A,0,\mathcal{W}(C,D),\mathcal{W}(E,F)\rangle$ where $Z,A,C,D,E,F$ have their usual meanings (and we have used Proposition \ref{genodd} to see that $B=0$). 

Now \[ D-4C=\left(P_{2k-1}\alpha-P_{2k+1}\right)-2\left(P_{2k+2}-P_{2k}\alpha\right)=P_{2k+1}\alpha-P_{2k+3},\] which is positive since, as noted above, $\alpha>\frac{P_{2k+3}}{P_{2k+1}}$.  So using Lemma \ref{genodd} (iii) we can rewrite $\langle Z;A,0,\mathcal{W}(C,D),\mathcal{W}(E,F)\rangle$ as \[ \langle A+C+E;A,0,C^{\times 4},\mathcal{W}(C,D-4C),\mathcal{W}(E,F)\rangle.\]  Applying the Cremona moves $\frak{c}_{023}$ and $\frak{c}_{045}$ successively yields $\delta_{023}=\delta_{045}=E-C$, so $(\gamma \beta,\gamma;\mathcal{W}(1,\alpha))$ is Cremona equivalent to \begin{equation} \label{first4} \langle A-C+3E;A-2C+2E,0,E^{\times 4},\mathcal{W}(C,D-4C),\mathcal{W}(E,F)\rangle.\end{equation}  Now we noted earlier that $D-4C=P_{2k+1}\alpha-P_{2k+3}$, while, using Lemma \ref{genodd}, \begin{align}\nonumber A-C+3E&= \frac{P_{2k}}{2}\alpha-\frac{P_{2k+2}}{2}+\frac{3P_{2k}}{2}\left(2\gamma(\beta+1)-\alpha-1\right)+\gamma(\beta-1)
\\ &= \gamma(3P_{2k}(\beta+1)+(\beta-1))-P_{2k}\alpha-\frac{1}{2}(P_{2k+2}+3P_{2k})\nonumber \\ &=\gamma(3P_{2k}(\beta+1)+(\beta-1))-P_{2k}\alpha-(P_{2k+1}+2P_{2k}).\label{ac3e} \end{align}
So \[
(A-C+3E)-(D-4C)= \gamma(3P_{2k}(\beta+1)+(\beta-1))-(P_{2k+1}+P_{2k})\alpha+(P_{2k+3}-P_{2k+1}-2P_{2k}).\]  We have $P_{2k+1}+P_{2k}=H_{2k+1}$, and $P_{2k+3}-P_{2k+1}-2P_{2k}=2P_{2k+2}-2P_{2k}=4P_{2k+1}$.  So we obtain 
\begin{align*} (A-C+3E)&-(D-4C)= \gamma(3P_{2k}(\beta+1)+(\beta-1))-H_{2k+1}\alpha+4P_{2k+1}
\\ &= \frac{(6P_{2k}P_{2k+1}-2P_{2k+1}P_{2k+2}+4H_{2k}P_{2k+1})(\beta+1)+(2P_{2k+1}+2P_{2k+1}-4P_{2k+1})(\beta-1)}{H_{2k}(\beta+1)-(\beta-1)}
\\ &= \frac{2P_{2k+1}(\beta+1)}{H_{2k}(\beta+1)-(\beta-1)}(3P_{2k}-P_{2k+2}+2H_{2k}) =0 
\end{align*} since $P_{2k+2}=2(P_{2k+1}-P_{2k})+3P_{2k}=2H_{2k}+3P_{2k}$.

\begin{center}
\begin{figure}
\includegraphics[height=2.5 in]{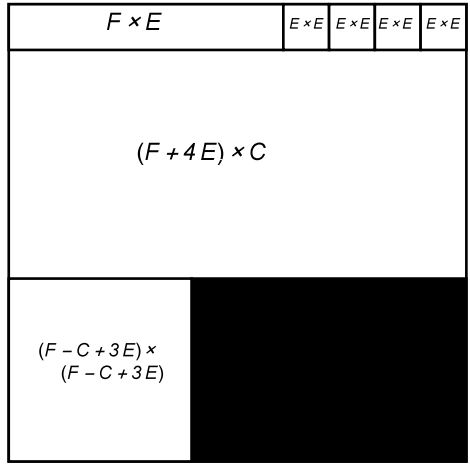}
\caption{The tiling used to prove Proposition \ref{lastplat}.}
\label{tile5fig}
\end{figure}
\end{center}

Furthermore, we have $F+4E=(P_{2k+1}+2P_{2k})(2\gamma(\beta+1)-\alpha-1)$, so by (\ref{ac3e}) we find \begin{align*} (A-C+3E)&-(F+4E)=\gamma\left((3P_{2k}-2P_{2k+1}-4P_{2k})(\beta+1)+(\beta-1)\right)+(P_{2k+1}+P_{2k})\alpha
\\ &= \gamma(-P_{2k+2}(\beta+1)+(\beta-1))+H_{2k+1}\alpha=0 \end{align*} since by definition we have $\alpha=\frac{P_{2k+2}(\beta+1)-(\beta-1)}{H_{2k+1}}\gamma$.  We have thus shown that $D-4C=A-C+3E=F+4E$, in view of which (\ref{first4}) can be rewritten as \[ \langle F+4E;F-C+3E,0,E^{\times 4},\mathcal{W}(C,F+4E),\mathcal{W}(E,F)\rangle.\]  So since we already know that $C,E,F\geq 0$, by joining together four squares of sidelength $E$ with a $F$-by-$E$ rectangle to form a $(F+4E)$-by-$E$ rectangle and then combining this with a $(F+4E)$-by-$C$ rectangle and a square of sidelength $F-C+3E$ as in Figure \ref{tile5fig}, it will follow that (\ref{first4}) satisfies the tiling criterion provided merely that we show that $F-C+3E=A-2C+2E$ is nonnegative. In fact, we find (using (\ref{2hp}) and (\ref{consec})) that \begin{align*} A-2C+2E&=\gamma(\beta-1)-P_{2k+2}+P_{2k}\alpha+P_{2k}(2\gamma(\beta+1)-\alpha-1)
\\ &= \gamma(2P_{2k}(\beta+1)+(\beta-1))-2H_{2k+1} \\ &= \frac{(4P_{2k}P_{2k+1}-2H_{2k}H_{2k+1})(\beta+1)+(2P_{2k+1}+2H_{2k+1})(\beta-1) }{H_{2k}(\beta+1)-(\beta-1)}
\\ &= \frac{-2(\beta+1)+2P_{2k+2}(\beta-1)}{H_{2k}(\beta+1)-(\beta-1)} \geq 0 \end{align*} since we assume that $\beta+1\leq P_{2k+2}(\beta-1)$.  So (\ref{first4}) indeed satisfies the tiling criterion and the proposition follows.
\end{proof}


We have now proven all of the propositions listed near the start of Section \ref{fmsupproof}, and hence have completed the proof of Theorem \ref{fmsup}.

\begin{remark} In the case that $k=0$, Corollary \ref{evencor} and Propositions \ref{typodd}, \ref{excodd}, and \ref{lastplat} can be read off from results in \cite{GU}.  (The $k=0$ case of Proposition \ref{p2k}, on the other hand, is vacuous since $H_{1}=1$ and so the condition $\beta+1\leq H_{2k+1}(\beta-1)$ can never be satisfied.)  Indeed \cite[Proposition 3.8]{GU} can be rephrased as saying that if $m\in \Z_+,\beta\in \R$ with $\beta\geq m$ then $C_{\beta}(m+\beta)\leq 1$. Since $\gamma_{-1,\beta}=1$, $b_1=2$, and $b_0=3$, and since easy calculations show that $s_{0}(\beta)=1+\beta$ and $s'_0(\beta)=2+\beta$, the $k=0$ case of Proposition \ref{typodd} amounts to the statement that $C_{\beta}(1+\beta)\leq  1$ when $1\leq \beta\leq 2$, and that of Proposition \ref{excodd} amounts to the statement that $C_{\beta}(2+\beta)\leq 1$ when $2\leq \beta\leq 3$.  Also the $k=0$ case of Proposition \ref{lastplat} shows that if $\beta\geq 3$ then $C_{\beta}(3+\beta)\leq 1$.  So when $k$ is set equal to zero each of these three propositions becomes a special case of \cite[Proposition 3.8]{GU}.  

Meanwhile, \cite[Proposition 3.10]{GU} (with $y=b=\beta$ and $a=1$) shows, for $1\leq \beta\leq 2$, that for all $\lambda>1$ there is a symplectic embedding $E\left(\frac{1+\beta}{3},2+\beta\right)\hookrightarrow \lambda P(1,\beta)$ and hence that $C_{\beta}\left(\frac{3(\beta+2)}{\beta+1}\right)\leq \frac{3}{\beta+1}$, which is easily seen to be equivalent to the $k=0$ case of Corollary \ref{evencor}.  
\end{remark}

\section{Finding the new staircases}\label{findsect}

\subsection{A criterion for an infinite staircase}\label{critsect}

Recall from the start of Section \ref{intronew} that we say that $C_{\beta}$ has an infinite staircase if there are infinitely many distinct affine functions each of which agrees with $C_{\beta}$ on some nonempty open interval; we have also defined in Section \ref{intronew} the notion of an accumulation point of an infinite staircase.  The new infinite staircases that we find in this paper will be deduced from the existence of certain infinite families of perfect classes $\left(a,b;\mathcal{W}(c,d)\right)$.  (Recall from the introduction that a class of the form $\left(a,b;\mathcal{W}(c,d)\right)$ is said to be perfect if it belongs to the set $\mathcal{E}$ of exceptional classes, and quasi-perfect if it satisfies the weaker condition of having Chern number $1$ and self-intersection $-1$.  Note that distinct perfect classes always have positive intersection number with each other.)   Before constructing these classes, we will derive in this section a general sufficient criterion (see Theorem \ref{infcrit}) for an infinite sequence of perfect classes to guarantee the existence of an infinite staircase for some $C_{\beta}$.

To put the following results into context, we remark that a very similar argument to those used in  \cite[Lemma 2.1.5]{MS} and \cite[Lemma 4.11]{FM} shows that, if $C=(a,b;\mathcal{W}(c,d))$ is a perfect class, then $C$ is the unique class in $\mathcal{E}$ with $C_{\frac{a}{b}}(c/d)=\mu_{\frac{c}{d},\frac{a}{b}}(C)$.  Thus every perfect class exerts some influence on the function of two variables $(\alpha,\beta)\mapsto C_{\beta}(\alpha)$.  We will require a somewhat more flexible version of this statement, showing that a perfect class $(a,b;\mathcal{W}(c,d))$ sharply obstructs embeddings of ellipsoids into dilates of $P(1,\beta )$ for all $\beta $ in an explicit neighborhood of $\frac{a}{b}$, not just into dilates of $P(1,\frac{a}{b})$.\footnote{We take $\beta $ to be less than or equal to $\frac{a}{b}$ in what follows because this turns out to be true in all of the examples that we consider later, but a similar argument with slightly different inequalities gives an analogous result when $\beta >\frac{a}{b}$.}

\begin{lemma}\label{aLbineq}
Assume that $C=(a,b;\mathcal{W}(c,d))\in \mathcal{E}$, with $c\geq d$, $\gcd(c,d)=1$, and $a\geq \beta b$ where $\beta > 1$.  Assume moreover that we have inequalities \begin{equation}\label{ineqs} (a-\beta b)^2<2\beta  \quad\mbox{and}\quad (a-b)(a-\beta b)<1+\beta .\end{equation} Then \[ \mu_{\frac{c}{d},\beta }(C) > \max\left\{\sqrt{\frac{c}{2\beta d}}, \sup_{E\in \mathcal{E}\setminus\{C\}}\mu_{\frac{c}{d},\beta }(E)\right\}.\] 
\end{lemma}

\begin{proof}
We first compare $\mu_{\frac{c}{d},\beta }(C)$ to the volume bound $\sqrt{\frac{c}{2\beta d}}$.  Since $\mathcal{W}(c,d)=dw(c/d)$, we have $w(c/d)\cdot W(c,d)=cd/d=c$, and so \[ \mu_{\frac{c}{d},\beta }(C)=\frac{c}{a+\beta b}.\]  This is greater than the volume bound if and only if $\frac{c^2}{(a+\beta b)^2}>\frac{c}{2\beta d}$, \emph{i.e.} if and only if \[ (a+\beta b)^2<2\beta cd=2\beta (2ab+1) \] where the equality $cd=2ab+1$ follows from $C$ having self-intersection $-1$.  Since $(a+\beta b)^2-4\beta ab=(a-\beta b)^2$ this latter inequality is equivalent to $(a-\beta b)^2<2\beta $, as is assumed to hold in the statement of the lemma.

As for the comparison to $\sup_{E\in \mathcal{E}\setminus\{C\}}\mu_{\frac{c}{d},\beta }(E)$, note first that if $E=(x,y;\vec{m})\in\mathcal{E}$ with $y>x$ then $(y,x;\vec{m})$ also lies in $\mathcal{E}$, and \begin{equation}\label{switch} \mu_{\frac{c}{d},\beta }((y,x;\vec{m}))=\frac{w(c/d)\cdot\vec{m}}{y+\beta x}>\frac{w(c/d)\cdot\vec{m}}{x+\beta y}=\mu_{\frac{c}{d},\beta }((x,y;\vec{m})) \end{equation} (since we assume that $\beta >1$).  Thus \[ \sup_{E\in E\setminus\{C\}}\mu_{\frac{c}{d},\beta }(E)=\max\left\{\mu_{\frac{c}{d},\beta }((b,a;\mathcal{W}(c,d)),\sup_{(x,y;\vec{m})\in\mathcal{E}\setminus\{C\},x\geq y}\mu_{\frac{c}{d},\beta }((x,y;\vec{m}))\right\}.\] We have $\mu_{\frac{c}{d},\beta }(C)>\mu_{\frac{c}{d},\beta }((b,a;\mathcal{W}(c,d)))$ as a special case of (\ref{switch}).  

Consider now any $(x,y;\vec{m})\in\mathcal{E}\setminus\{C\}$ with $x\geq y$.  Since $(x,y;\vec{m})$ and $C=(a,b;\mathcal{W}(c,d))$ belong to $\mathcal{E}$ they have nonnegative intersection number, \emph{i.e.} $bx+ay-\mathcal{W}(c,d)\cdot\vec{m}\geq 0$.  So since $w(c/d)=\frac{1}{d}\mathcal{W}(c,d)$, we find \begin{align*} \mu_{\frac{c}{d},\beta }(E)&=\frac{w(c/d)\cdot\vec{m}}{x+\beta y} \leq \frac{bx+ay}{d(x+\beta y)}
\\ &= \frac{b}{d}\frac{1+\frac{a}{b}\frac{y}{x}}{1+\beta \frac{y}{x}} 
\\ & \leq \frac{a+b}{d(1+\beta )}. \end{align*}  Here the final inequality follows from the assumption that $y\leq x$ and the fact that, since we assume $\frac{a}{b}\geq \beta $, the function $t\mapsto \frac{1+\frac{a}{b}t}{1+\beta t}$ is nondecreasing.

So to complete the proof it suffices to show that $\frac{a+b}{d(1+\beta )}<\frac{c}{a+\beta b}$, \emph{i.e.} that \begin{equation}\label{a+b} (a+b)(a+\beta b)<cd(1+\beta ). \end{equation} But since $C$ has self-intersection $-1$ we have $cd=2ab+1$, so (\ref{a+b}) follows immediately from the second inequality in (\ref{ineqs}) and the observation that $(a+b)(a+\beta b)-2(1+\beta )ab=(a-b)(a-\beta b)$.
\end{proof}

\begin{prop}\label{interval} Assume that $C\in\mathcal{E}$ satisfies the hypotheses of Lemma \ref{aLbineq}. Then there is $\delta>0$ such that for all rational $\alpha\geq 1$ with $|\alpha-\frac{c}{d}|<\delta$, the class $C$ is the unique class in $\mathcal{E}$ with \[ C_{\beta }(\alpha)=\mu_{\alpha,\beta }(C).\]
\end{prop}

\begin{proof} 
For any rational $\alpha\geq 1$ we have \begin{equation}\label{obschar} C_{\beta }(\alpha)=\sup\left(\left\{\sqrt{\frac{\alpha}{2\beta }}\right\}\cup \{\mu_{\alpha,\beta }(E)|E\in\mathcal{E}\}\right) \end{equation} (see \cite[Section 2.1]{CFS}).  So Lemma \ref{aLbineq} shows that our class $C$ obeys $C_{\beta}\left(\frac{c}{d}\right)=\mu_{\frac{c}{d},\beta }(C)$, and  more strongly that, for $\alpha=\frac{c}{d}$, there is a strict positive lower bound on the difference between $\mu_{\alpha,\beta }(C)$ and any of the other values in the set over which the supremum is taken on the right hand side of (\ref{obschar}).  It will follow that this property continues to hold for all $\alpha$ sufficiently close to $\frac{c}{d}$ provided that the family of functions $\left\{\left.\alpha\mapsto \mu_{\alpha,\beta }(E)\right|E\in\mathcal{E}\right\}$ is equicontinuous at $\alpha=\frac{c}{d}$ (\emph{i.e.} for any $\ep>0$ there should be $\delta>0$ independent of $E$ such that if $|\alpha-\frac{c}{d}|<\delta$ then $|\mu_{\alpha,\beta }(E)-\mu_{\frac{c}{d},\beta }(E)|<\ep$ for all $E\in\mathcal{E}$). Of course since $\mu_{\alpha,\beta}(E'_i)=0$ for all $\alpha,\beta$ it suffices to restrict attention to $E\in\mathcal{E}\setminus\cup_i\{E'_i\}$.

Now we find, for $E=(x,y;\vec{m})\in\mathcal{E}\setminus \cup_i\{E'_i\}$ (so that $x,y$ are not both zero, and $2xy-\|\vec{m}\|^2=-1$ since $E$ has self-intersection $-1$) and $\alpha_0,\alpha_1\in\mathbb{Q}$, \begin{align*} |\mu_{\alpha_0,\beta }(E)-\mu_{\alpha_1,\beta }(E)|&=\left|\frac{(w(\alpha_0)-w(\alpha_1))\cdot\vec{m}}{x+\beta y}\right|\leq \frac{\|w(\alpha_0)-w(\alpha_1)\|\|\vec{m}\|}{x+\beta y}
\\ &= \|w(\alpha_0)-w(\alpha_1)\|\sqrt{\frac{2xy+1}{(x+\beta y)^2}}\leq \|w(\alpha_0)-w(\alpha_1)\|\sqrt{\frac{(\frac{x}{\sqrt{\beta}}+\sqrt{\beta}y)^2}{(x+\beta y)^2}+1}
\\ & = \|w(\alpha_0)-w(\alpha_1)\|\sqrt{\frac{1}{\beta}+1}.
\end{align*}  So our family is equicontinuous at $\frac{c}{d}$ provided that the weight sequence function $w\co\Q\cap [1,\infty)\to \Q^{\infty}=\cup_n\Q^n$ is continuous at $\frac{c}{d}$ (with respect to the obvious metric on $\Q^{\infty}$ that restricts to each $\Q^n$ as the Euclidean metric).  This latter fact follows easily from \cite[Lemma 2.2.1]{MS}, which shows that, if the length of the vector $w(c/d)$ is $n_{0}$, then for $\alpha$ in a suitable neighborhood of $\frac{c}{d}$, we can write $w(\alpha)$ as $(\vec{\ell}(\alpha), \vec{r}(\alpha))$ where $\vec{\ell}$ is a piecewise linear $\Q^{n_0}$-valued function equal to $w(c/d)$ at $\alpha=c/d$.  Thus within this neighborhood we have a Lipschitz bound $\|\vec{\ell}(\alpha)-w(c/d)\|\leq M|\alpha-c/d|$, and so \begin{align*} \|\vec{r}(\alpha)\|^2&=\alpha-\|\vec{\ell}(\alpha)\|^2 \\ & \leq \alpha-\left(\|w(c/d)\|-M|\alpha-c/d|\right)^{2}\leq (\alpha-c/d)+2M|\alpha-c/d|\sqrt{c/d}.\end{align*}  Thus $|w(\alpha)-w(c/d)|^{2}=\|\vec{\ell}(\alpha)-w(c/d)\|^2+\|\vec{r}(\alpha)\|^2$ converges to zero as $\alpha\to c/d$.  
\end{proof}

Having shown that, for any perfect class $C=(a,b;\mathcal{W}(c,d))$, $\mu_{\cdot,\beta }(C)$ is equal to $C_{\beta} $ in a neighborhood of $\frac{c}{d}$, we now use results from \cite{MS} to identify $\mu_{\cdot,\beta }(C)$ in such a neighborhood with the function $\Gamma_{\cdot,\beta}(C)$ from Proposition \ref{Gammadef}.

\begin{prop}\label{obsshape}
If $C=(a,b;\mathcal{W}(c,d))$ is a perfect class with $c\geq d$ and $\gcd(c,d)=1$, and if $\beta \geq 1$ is arbitrary, then there is $\delta>0$ such that for all $\alpha\in (\frac{c}{d}-\delta,\frac{c}{d}+\delta)$ we have \[ \mu_{\alpha,\beta }(C)=\Gamma_{\alpha,\beta}(C)=\left\{\begin{array}{ll} \frac{d\alpha}{a+\beta b} & \mbox{ if }\alpha\leq \frac{c}{d} \\ \frac{c}{a+\beta b} & \mbox{ if }\alpha\geq \frac{c}{d} \end{array}\right.. \]
\end{prop}

\begin{proof}
Since \[ \mu_{\alpha,\beta }(C)=\frac{w(\alpha)\cdot \mathcal{W}(c,d)}{a+\beta b}=\frac{d}{a+\beta b}w(\alpha)\cdot w(c/d) \] the proposition is equivalent to the statement that, for all $\alpha$ in a neighborhood of $\frac{c}{d}$, we have \begin{equation}\label{wdot} w(\alpha)\cdot w(c/d)=\left\{ \begin{array}{ll} \alpha & \mbox{ if }\alpha\leq\frac{c}{d} \\ \frac{c}{d} & \mbox{ if } \alpha\geq\frac{c}{d}\end{array}\right..\end{equation} But (\ref{wdot}) can be read off directly from  \cite[Lemma 2.2.1 and Corollary 2.2.7]{MS}.  (Note that, as is alluded to at the start of \cite[Section 2.2]{MS}, the number denoted therein as $N$ can be taken either even or odd according to taste, by allowing the possibility of taking the final term $\ell_N$ in the continued fraction expansion of $\frac{c}{d}$ to be $1$.
\end{proof}

\begin{theorem} \label{infcrit}
Suppose that $\beta >1$ has the property that there is an infinite collection $\{(a_i,b_i;\mathcal{W}(c_i,d_i))\}_{i=1}^{\infty}$ of distinct perfect classes all having $a_i\geq \beta b_i$, $c_i\geq d_i$, $\gcd(c_i,d_i)=1$, $(a_i-\beta b_i)^2<2\beta $, and $(a_i-b_i)(a_i-\beta b_i)<1+\beta $.  Then the embedding capacity function $C_{\beta}$ has an infinite staircase.  

Moreover this infinte staircase has an accumulation point $S=\lim_{i\to\infty}\frac{c_i}{d_i}$ which is the unique number greater than $1$ that obeys \[  \frac{(1+S)^2}{S}=\frac{2(1+\beta )^2}{\beta }.\]
\end{theorem}

\begin{proof} Write $C_i=(a_i,b_i;\mathcal{W}(c_i,d_i))$.  Proposition \ref{interval} shows that for each $i$ there is an open interval $I_i$ containing $\frac{c_i}{d_i}$ on which $C_{\beta}$ is identically equal to $\alpha\mapsto \mu_{\alpha,\beta }(C_i)$; moreover the uniqueness statement in that proposition implies that the intervals $I_i$ can be taken pairwise disjoint (so in particular the various $\frac{c_i}{d_i}$ are all distinct). Now the function $C_{\beta}$ is nondecreasing, and Proposition \ref{obsshape} shows that, for each $i$, $C_{\beta}$ is equal to the constant $\frac{c_i}{a_i+\beta b_i}$ on a nonempty open subinterval of $I_i$, but takes a strictly larger value at the right endpoint of $I_i$ than at the left endpoint.  In particular this forces the various $\frac{c_i}{a_i+\beta b_i}$ to all be distinct (more specifically, if $\frac{c_i}{d_i}<\frac{c_j}{d_j}$ then $\frac{c_i}{a_i+\beta b_i}<\frac{c_j}{a_j+\beta b_j}$).  So there are infinitely many distinct real numbers $\frac{c_i}{a_i+\beta b_i}$ such that $C_{\beta}$ is identically equal to $\frac{c_i}{a_i+\beta b_i}$ on some nonempty open interval, which suffices to prove that $C_{\beta}$ has an infinite staircase.

As for the accumulation point, since the open intervals on which $C_{\beta}$ is equal to $\frac{c_i}{a_i+\beta b_i}$  contain $\frac{c_i}{d_i}$ in their respective closures, it is clear that if $\lim_{i\to\infty}\frac{c_i}{d_i}$ exists then this limit is an accumulation point for the infinite staircase.  So it remains only to show that $\frac{c_i}{d_i}\to S$ where $S$ is as described in the statement of the corollary. 

The fact that we have a bound $(a_i-\beta b_i)^2<2\beta $ where $\beta >1$ and the $(a_i,b_i)$ comprise an infinite subset of $\N^2$ implies that the values $|a_i- b_i|$ diverge to infinity, in view of which the bound $(a_i-b_i)(a_i-\beta b_i)<1+\beta $ (and the fact that both factors are nonnegative) forces $a_i-\beta b_i\to 0$.  
Now since $C_i$ has Chern number $1$ and self-intersection $-1$, we have \begin{align*} 
c_i+d_i &= 2(a_i+b_i) \\ c_id_i &= 2a_ib_i+1.  
\end{align*}  Dividing the square of the first equation by the second shows that \[ \frac{(c_i+d_i)^2}{c_id_i}=\frac{4\left((1+\beta )b_i+(a_i-\beta b_i)\right)^2}{2\beta b_i^2+2(a_i-\beta b_i)b_i+1} .\]  Our hypotheses imply that the $b_i$ diverge to $\infty$ while (since $a_i$ is a bounded distance from $\beta b_i$ and $(a_i-\beta b_i)(a_i-b_i)$ is bounded with $\beta \neq 1$) $(a_i-\beta b_i)b_i$ remains bounded.  Hence the limit of the right hand side above is $\frac{2(1+\beta )^2}{\beta }$.  So writing $S_i=\frac{c_i}{d_i}$ we obtain $\lim_{i\to\infty}\frac{(1+S_i)^2}{S_i}=\frac{2(1+\beta )^2}{\beta }$.  But the function $x\mapsto \frac{(1+x)^2}{x}=x+2+\frac{1}{x}$ restricts to $[1,\infty)$ as  proper and strictly increasing, so this forces $S_i\to S$ where $S$ is the unique solution that is greater than $1$ to the equation $\frac{(1+S)^2}{S}=\frac{2(1+\beta )^2}{\beta }$.
\end{proof}

\subsection{The classes $A_{i,n}^{(k)}$}\label{ainksect}

We now begin the construction of explicit families of perfect classes that give rise to infinite staircases.   
For any positive integer $n$  let us define a sequence of elements \begin{equation}\label{vindef} \vec{v}_{i,n}=(a_{i,n},b_{i,n},c_{i,n},d_{i,n})\in \Z^4\end{equation} recursively by \begin{align*} \vec{v}_{0,n}&=(1,0,1,1) \\
\vec{v}_{1,n} &= (n,1,2n+1,1) \\
\vec{v}_{i+2,n} &= 2n\vec{v}_{i+1,n}-\vec{v}_{i,n}.\end{align*}

We then let \[A_{i,n}=\left(a_{i,n},b_{i,n};\mathcal{W}(c_{i,n},d_{i,n})\right).\]  Notice that the identity $c_{i,n}+d_{i,n}=2(a_{i,n}+b_{i,n})$ clearly holds by induction on $i$, so that the $A_{i,n}$ can be expressed in the form (\ref{xdeform}).  The classes $A_{i,n}^{(k)}$ promised in the introduction are then obtained by applying the $k$th-order Brahmagupta move of Definition \ref{bmovedef} to $A_{i,n}$.  (In particular, $A_{i,n}^{(0)}=A_{i,n}$).

Clearly $A_{0,n}=\left(1,0;1\right)=\langle 0;0,-1\rangle\in\mathcal{E}$.  It is not hard to check that $A_{1,n}=(n,1;\mathcal{W}(2n+1,1))\in\mathcal{E}$; indeed $A_{1,n}$ coincides with the class denoted $E_n$ in \cite{CFS}, and this class is shown to belong to $\mathcal{E}$ in \cite[Lemma 3.2]{CFS}. Consequently the following will show that every $A_{i,n}\in\mathcal{E}$.

\begin{lemma}\label{indain}
For $i\geq 2$ and $n\geq 1$, $A_{i,n}$ is Cremona equivalent to $A_{i-2,n}$.
\end{lemma}

\begin{proof}
First we relate $\mathcal{W}(c_{i,n},d_{i,n})$ to $\mathcal{W}(c_{i-2,n},d_{i-2,n})$.  We have $c_{0,n}=1$, $c_{1,n}=2n+1$, $d_{0,n}=d_{1,n}=1$, and $c_{j,n}=2nc_{j-1,n}-c_{j-2,n}$ and $d_{j,n}=2nd_{j-1,n}-d_{j-2,n}$. So we find $c_{2,n}=4n^2+2n-1$ and $d_{2,n}=2n-1$. It is sometimes convenient to allow the index $j$ in $c_{j,n}$ and $d_{j,n}$ to take negative values; the recurrences then give $c_{-1,n}=-1$ and $d_{-1,n}=2n-1$.  From these formulas we see that, for both $j=1$ and $j=2$, we have identities \begin{align}\nonumber c_{j,n}-(2n+2)d_{j,n}&=c_{j-2,n} \\ d_{j,n}-(2n-2)c_{j-2,n}&=d_{j-2,n}. \label{cdred}\end{align}  But then the recurrence relations defining our sequences make clear that if the identities (\ref{cdred}) hold for $j=1,2$ then they continue to hold for all $j$.  

Now it is easy to see that (using that $n\geq 1$) $\{c_{i,n}\}_{i=0}^{\infty}$ and $\{d_{i,n}\}_{i=0}^{\infty}$ are monotone increasing sequences, and so in particular $c_{i,n},d_{i,n}>0$ for all $i\geq 0$.  If $i\geq 2$ we have, using (\ref{cdred}) \begin{align*}
\mathcal{W}(c_{i,n},d_{i,n}) &= \left(d_{i,n}^{\times(2n+2)}\right)\sqcup\mathcal{W}(c_{i-2,n},d_{i,n})
\\ &= \left(d_{i,n}^{\times(2n+2)},c_{i-2,n}^{\times(2n-2)}\right)\sqcup\mathcal{W}(c_{i-2,n},d_{i-2,n}).
\end{align*}

We thus have \begin{align*} A_{i,n}&=\left(a_{i,n},b_{i,n};d_{i,n}^{\times(2n+2)},c_{i-2,n}^{\times(2n-2)},\mathcal{W}(c_{i-2,n},d_{i-2,n})\right) 
\\ &= \left\langle a_{i,n}+b_{i,n}-d_{i,n};a_{i,n}-d_{i,n},b_{i,n}-d_{i,n},d_{i,n}^{\times(2n+1)},c_{i-2,n}^{\times(2n-2)},\mathcal{W}(c_{i-2,n},d_{i-2,n})\right\rangle.\end{align*}  Applying the sequence $\frak{c}_{023}\circ \frak{c}_{045}\circ\cdots\circ \frak{c}_{0,2n,2n+1}$ of $n$ Cremona moves each with $\delta=b_{i,n}-2d_{i,n}$ yields \begin{align*} &\left\langle a_{i,n}+(n+1)b_{i,n}-(2n+1)d_{i,n};a_{i,n}+ nb_{i,n}-(2n+1)d_{i,n},(b_{i,n}-d_{i,n})^{\times (2n+1)},d_{i,n},\right.\\ & \qquad\qquad \qquad \qquad \qquad \qquad \qquad \qquad \left. c_{i-2,n}^{\times(2n-2)},\mathcal{W}(c_{i-2,n},d_{i-2,n})\right\rangle. \end{align*} Now for all $i$ one has $a_{i,n}+(n+1)b_{i,n}=c_{i,n}$ (indeed this clearly holds for $i=0,1$, and therefore it holds for all $i$ since $a_{i,n},b_{i,n},c_{i,n}$ all satisfy the same linear recurrence), and the first equation in (\ref{cdred}) shows that $c_{i,n}-(2n+1)d_{i,n}=c_{i-2,n}+d_{i,n}$.  So the class displayed above can be rewritten as \[ \left\langle c_{i-2,n}+d_{i,n};c_{i-2,n}-(b_{i,n}-d_{i,n}),(b_{i,n}-d_{i,n})^{\times (2n+1)},d_{i,n},c_{i-2,n}^{\times(2n-2)},\mathcal{W}(c_{i-2,n},d_{i-2,n})\right\rangle. \] We can then apply $n-1$ Cremona moves $\frak{c}_{2n+2,2n+3,2n+4},\frak{c}_{2n+2,2n+5,2n+6},\ldots,\frak{c}_{2n+2,4n-1,4n}$, each with $\delta=-c_{i-2,n}$, to obtain \[ \left\langle d_{i,n}-(n-2)c_{i-2,n};c_{i-2,n}-(b_{i,n}-d_{i,n}), (b_{i,n}-d_{i,n})^{\times(2n+1)},d_{i,n}-(n-1)c_{i-2,n},0^{\times(2n-2)},\mathcal{W}(c_{i-2,n},d_{i-2,n})\right\rangle.\]  Now delete the zeros and apply $n$ Cremona moves $\frak{c}_{1,2,2n+2},\frak{c}_{3,4,2n+2},\ldots,\frak{c}_{2n-1,2n,2n+2}$, each with $\delta=c_{i-2,n}-2(b_{i,n}-d_{i,n})$, to obtain \begin{align}\label{ain} & \left\langle (2n+1)d_{i,n}+2c_{i-2,n}-2nb_{i,n};(c_{i-2,n}-(b_{i,n}-d_{i,n}))^{\times(2n+1)},b_{i,n}-d_{i,n},\right.\\ & \qquad \qquad \qquad \qquad \qquad \qquad \left.d_{i,n}+c_{i-2,n}-2n(b_{i,n}-d_{i,n}),\mathcal{W}(c_{i-2,n},d_{i-2,n})\right\rangle.\nonumber\end{align}  Let us simplify some of the terms in (\ref{ain}). First, using the first equation in (\ref{cdred}),\[ (2n+1)d_{i,n}+2c_{i-2,n}-2nb_{i,n}=c_{i-2,n}+c_{i,n}-d_{i,n}-2nb_{i,n}=c_{i-2,n} \] where the fact that \begin{equation}\label{cdnb} c_{i,n}-d_{i,n}-2nb_{i,n}=0\end{equation} follows by a straightforward induction argument.  Also, we find that \begin{equation}\label{bdcb} b_{i,n}-d_{i,n}=c_{i-2,n}-b_{i-2,n};\end{equation} indeed (extending the definition of $b_{j,n}$ using the recurrence, so that $b_{-1,n}=2nb_{0,n}-b_{1,n}=-1$) this is easily seen to hold for $i=1,2$ and hence it holds for all $i$ since $b_{i,n},c_{i,n},d_{i,n}$ all obey the same linear recurrence.  This implies that \[ c_{i-2,n}-(b_{i,n}-d_{i,n})=b_{i-2,n}.\] Moreover \[ d_{i,n}+c_{i-2,n}-2n(b_{i,n}-d_{i,n})=-2nb_{i,n}-d_{i,n}+\left(c_{i-2,n}+(2n+2)d_{i,n}\right)=-2n b_{i,n}-d_{i,n}+c_{i,n}=0\] where we have used (\ref{cdred}) and (\ref{cdnb}).  Thus after deleting a zero the class (\ref{ain}) (which is Cremona equivalent to $A_{i,n}$) can be rewritten as \[ \left\langle c_{i-2,n};b_{i-2,n}^{\times(2n+1)},c_{i-2,n}-b_{i-2,n},\mathcal{W}(c_{i-2,n},d_{i-2,n})\right\rangle \] The sequence of $n$ Cremona moves $\frak{c}_{1,2,2n+1}\circ \frak{c}_{3,4,2n+1}\circ\cdots\circ \frak{c}_{2n-1,2n,2n+1}$, each with $\delta=-b_{i-2,n}$, transforms this to \[ \left\langle c_{i-2,n}-nb_{i-2,n};b_{i-2,n},0^{\times(2n)},c_{i-2,n}-(n+1)b_{i-2,n},d_{i-2,n},\mathcal{W}(c_{i-2,n}-d_{i-2,n},d_{i-2,n})\right\rangle.\]  Deleting the zeros, then applying one last Cremona move $\frak{c}_{012}$ with $\delta=-d_{i-2,n}$, and then deleting another zero gives \begin{align*} & \left\langle c_{i-2,n}-nb_{i-2,n}-d_{i-2,n};b_{i-2,n}-d_{i-2,n},c_{i-2,n}-(n+1)b_{i-2,n}-d_{i-2,n},\mathcal{W}(c_{i-2,n}-d_{i-2,n},d_{i-2,n})\right\rangle \\ & \qquad =\left(c_{i-2,n}-(n+1)b_{i-2,n},b_{i-2,n};d_{i-2,n},\mathcal{W}(c_{i-2,n}-d_{i-2,n},d_{i-2,n})\right) \\ & \qquad =\left(a_{i-2,n},b_{i-2,n};\mathcal{W}(c_{i-2,n},d_{i-2,n})\right) =A_{i-2,n} \end{align*} where the identity $c_{j,n}-(n+1)b_{j,n}=a_{j,n}$ (applied here with $j=i-2$) follows as usual by checking it for $j=0,1$ and using the fact that $a_{j,n},b_{j,n},c_{j,n}$ all satisfy the same recurrence.  So we have found a sequence of Cremona moves that maps $A_{i,n}$ to $A_{i-2,n}$ when $i\geq 2$.
\end{proof}

\begin{cor}\label{aine}
For any $n\geq 1$ and $i,k\geq 0$ the class $A_{i,n}^{(k)}$ belongs to $\mathcal{E}$.
\end{cor}
\begin{proof} This follows immediately from Lemma \ref{indain}, Proposition \ref{pellcrem}, and the fact that $A_{0,n},A_{1,n}\in\mathcal{E}$.\end{proof}

\subsection{Verifying the infinite staircase criterion}\label{versect}

To analyze the obstructions imposed by the classes $A_{i,n}^{(k)}\in\mathcal{E}$ it is useful to first give closed-form expressions for the integers $a_{i,n},b_{i,n},c_{i,n},d_{i,n}$ from (\ref{vindef}).  Throughout this section we will assume that $n\geq 2$. Let  \[ \omega_n=n+\sqrt{n^2-1}.\]  Then $\omega_n$ is the larger solution to the equation $t^2=2nt-1$, the smaller solution being $\omega_{n}^{-1}=n-\sqrt{n^2-1}$.  Thus both $x_i=\omega_{n}^{i}$ and $x_{i}=\omega_{n}^{-i}$ give solutions to the recurrence $x_{i+2}=2nx_{i+1}-x_i$; these solutions are linearly independent since we assume that $n> 1$, and so any solution $\{x_i\}_{i=0}^{\infty}$ to $x_{i+2}=2nx_{i+1}-x_i$ must be a linear combination (with coefficients independent of $i$) of $\omega_{n}^{i}$ and $\omega_{n}^{-i}$.  This in particular applies to each of the sequences $\{a_{i,n}\}_{i=0}^{\infty}$, $\{b_{i,n}\}_{i=0}^{\infty}$, $\{c_{i,n}\}_{i=0}^{\infty}$, and $\{d_{i,n}\}_{i=0}^{\infty}$, and using the initial conditions from (\ref{vindef}) (and the fact that $\omega_{n}-\omega_{n}^{-1}=2\sqrt{n^2-1}$ while $\omega_{n}+\omega_{n}^{-1}=2n$) to evaluate the coefficients shows that: \begin{align}\label{abcd}
a_{i,n} &= \frac{1}{2}(\omega_{n}^{i}+\omega_{n}^{-i}), \\ b_{i,n} &=  \nonumber \frac{1}{2\sqrt{n^2-1}}(\omega_{n}^{i}-\omega_{n}^{-i}),
\\ c_{i,n} &= \nonumber  \frac{1}{2\sqrt{n^2-1}}\left((\omega_n+1)\omega_{n}^{i}-(\omega_{n}^{-1}+1)\omega_{n}^{-i}\right),
\\ d_{i,n} &=\nonumber  \frac{1}{2\sqrt{n^2-1}}\left((1-\omega_{n}^{-1})\omega_{n}^{i}+(\omega_{n}-1)\omega_{n}^{-i}\right).\end{align}

In particular, \begin{equation}\label{n21} \lim_{i\to\infty}\frac{a_{i,n}}{b_{i,n}}=\sqrt{n^2-1}.\end{equation}

Applying the $k$th-order Brahmagupta move, we find that the classes $A_{i,n}^{(k)}$ can be written as $A_{i,n}^{(k)}=(a_{i,n,k},b_{i,n,k};\mathcal{W}(c_{i,n,k},d_{i,n,k}))$ where, as one can verify using the facts that $c_{i,n}+d_{i,n}=2(a_{i,n}+b_{i,n})$ (since $A_{i,n}$ has Chern number $1$) and that $c_{i,n}-d_{i,n}=2nb_{i,n}$ (by (\ref{cdnb})) together with (\ref{hpadd}), the entries $a_{i,n,k},b_{i,n,k},c_{i,n,k},d_{i,n,k}$ are given by \begin{align}\nonumber
a_{i,n,k} &= \frac{1}{2}\left((H_{2k}+1)a_{i,n}+(H_{2k}+2nP_{2k}-1)b_{i,n}\right) \\ \nonumber
b_{i,n,k} &= \frac{1}{2}\left((H_{2k}-1)a_{i,n}+(H_{2k}+2nP_{2k}+1)b_{i,n}\right) \\ \nonumber
c_{i,n,k} &= \frac{1}{2}(P_{2k+2}c_{i,n}-P_{2k}d_{i,n}) \\ \label{ink}
d_{i,n,k} &= \frac{1}{2}(P_{2k}c_{i,n}-P_{2k-2}d_{i,n}).\end{align}

Let $L_{n,k}=\lim_{i\to\infty}\frac{a_{i,n,k}}{b_{i,n,k}}$, so using (\ref{n21}) \begin{equation}\label{Lnkdef} L_{n,k}=\frac{H_{2k}(\sqrt{n^2-1}+1)+2nP_{2k}+(\sqrt{n^2-1}-1)}{H_{2k}(\sqrt{n^2-1}+1)+2nP_{2k}-(\sqrt{n^2-1}-1)}.\end{equation}


We will see that each $A_{i,n}^{(k)}$ satisfies the requirements of Lemma \ref{aLbineq} with $\beta =L_{n,k}$.

\begin{lemma}\label{aLb} For $i,k\geq 0$ and $n\geq 2$ we have \[ a_{i,n,k}-L_{n,k}b_{i,n,k}=\frac{2\omega_{n}^{-i}(H_{2k}+nP_{2k})}{H_{2k}(\sqrt{n^2-1}+1)+2nP_{2k}-(\sqrt{n^2-1}-1)}.\]
\end{lemma}

\begin{proof}
The above formulas for $a_{i,n},b_{i,n},a_{i,n,k},b_{i,n,k}$ yield \begin{align*} a_{i,n,k} =\frac{1}{4\sqrt{n^2-1}}\left((H_{2k}(\sqrt{n^2-1}+1)\right. & +2nP_{2k}+(\sqrt{n^2-1}-1))\omega_{n}^{i}\\ & \left. +(H_{2k}(\sqrt{n^2-1}-1)-2nP_{2k}+(\sqrt{n^2-1}+1))\omega_{n}^{-i}\right) \end{align*} and \begin{align*} b_{i,n,k} = \frac{1}{4\sqrt{n^2-1}}\left((H_{2k}(\sqrt{n^2-1}+1) \right. & +2nP_{2k}-(\sqrt{n^2-1}-1))\omega_{n}^{i}\\ & \left. +(H_{2k}(\sqrt{n^2-1}-1)-2nP_{2k}-(\sqrt{n^2-1}+1))\omega_{n}^{-i}\right).\end{align*}  The coefficient of $\omega_{n}^{i}$ in the formula for $a_{i,n,k}$ is the numerator of (\ref{Lnkdef}) and the coefficient of $\omega_{n}^{i}$ in the formula for $b_{i,n,k}$ is the denominator of (\ref{Lnkdef}) (of course this is not a coincidence since $L_{n,k}=\lim_{i\to\infty}\frac{a_{i,n,k}}{b_{i,n,k}}$), and we obtain \begin{align*} & a_{i,n,k}-L_{n,k}b_{i,n,k} = \frac{\omega_{n}^{-i}}{4\sqrt{n^2-1}}\left((L_{n,k}-1)(2nP_{2k}-H_{2k}(\sqrt{n^2-1}-1))+(L_{n,k}+1)(\sqrt{n^2-1}+1)\right)
\\ &= \frac{\omega_{n}^{-i}}{4\sqrt{n^2-1}} \frac{2(\sqrt{n^2-1}-1)(2nP_{2k}-H_{2k}(\sqrt{n^2-1}-1))+2(H_{2k}(\sqrt{n^2-1}+1)+2nP_{2k})(\sqrt{n^2-1}+1)}{H_{2k}(\sqrt{n^2-1}+1)+2nP_{2k}-(\sqrt{n^2-1}-1)} 
\\ &= \frac{\omega_{n}^{-i}\left(H_{2k}\left((\sqrt{n^2-1}+1)^2-(\sqrt{n^2-1}-1)^2\right)+2nP_{2k}\left((\sqrt{n^2-1}+1)+(\sqrt{n^2-1}-1)\right)\right)}{2\sqrt{n^2-1}\left(H_{2k}(\sqrt{n^2-1}+1)+2nP_{2k}-(\sqrt{n^2-1}-1)\right)}
\\ &= \frac{\omega_{n}^{-i}(2H_{2k}+2nP_{2k})}{H_{2k}(\sqrt{n^2-1}+1)+2nP_{2k}-(\sqrt{n^2-1}-1)}.
\end{align*}
\end{proof}



We can now prove the inequalities in (\ref{aLbineq}) for our classes $A_{i,n}^{(k)}$.

\begin{prop}\label{beatvol} For any $i,k\geq 0$ and $n\geq 2$ we have $(a_{i,n,k}-L_{n,k}b_{i,n,k})^2< L_{n,k}$.\end{prop}

\begin{proof} It is clear from Lemma \ref{aLb} that it suffices to prove this in the case that $i=0$, as the left hand side is a decreasing function of $i$.  We find from Lemma \ref{aLb} that \begin{align*} &
\frac{(a_{0,n,k}-L_{n,k}b_{0,n,k})^2}{L_{n,k}} \\  & \,\, = \frac{(2H_{2k}+2nP_{2k})^2}{(H_{2k}(\sqrt{n^2-1}+1)+2nP_{2k}+(\sqrt{n^2-1}-1))(H_{2k}(\sqrt{n^2-1}+1)+2nP_{2k}-(\sqrt{n^2-1}-1)) }\\ &\quad = 
\frac{(2H_{2k}+2nP_{2k})^2}{(H_{2k}(\sqrt{n^2-1}+1)+2nP_{2k})^{2}-(\sqrt{n^2-1}-1)^2}
\\ &\quad = \frac{(2H_{2k}+2nP_{2k})^{2}}{\left((2H_{2k}+2nP_{2k})+H_{2k}(\sqrt{n^2-1}-1)\right)^2-(\sqrt{n^2-1}-1)^2}<1.
\end{align*}
\end{proof}

\begin{prop}\label{beatsup} For any $i,k\geq 0$ and $n\geq 2$ we have \[ (a_{i,n,k}-b_{i,n,k})(a_{i,n,k}-L_{n,k}b_{i,n,k})\leq 1.\]\end{prop}
\begin{proof}  We have \[ a_{i,n,k}-b_{i,n,k}=a_{i,n}-b_{i,n}=\frac{1}{2\sqrt{n^2-1}}\left((\sqrt{n^2-1}-1)\omega_{n}^{i}+(\sqrt{n^2-1}+1)\omega_{n}^{-i}\right).\]  Multiplying this by the identity in Lemma \ref{aLb} gives \begin{align*} (a_{i,n,k}-b_{i,n,k})&(a_{i,n,k}-L_{n,k}b_{i,n,k})=\frac{(\sqrt{n^2-1}-1)(H_{2k}+nP_{2k})}{\sqrt{n^2-1}\left(H_{2k}(\sqrt{n^2-1}+1)+2nP_{2k}-(\sqrt{n^2-1}-1)\right)}\\ &+\omega_{n}^{-2i}\frac{(\sqrt{n^2-1}+1)(H_{2k}+nP_{2k})}{\sqrt{n^2-1}\left(H_{2k}(\sqrt{n^2-1}+1)+2nP_{2k}-(\sqrt{n^2-1}-1)\right)}.\end{align*}

In particular, $(a_{i,n,k}-b_{i,n,k})(a_{i,n,k}-L_{n,k}b_{i,n,k})$ is a decreasing function of $i$, and its value for $i=0$ is \begin{equation}\label{prod0} \frac{2(H_{2k}+nP_{2k})}{H_{2k}(\sqrt{n^2-1}+1)+2nP_{2k}-(\sqrt{n^2-1}-1)}.\end{equation}  This is  at most $1$, since \[ H_{2k}(\sqrt{n^2-1}+1)-(\sqrt{n^2-1}-1)=2H_{2k}+(H_{2k}-1)(\sqrt{n^2-1}-1)\geq 2H_{2k}.\]
\end{proof}

We can now establish Theorem \ref{stairmain} from the introduction.  Specifically:

\begin{cor} \label{nkstair}
For any $k\geq 0$ and $n\geq 2$, $C_{L_{n,k}}$ has an infinite staircase, with accumulation point \begin{equation}\label{snkformula} S_{n,k}=\frac{(\sqrt{n^2-1}+1)P_{2k+1}+nH_{2k+1}}{(\sqrt{n^2-1}+1)P_{2k-1}+nH_{2k-1}} \end{equation} satisfying the equation $\frac{(1+S_{n,k})^2}{S_{n,k}}=\frac{2(1+L_{n,k})^2}{L_{n,k}}$.
\end{cor}

\begin{proof}
Indeed it follows quickly from what we have done that the distinct perfect classes $A_{i,n}^{(k)}$ all satisfy the criteria of Theorem \ref{infcrit} with $\beta=L_{n,k}$: that $a_{i,n,k}\geq L_{n,k}b_{i,n,k}$ is immediate from Lemma \ref{aLb}, and the inequalities $(a_{i,n,k}-L_{n,k}b_{i,n,k})^2<2L_{n,k}$ and $(a_{i,n,k}-b_{i,n,k})(a_{i,n,k}-L_{n,k}b_{i,n,k})< 1+L_{n,k}$ are given by Propositions \ref{beatvol} and \ref{beatsup} respectively. Finally, we need to check that $c_{i,n,k}$ and $d_{i,n,k}$ are relatively prime.  The proof of Corollary \ref{bqperfect} shows that the Brahmagupta moves preserve the property that $\gcd(c,d)=1$, so it suffices to show that $\gcd(c_{i,n},d_{i,n})=1$. By (\ref{cdred}), for $i\geq 2$ the ideal generated by $c_{i,n}$ and $d_{i,n}$ is the same as that generated by $c_{i-2,n}$ and $d_{i-2,n}$, reducing us to the case that $i\in\{0,1\}$.  But this case is obvious since $d_{0,n}=d_{1,n}=1$.  Thus Theorem \ref{infcrit} immediately implies the corollary, except that we still need to prove the formula (\ref{snkformula}) for $S_{n,k}$.  

To prove this formula, recall that by Theorem \ref{infcrit} our infinite staircase will have an accumulation point at $S_{n,k}=\lim_{i\to\infty}\frac{c_{i,n,k}}{d_{i,n,k}}$ so we just need to check that this limit is equal to the right hand side of (\ref{snkformula}).  Now it is immediate from (\ref{abcd}) that $\lim_{i\to\infty}\frac{c_{i,n}}{d_{i,n}}=\frac{1+\omega_n}{1-\omega_{n}^{-1}}$.  So by taking the limit as $i\to\infty$ of the ratio of the last two equations in (\ref{ink}) we find (using the identities $P_{m}-P_{m-2}=2P_{m-1}$ and $P_m+P_{m-2}=2H_{m-1}$): \begin{align*}
S_{n,k}&=\frac{P_{2k+2}(1+\omega_n)-P_{2k}(1-\omega_{n}^{-1})}{P_{2k}(1+\omega_n)-P_{2k-2}(1-\omega_{n}^{-1})} 
\\  &= \frac{2P_{2k+1}+P_{2k+2}(n+\sqrt{n^2-1})+P_{2k}(n-\sqrt{n^2-1})}{2P_{2k-1}+P_{2k}(n+\sqrt{n^2-1})+P_{2k-2}(n-\sqrt{n^2-1})}
\\ &= \frac{(1+\sqrt{n^2-1})P_{2k+1}+nH_{2k+1}}{(1+\sqrt{n^2-1})P_{2k-1}+nH_{2k-1}}.
\end{align*}
\end{proof}

\subsection{$A_{i,n}$ and the volume constraint}\label{undervolsect}

The proof of Corollary \ref{nkstair} (specifically, a combination of Propositions \ref{interval}, \ref{obsshape}, \ref{beatvol}, and \ref{beatsup})  shows that for each $i$ there is an open interval $U$ around $\frac{c_{i,n,k}}{d_{i,n,k}}$ such that, \begin{equation}\label{stairnhd} \mbox{for all $\alpha\in U$},\,\,\, C_{L_{n,k}}(\alpha)=\Gamma_{\alpha,L_{n,k}}(A_{i,n}^{(k)}).\end{equation}   This is enough to show that $C_{L_{n,k}}$ has an infinite staircase, but it leaves a number of other questions about $C_{L_{n,k}}$ unanswered since these open intervals may be rather small.   

Proposition \ref{Gammadef} shows that, for \emph{all} $\alpha$, we have $C_{L_{n,k}}(\alpha)\geq\sup_i\Gamma_{\alpha,L_{n,k}}(A_{i,n}^{(k)})$.  
Now the formulas (\ref{abcd}) obviously imply that $\frac{c_{i,n}}{d_{i,n}}<\frac{c_{i+1,n}}{d_{i+1,n}}$ and then based on (\ref{ink}) and (\ref{psquared}) (specifically the fact that $P_{2k}^{2}-P_{2k-2}P_{2k+2}=4>0$) it follows that $\frac{c_{i,n,k}}{d_{i,n,k}}<\frac{c_{i+1,n,k}}{d_{i+1,n,k}}$ for all $i,n,k$.  So by examining $\sup_i\Gamma_{\cdot,L_{n,k}}(A_{i,n}^{(k)})$ on the interval $[\frac{c_{i,n,k}}{d_{i,n,k}},\frac{c_{i+1,n,k}}{d_{i+1,n,k}}]$ we conclude that
\begin{equation}
\label{iiplus1} 
C_{L_{n,k}}(\alpha)\geq \max\left\{\frac{c_{i,n,k}}{a_{i,n,k}+L_{n,k}b_{i,n,k}} , \frac{d_{i+1,n,k}\alpha}{a_{i+1,n,k}+L_{n,k}b_{i+1,n,k}}\right\}\quad \mbox{for }\alpha\in \left[\frac{c_{i,n,k}}{d_{i,n,k}},\frac{c_{i+1,n,k}}{d_{i+1,n,k}}\right],\end{equation} with equality on a neighborhood of the endpoints of $\left[\frac{c_{i,n,k}}{d_{i,n,k}},\frac{c_{i+1,n,k}}{d_{i+1,n,k}}\right]$. Moreover our analysis shows that the maximum above is attained by the first term for $\alpha=\frac{c_{i,n,k}}{d_{i,n,k}}$ and by the second term for $\alpha=\frac{c_{i+1,n,k}}{d_{i,n,k}}$, so the value $\alpha=\frac{c_{i,n,k}(a_{i+1,n,k}+L_{n,k}b_{i+1,n,k})}{d_{i+1,n,k}(a_{i,n,k}+L_{n,k}b_{i,n,k})}$ at which the two terms are equal lies in the interval $\left[\frac{c_{i,n,k}}{d_{i,n,k}},\frac{c_{i+1,n,k}}{d_{i+1,n,k}}\right]$.
It turns out that (at least for $k=0$, though more extensive calculations would likely yield the same conclusion for arbitrary $k$) equality does \emph{not} hold in (\ref{iiplus1}) throughout the entire interval $\left[\frac{c_{i,n,k}}{d_{i,n,k}},\frac{c_{i+1,n,k}}{d_{i+1,n,k}}\right]$.

More specifically, restricting to the case $k=0$,  we have $L_{n,0}=\sqrt{n^2-1}$ and \[ a_{i,n,0}+L_{n,0}b_{i,n,0}=a_{i,n}+\sqrt{n^2-1}b_{i,n} = \omega_{n}^{i}.\]  Thus the right hand side of (\ref{iiplus1}) simplifies to \[ \max\{c_{i,n}\omega_{n}^{-i},d_{i+1,n}\alpha\omega_{n}^{-i-1}\} \] and the two terms in the maximum are equal to each other when $\alpha=\frac{c_{i,n}\omega_n}{d_{i+1,n}}$.

\begin{prop}\label{undervol}  For $i\geq 0$ and $n\geq 2$ we have \begin{equation}\label{undervolineq} c_{i,n}^2\omega_{n}^{-2i}<\frac{c_{i,n}\omega_n}{2\sqrt{n^2-1}d_{i+1,n}}.\end{equation}
\end{prop}

\begin{proof}
Note first of all that \[ \omega_{n}^{2}-1=\left(2n^2-1+2n\sqrt{n^2-1}\right)-1 = 2\sqrt{n^2-1}\left(n+\sqrt{n^2-1}\right)=2\sqrt{n^2-1}\omega_n.\]  

From this together with (\ref{abcd}) one computes: \begin{align*}
2\sqrt{n^2-1}c_{i,n}d_{i+1,n}&=\frac{1}{2\sqrt{n^2-1}}\left((\omega_n+1)\omega_{n}^{i}-(\omega_{n}+1)\omega_{n}^{-i-1}\right)\left((\omega_n-1)\omega_{n}^{i}+(\omega_n-1)\omega_{n}^{-i-1}\right)
\\ &= \frac{\omega_{n}^{2}-1}{2\sqrt{n^2-1}}\left(\omega_{n}^{2i}-\omega_{n}^{-2i-2}\right)
\\ &= \omega_{n}^{2i+1}-\omega_{n}^{-2i-1}.
\end{align*} 
Since (\ref{undervolineq}) is obviously equivalent to the inequality \[ 2\sqrt{n^2-1}c_{i,n}d_{i+1,n}<\omega_{n}^{2i+1} \] the conclusion is immediate.
\end{proof}

\begin{cor}\label{undervolcor}
For $\alpha=\frac{c_{i,n}\omega_n}{d_{i+1,n}}$, we have \begin{equation}\label{undervolcorineq} C_{L_{n,0}}(\alpha)\geq \sqrt{\frac{\alpha}{2L_{n,0}}}>\sup_i\Gamma_{\alpha,L_{n,0}}(A_{i,n}).\end{equation}
\end{cor}
\begin{proof}
Indeed the first inequality is just the volume constraint.  To prove the second inequality, observe that the left hand side of (\ref{undervolineq}) is the square of the right hand side of (\ref{iiplus1}) at $\alpha=\frac{c_{i,n}\omega_n}{d_{i+1,n}}$, which we noted earlier is equal to $\sup_{i}\Gamma_{\alpha,L_{n,0}}(A_{i,n})$, while the right hand side of (\ref{undervolineq}) is the square of the volume obstruction $\sqrt{\frac{\alpha}{2L_{n,0}}}$. \end{proof}

Thus the classes $A_{i,n}$ do not suffice by themselves to determine the behavior of $C_{L_{n,0}}$ between the stairs (at $\frac{c_{i,n}}{d_{i,n}}$) that we identified in proving Corollary \ref{nkstair}.  In fact we will see in Section \ref{ahatsect} that the first inequality in (\ref{undervolcorineq}) is also strict, as there are classes $\hat{A}_{i,n}$ which give obstructions that are stronger than the volume on the intervals on which $\sup_i\Gamma_{\alpha,L_{n,0}}(A_{i,n})$ falls under the volume constraint.

Computer calculations show that the analogous statement to Corollary \ref{undervolcor} continues to hold at least for many other values of $n$ and $k$ (not just for $k=0$ as above).  If $\alpha=\frac{c_{i,n,k}(a_{i+1,n,k}+L_{n,k}b_{i+1,n,k})}{d_{i+1,n,k}(a_{i,n,k}+L_{n,k}b_{i,n,k})}$ is the point at which the two terms in the maximum on the right hand side of (\ref{iiplus1}) agree, the statement that $\sqrt{\frac{\alpha}{2L_{n,k}}}>\sup_i\Gamma_{\alpha,L_{n,k}}(A_{i,n}^{(k)})$ is easily seen to be equivalent to the statement that \[ 2L_{n,k}c_{i,n,k}d_{i+1,n,k}-(a_{i,n,k}+L_{n,k}b_{i,n,k})(a_{i+1,n,k}+L_{n,k}b_{i+1,n,k}) <0. \]  Using the definitions (\ref{abcd}),(\ref{ink}),(\ref{Lnkdef}), one can expand the left hand side above in the form \[ r_{n,k}\omega_{n}^{2i}+s_{n,k}+t_{n,k}\omega_{n}^{-2i} \] where $r_{n,k},s_{n,k},t_{n,k}$ are (at least at first sight) rather complicated expressions involving Pell numbers and $\sqrt{n^2-1}$ but are independent of $i$.  Carrying this out in Mathematica, we have in fact found that $r_{n,k}=s_{n,k}=0$ for all $n,k$, and that $t_{n,k}<0$ whenever $n,k\leq 100$.  Thus for all $n,k\leq 100$ and all $i$ there are points $\alpha\in \left[\frac{c_{i,n,k}}{d_{i,n,k}},\frac{c_{i+1,n,k}}{d_{i+1,n,k}}\right]$ at which $\sqrt{\frac{\alpha}{2L_{n,k}}}>\sup_i\Gamma_{\alpha,L_{n,k}}(A_{i,n}^{(k)})$.

\subsection{Some facts about $L_{n,k}$ and $S_{n,k}$}\label{lnksnk}

Since the formulas (\ref{Lnkdef}) and (\ref{snkformula}) for the ``aspect ratios'' $L_{n,k}$ and the corresponding accumulation points $S_{n,k}$ are a bit complicated, let us point out a few elementary facts about these numbers.

\begin{prop}\label{snkint} For all $n\geq 2$ we we have \begin{equation}\label{sn0bound} 2n+2<S_{n,0}<2n+2+\frac{1}{2n-2},\end{equation} and, for $k\geq 1$, \[ \frac{P_{2k+4}}{P_{2k+2}}<S_{n,k}<S_{n+1,k}<\frac{P_{2k+2}}{P_{2k}}.\]
\end{prop}

\begin{proof} The $k=0$ case of (\ref{snkformula}) gives \[ S_{n,0}=\frac{\sqrt{n^2-1}+1+n}{\sqrt{n^2-1}+1-n}=\frac{1+\omega_n}{1-\omega_{n}^{-1}}.\]
Using that $\omega_{n}^{2}=2n\omega_{n}-1$ (and hence that $-2-\omega_{n}^{-1}=\omega_{n}-(2n+2)$), we find \begin{align*} 
(\omega_{n}-\omega_{n}^{-1})S_{n,0}&=\frac{(\omega_{n}-\omega_{n}^{-1})(\omega_n+1)}{1-\omega_{n}^{-1}}
\\ &= \frac{\omega_{n}^{2}+\omega_n-1-\omega_{n}^{-1}}{1-\omega_{n}^{-1}}=\frac{(2n+1)\omega_{n}-2-\omega_{n}^{-1}}{1-\omega_{n}^{-1}}
\\ &= \frac{(2n+2)\omega_{n}-(2n+2)}{1-\omega_{n}^{-1}}=(2n+2)\omega_{n}.
\end{align*} Thus \[ S_{n,0}=(2n+2)\frac{\omega_{n}}{\omega_{n}-\omega_{n}^{-1}}>2n+2.\]  Furthermore \begin{align*}
(2n-2)(S_{n,0}-(2n+2))&=\frac{4(n^2-1)\omega_{n}^{-1}}{\omega_{n}-\omega_{n}^{-1}}=\frac{4(n^2-1)(n-\sqrt{n^2-1})}{2\sqrt{n^2-1}}
\\ &= 2n\sqrt{n^2-1}-2(n^2-1)<1
\end{align*}
where the last inequality follows from the fact that $\left(n-\frac{1}{2n}\right)^2>n^2-1$, so that $2n\sqrt{n^2-1}<2n\left(n-\frac{1}{2n}\right)=2n^2-1$.  This completes the proof of (\ref{sn0bound}).

For the remaining statement, recall that by construction we have $S_{n,k}=\lim_{i\to\infty}\frac{c_{i,n,k}}{d_{i,n,k}}$.  So by taking the limit of the ratio of the last two equations of (\ref{ink}) as $i\to\infty$ it is clear that \[ S_{n,k}=\frac{P_{2k+2}S_{n,0}-P_{2k}}{P_{2k}S_{n,0}-P_{2k-2}} \] for $k\geq 1$.  Since $\frac{P_{2k+2}}{P_{2k}}<\frac{P_{2k}}{P_{2k-2}}$ by Proposition \ref{orderratios}, the function $t\mapsto \frac{P_{2k+2}t-P_{2k}}{P_{2k}t-P_{2k-2}}$ is strictly increasing.  By (\ref{sn0bound}) and our assumption that $n\geq 2$, we have $S_{n+1,0}>S_{n,0}>6$ and hence \[ \frac{6P_{2k+2}-P_{2k}}{6P_{2k}-P_{2k-2}}< S_{n,k}<S_{n+1,0}<\lim_{t\to\infty}\frac{P_{2k+2}t-P_{2k}}{P_{2k}t-P_{2k-2}}.\]  But the limit on the right is equal to $\frac{P_{2k+2}}{P_{2k}}$, while  (\ref{4gapP}) shows that the expression on the left is equal to $\frac{P_{2k+4}}{P_{2k+2}}$.
\end{proof}

We now describe the locations of the $L_{n,k}$, in particular indicating how they compare to the various $b_m$ from (\ref{bdef}).

\begin{prop}\label{Lsize} For any $n\geq 2$ and $k\geq 0$ we have $L_{n,k}<L_{n+1,k}$, and \[ L_{n,k}\in\left\{\begin{array}{ll} (b_{2k},b_{2k-1}) & \mbox{if }n\geq 4 \\ (b_{2k+1},b_{2k}) & \mbox{if } n=3 \\ (b_{2k+2},b_{2k+1}) & \mbox{if }n=2.\end{array}\right.\]  Also, for all $n\geq 2$ and all $k$ we have \[ \frac{2P_{2k+2}+1}{2P_{2k+2}-1}<L_{n,k}<\frac{H_{2k+1}+1}{H_{2k+1}-1}=\lim_{n\to\infty}L_{n,k}.\]
\end{prop}

\begin{proof}
Since $x\mapsto \frac{(1+x)^2}{x}$ is a strictly increasing function of $x$ as long as  $x>1$, the equation $\frac{(1+S_{n,k})^2}{S_{n,k}}=\frac{2(1+L_{n,k})^2}{L_{n,k}}$ and the fact that $S_{n,k},L_{n,k}>1$ allow us to conclude that $L_{n,k}<L_{n+1,k}$ directly from the fact (proven in Proposition \ref{snkint}) that $S_{n,k}<S_{n+1,k}$.

Let \begin{equation}\label{mnkdef} m_{n,k}=\frac{L_{n,k}-1}{L_{n,k}+1}=\frac{\sqrt{n^2-1}-1}{H_{2k}(\sqrt{n^2-1}+1)+2nP_{2k}}.\end{equation}  Since $t\mapsto \frac{t-1}{t+1}$ is a strictly increasing function for $t>0$, to show that $L_{n,k}$ lies in some interval $(s,t)$ it suffices to show that $m_{n,k}$ lies in $\left(\frac{s-1}{s+1},\frac{t-1}{t+1}\right)$.

From (\ref{bdef}) one finds \[ \frac{b_{2k}+1}{b_{2k}-1}=P_{2k+2},\qquad \frac{b_{2k+1}+1}{b_{2k+1}-1}=H_{2k+2}.\]    We see that \begin{align*} \frac{1}{m_{2,k}}&=\frac{(\sqrt{3}+1)H_{2k}+4P_{2k}}{\sqrt{3}-1}=(2+\sqrt{3})H_{2k}+(2+2\sqrt{3})P_{2k}
\\ &= 2P_{2k+1}+\sqrt{3}H_{2k+1}\in (H_{2k+2},2P_{2k+2}) \end{align*} since $H_{2k+2}=2P_{2k+1}+H_{2k+1}$ while $2P_{2k+2}=2P_{2k+1}+2H_{2k+1}$. Thus $m_{2,k}\in(\frac{1}{2P_{2k+2}},\frac{1}{H_{2k+2}})\subset \left(\frac{b_{2k+2}-1}{b_{2k+2}+1},\frac{b_{2k+1}-1}{b_{2k+1}+1}\right)$ and so $L_{2,k}\in (b_{2k+2},b_{2k+1})$ (and the argument shows more specifically that $L_{2,k}>\frac{2P_{2k+2}+1}{2P_{2k+2}-1}$, which is larger than $b_{2k+2}$).  

Also, \begin{align*} \frac{1}{m_{3,k}} &= \frac{\sqrt{8}+1}{\sqrt{8}-1}H_{2k}+\frac{6}{\sqrt{8}-1}P_{2k}=\left(\frac{9+4\sqrt{2}}{7}\right)H_{2k}+\left(\frac{6+12\sqrt{2}}{7}\right)P_{2k}
\\ &\quad\in \left( 2H_{2k}+3P_{2k},3H_{2k}+4P_{2k}\right) = \left(2P_{2k+1}+P_{2k},3P_{2k+1}+P_{2k}\right)\\ &\quad\,\, =(P_{2k+2},P_{2k+2}+P_{2k+1})=(P_{2k+2},H_{2k+2})  \end{align*} where we have used both equalities in (\ref{hpadd}) and the facts that $2<\frac{9+4\sqrt{2}}{7}<3$ and $3<\frac{6+12\sqrt{2}}{7}<4$.  So $m_{3,k}\in (\frac{1}{H_{2k+2}},\frac{1}{P_{2k+2}})$ and so $L_{3,k}\in (b_{2k+1},b_{2k})$.

Similarly, \begin{align*} \frac{1}{m_{4,k}} &= \frac{\sqrt{15}+1}{\sqrt{15}-1}
H_{2k}+\frac{8}{\sqrt{15}-1}P_{2k}=\left(\frac{8+\sqrt{15}}{7}  \right)H_{2k}+
\left(\frac{4+4\sqrt{15}}{7} \right)P_{2k} \\ &< P_{2k+2} \end{align*} which 
implies that $L_{4,k}>b_{2k}$.  So since $L_{n,k}<L_{n+1,k}$ for all $n$ we 
have $L_{n,k}>b_{2k}$ for $n\geq 4$. Furthermore \[ L_{n,k}<\lim_{n\to\infty}L
_{n,k}=\frac{H_{2k}+2P_{2k}+1}{H_{2k}+2P_{2k}-1}=\frac{H_{2k+1}+1}{H_{2k+1}-1}.
\] Since $\frac{H_{2k+1}+1}{H_{2k+1}-1}<\frac{H_{2k}+1}{H_{2k}-1}=b_{2k-1}$ 
this suffices to complete the proof.
\end{proof}

We now give a bit more information about the function of two variables $(\alpha,\beta)\mapsto C_{\beta}(\alpha)$ near $(\alpha,\beta)=(S_{n,k},L_{n,k})$.  Some useful context is provided by the following mild extension of \cite[Corollary 4.9]{FM}:

\begin{prop}\label{finaway}
Fix $\gamma>1$.  Then there are only finitely many elements $E=(x,y;\vec{m})\in\tilde{\mathcal{E}}$ having all $m_i\neq 0$ for which there exist $\alpha,\beta>1$ such that $\mu_{\alpha,\beta}(E)\geq \gamma\sqrt{\frac{\alpha}{2\beta}}$.
\end{prop}

(Recall from the introduction that $\tilde{\mathcal{E}}$ consists of classes $(x,y;\vec{m})\in\mathcal{H}^2$ having Chern number $1$ and self-intersection $-1$ which are either equal to the Poincar\'e duals of the standard exceptional divisors $E'_i$ or else have all coordinates of $\vec{m}$ nonnegative; in particular $\mathcal{E}\subset\tilde{\mathcal{E}}$. The provision that all $m_i\neq 0$ is included due to the trivial point that if $(x,y;\vec{m})$ satisfies the condition then so does $(x,y;0,\ldots,0,\vec{m})$ (which formally speaking represents a different element of $\mathcal{E}\subset \mathcal{H}^2$) where the string of zeros can be arbitrarily long.)
\begin{proof}
Suppose that $E\in\mathcal{E}$ has $\mu_{\alpha,\beta}(E)\geq \gamma\sqrt{\frac{\alpha}{2\beta}}$ where $\alpha,\beta\geq 1$   and write $E=(x,y;\vec{m})$, so the fact that $E$ has self-intersection $-1$ shows that $|\vec{m}|=2xy+1$. Since $\mu_{\alpha,\beta}(E'_i)=0$ by definition, $E$ is not one of the standard classes $E'_i$, so all coordinates of $\vec{m}$ are nonnegative and hence $x+y>0$ with $x,y\geq 0$.  Using the Cauchy-Schwarz inequality and the fact that $|w(\alpha)|^2=\alpha$ we then find that \begin{equation}\label{boundaway} \gamma\sqrt{\frac{\alpha}{2\beta}}\leq \mu_{\alpha,\beta}(E)=\frac{w(\alpha)\cdot\vec{m}}{x+\beta y}\leq \frac{\sqrt{\alpha}\sqrt{2xy+1}}{x+\beta y}.\end{equation}

Now since $0\leq (x-\beta y)^2=(x+\beta y)^2-4\beta xy$ and since $x+\beta y>0$, rearranging (\ref{boundaway}) gives \[ \gamma\leq \sqrt{\frac{4\beta xy+2\beta}{4\beta xy}}=\sqrt{1+\frac{1}{2xy}},\] \emph{i.e.} \[ 2xy\leq \frac{1}{\gamma^2-1}.\]  But there are only finitely many pairs of positive integers $x,y$ obeying $2xy\leq \frac{1}{\gamma^2-1}$; for each of these pairs $(x,y)$ there are only finitely many sequences of positive integers $m_i$ obeying $2xy-\sum m_{i}^{2}=-1$.    Thus there are only finitely many classes satisfying the condition that have both $x,y>0$.  We should also consider the possibility that one of $x,y$ is zero, but in this case the condition $2xy-\sum_{i=1}^{N}m_{i}^{2}=-1$ with all $m_i$ nonzero forces $N$ to be $1$ and $m_1$ to be $\pm 1$, so that the Chern number condition $2(x+y)-\sum m_i=1$ forces either $x+y=0$ and $m_1=-1$ or $x+y=1$ and $m_1=1$.  So allowing $x$ or $y$ to be zero does not change the fact that only finitely many classes in $\tilde{\mathcal{E}}$ obey the condition.
\end{proof}

\begin{cor}\label{accvol}
If $\beta,S\geq 1$ and if the function $C_{\beta}$ has an infinite staircase accumulating at $S$ then $C_{\beta}(S)=\sqrt{\frac{S}{2\beta}}$.
\end{cor}

\begin{proof}
Of course $C_{\beta}(S)\geq \sqrt{\frac{S}{2\beta}}$ by volume considerations.  If equality failed to hold then we could find a neighborhood $U$ of $S$ and a value $\gamma>1$ such that $C_{\beta}(\alpha)\geq \gamma\sqrt{\frac{\alpha}{2\beta}}$ for all $\alpha\in U$.  But then by (\ref{dirlim}) and Proposition \ref{finaway} $C_{\beta}(\alpha)$ would be given for $\alpha\in U$ as the maximum of the values $\mu_{\alpha,\beta}(E)$ where $E$ varies through a finite subset of $\mathcal{E}$ that is independent of $\alpha$.  Since the functions $\alpha\mapsto \mu_{\alpha,\beta}(E)$ are piecewise affine (with finitely many pieces) by \cite[Section 2]{MS} this would contradict the fact that $S$ is an accumulation point of an infinite staircase.  
\end{proof}

In particular Corollary \ref{accvol} applies with $\beta=L_{n,k}$ and $S=S_{n,k}$, so that $C_{L_{n,k}}(S_{n,k})$ agrees with the volume bound, and we have \begin{equation}\label{slvol} \mu_{S_{n,k},L_{n,k}}(E)\leq \sqrt{\frac{S_{n,k}}{2L_{n,k}}}\mbox{  for all }E\in\mathcal{E}.\end{equation}  We will see now that there are at least two distinct choices of $E\in\mathcal{E}$ for which the bound (\ref{slvol}) is sharp.  This will later help give us some indication of what happens to our infinite staircases when $\beta$ is varied away from one of the $L_{n,k}$. 

\begin{prop}\label{2221}
Let $n\geq 2$ and $k\geq 1$.  Then \[ \mu_{S_{n,k},L_{n,k}}\left((2,2;2,1^{\times 5})\right)=\sqrt{\frac{S_{n,k}}{2L_{n,k}}}.\]
\end{prop}

\begin{proof}
By Propositions \ref{orderratios} and \ref{snkint} we have \[ \sigma^2<S_{n,k}<\frac{P_{4}}{P_2}=6,\] so since $\sigma^2=3+2\sqrt{2}>5$ we have \[ w(S_{n,k})=(1^{\times 5},\mathcal{W}(1,S_{n,k}-5))=(1^{\times 5},S_{n,k}-5,\mathcal{W}(6-S_{n,k},S_{n,k}-5)).\]  Hence \begin{align*}
\mu_{S_{n,k},L_{n,k}}&\left((2,2;2,1^{\times 5})\right)= \frac{(2,1^{\times 5})\cdot (1^{\times 5},S_{n,k}-5,\mathcal{W}(6-S_{n,k},S_{n,k}-5))}{2+2L_{n,k}} \\ &= \frac{2+4+S_{n,k}-5}{2(1+L_{n,k})}=\frac{1}{2}\frac{1+S_{n,k}}{1+L_{n,k}}.
\end{align*} But the identity $\frac{(1+S_{n,k})^2}{S_{n,k}}=\frac{2(1+L_{n,k})^2}{L_{n,k}}$ from Corollary \ref{nkstair} immediately implies that $\frac{1}{2}\frac{1+S_{n,k}}{1+L_{n,k}}=\sqrt{\frac{S_{n,k}}{2L_{n,k}}}$.
\end{proof}

Proposition \ref{2221} does not apply to the case $k=0$ because $S_{n,0}>6$, leading $w(S_{n,0})$ to have a different form.  Here is the analogous statement for that case.

\begin{prop}  \label{gnmatch}
For any $n\geq 2$ let $G_n=\left(2n^2-n-1,2n-1;2n-1,(2n-2)^{\times(2n+1)},1^{\times (2n-2)}\right)$.  Then $G_n\in\mathcal{E}$, and \[ \mu_{S_{n,0},L_{n,0}}(G_n)=\sqrt{\frac{S_{n,0}}{2L_{n,0}}}=\frac{\omega_n}{\omega_n-1}.\]
\end{prop}

\begin{proof} Changing basis as usual we find \[ G_n=\left\langle 2n^2-n-1;2n^2-3n,0,(2n-2)^{\times (2n+1)},1^{\times (2n-2)}\right\rangle.\]  Applying $n$ Cremona moves $\frak{c}_{034},\frak{c}_{056},\ldots,\frak{c}_{0,2n+1,2n+2}$, each with $\delta=-2n+3$, reduces this to \[ \langle 2n-1;0,0,2n-2,1^{\times (4n-2)}\rangle \] which after deleting zeros and changing basis simply yields $(2n-2,1;\mathcal{W}(4n-3,1))$ which is the familiar class $A_{1,2n-2}$ that of course belongs to $\mathcal{E}$. Thus $G_n\in\mathcal{E}$.

Now in view of Proposition \ref{snkint} we have \begin{align*} w(S_{n,0})&=\left(1^{\times(2n+2)},\mathcal{W}(1,S_{n,0}-(2n+2))\right)
\\ &= \left(1^{\times(2n+2)},(S_{n,0}-(2n+2))^{\times (2n-2)},\mathcal{W}(1-(2n-2)(S_{n,0}-(2n+2)),S_{n,0}-(2n+2))\right) \end{align*}
and hence \begin{align*} & \left(2n-1,(2n-2)^{\times(2n+1)},1^{\times (2n-2)}\right)\cdot w(S_{n,0}) = 2n-1+(2n-2)(2n+1)+(2n-2)\left(S_{n,0}-(2n+2)\right) \\ \qquad &= 1+(2n-2)(2n+2)+(2n-2)\left(S_{n,0}-(2n+2)\right) =(2n-2)S_{n,0}+1.\end{align*}  Hence \begin{equation}\label{gnobs} \mu_{S_{n,0},L_{n,0}}(G_n)= \frac{(2n-2)S_{n,0}+1}{(2n^2-n-1)+(2n-1)L_{n,0}}.\end{equation}

Now $L_{n,0}=\sqrt{n^2-1}$, so the denominator of the above fraction is \[ (2n^2-n-1)+(2n-1)\sqrt{n^2-1}=(2n-1)(n+\sqrt{n^2-1})-1=\omega_{n}^{2}-\omega_{n}.\]  So since $S_{n,0}=\frac{\omega_{n}+1}{1-\omega_{n}^{-1}}$ and $\omega_{n}^{2}=2n\omega_{n}-1$, \begin{align} \nonumber
\mu_{S_{n,0},L_{n,0}}(G_n) &= \frac{(2n-2)\frac{\omega_{n}+1}{1-\omega_{n}^{-1}}+1}{\omega_{n}^{2}-\omega_{n}} =\frac{(2n-2)(\omega_{n}+1)+(1-\omega_{n}^{-1})}{(\omega_{n}-1)^2}\\  \nonumber &= \frac{\omega_{n}^{2}-2\omega_n+2n-\omega_{n}^{-1}}{(\omega_{n}-1)^2} = \frac{\omega_{n}^{2}-\omega_{n}}{(\omega_{n}-1)^2} \\ \label{mugncalc} &= \frac{\omega_n}{\omega_n-1}.\end{align}

On the other hand since $2L_{n,0}=2\sqrt{n^2-1}=\omega_{n}-\omega_{n}^{-1}$ we have \begin{align*}
\frac{S_{n,0}}{2L_{n,0}} &= \frac{\omega_n+1}{(1-\omega_{n}^{-1})(\omega_{n}-\omega_{n}^{-1})}=\frac{\omega_{n}^{2}(\omega_{n}+1)}{(\omega_{n}-1)(\omega_{n}^2-1)}
\\ &= \frac{\omega_{n}^{2}}{(\omega_n-1)^2}.
\end{align*}
So by (\ref{mugncalc}) we indeed have $\mu_{S_{n,0},L_{n,0}}(G_n)=\sqrt{\frac{S_{n,0}}{2L_{n,0}}}=\frac{\omega_n}{\omega_n-1}$.
\end{proof}

\begin{prop}\label{an+1}
Let $k\geq 0$ and $n\geq 2$.  Then \[ \Gamma_{S_{n,k},L_{n,k}}(A_{1,n+1}^{(k)}) = \sqrt{\frac{S_{n,k}}{2L_{n,k}}}.\]  
\end{prop}

\begin{proof}
We have, freely using identities from Section \ref{pellprelim}, \begin{align*} A_{1,n+1}^{(k)}&=\left(\frac{1}{2}((H_{2k}+1)(n+1)+H_{2k}+(2n+2)P_{2k}-1),\frac{1}{2}((H_{2k}-1)(n+1)+H_{2k}+(2n+2)P_{2k}+1);\right.\\ & \quad \left.\mathcal{W}\left(\frac{1}{2}(P_{2k+2}(2n+3)-P_{2k}),\frac{1}{2}(P_{2k}(2n+3)-P_{2k-2})\right)\right)
\\ &= \left(\frac{1}{2}(n(H_{2k+1}+1)+2P_{2k+1}),\frac{1}{2}(n(H_{2k+1}-1)+2P_{2k+1});\mathcal{W}(nP_{2k+2}+H_{2k+2},nP_{2k}+H_{2k})\right).\end{align*}

Now we have $S_{n,k}=\frac{P_{2k+2}S_{n,0}-P_{2k}}{P_{2k}S_{n,0}-P_{2k-2}}$ and $S_{n,0}<2n+3$ by Proposition \ref{snkint}, so since $t\mapsto \frac{P_{2k+2}t-P_{2k}}{P_{2k}t-P_{2k-2}}$ is an increasing function (as can be seen from Proposition \ref{orderratios}) it follows that $S_{n,k}<\frac{c_{1,n+1,k}}{d_{1,n,k}}=\frac{P_{2k+2}(n+1)-P_{2k}}{P_{2k}(n+1)-P_{2k-2}}$.  So $S_{n,k}$ lies in the region on which $\Gamma_{\cdot,L_{n,k}}(A_{1,n+1}^{(k)})$ is linear, and we have \begin{align*} \Gamma_{S_{n,k},L_{n,k}}(A_{1,n+1}^{(k)}) &=\frac{(nP_{2k}+H_{2k})S_{n,k}}{\frac{1}{2}(n(H_{2k+1}+1)+2P_{2k+1})+\frac{1}{2}(n(H_{2k+1}-1)+2P_{2k+1})L_{n,k}} \\ &= \frac{(nP_{2k}+H_{2k})S_{n,k}}{(nH_{2k+1}+2P_{2k+1})(L_{n,k}+1)/2-n(L_{n,k}-1)/2}
\\ &= \frac{(nP_{2k}+H_{2k})\frac{(\sqrt{n^2-1}+1)P_{2k+1}+nH_{2k+1}}{(\sqrt{n^2-1}+1)P_{2k-1}+nH_{2k-1}}}{\frac{(nH_{2k+1}+2P_{2k+1})((\sqrt{n^2-1}+1)H_{2k}+2nP_{2k})-n(\sqrt{n^2-1}-1)    }{H_{2k}(\sqrt{n^2-1}+1)+2nP_{2k}-(\sqrt{n^2-1}-1)    }}\end{align*} 

Meanwhile, \begin{align*} \sqrt{\frac{S_{n,k}}{2L_{n,k}}}&=\frac{S_{n,k}+1}{2(L_{n,k}+1)}
\\ &= \frac{\frac{2(\sqrt{n^2-1}+1)H_{2k}+4nP_{2k}}{(\sqrt{n^2-1}+1)P_{2k-1}+nH_{2k-1}}}{\frac{4(\sqrt{n^2-1}+1)H_{2k}+8nP_{2k}}{H_{2k}(\sqrt{n^2-1}+1)+2nP_{2k}-(\sqrt{n^2-1}-1)}}.\end{align*}

Thus \begin{align*} \frac{\Gamma_{S_{n,k},L_{n,k}}(A_{1,n+1}^{(k)})}{\sqrt{S_{n,k}/(2L_{n,k})}} &= \frac{2(nP_{2k}+H_{2k})((\sqrt{n^2-1}+1)P_{2k+1}+nH_{2k+1})}{(nH_{2k+1}+2P_{2k+1})((\sqrt{n^2-1}+1)H_{2k}+2nP_{2k})-n(\sqrt{n^2-1}-1)}.\end{align*}  By expanding out both the numerator and the denominator and twice using the identity $H_{2k}H_{2k+1}=2P_{2k}P_{2k+1}+1$, one easily finds that the above fraction simplifies to $1$.
\end{proof}

Thus at the accumulation point $S_{n,k}$ of each of our infinite staircases we have two distinct classes (namely $G_n$ and $A_{1,n+1}$ if $k=0$, and $(2,2;2,1^{\times 5})$ and $A_{1,n+1}^{(k)}$ if $k>0$), which are not themselves involved in the infinite staircase for $L_{n,k}$, but whose associated obstructions exactly match the volume bound at $(S_{n,k},L_{n,k})$.  The following discussion will show that these classes lead the infinite staircase to disappear when the aspect ratio $\beta$ of the target polydisk is varied from $L_{n,k}$.  We leave the proof of the following simple calculus exercise to the reader.

\begin{prop}\label{movesign}
Let $a,b,c,d,t_0>0$ and suppose that $\frac{c}{a+t_0b}=\frac{d}{\sqrt{t_0}}$.  Then \[ \frac{d}{dt}\left[\frac{c}{a+tb}-\frac{d}{\sqrt{t}}\right]_{t=t_0} \mbox{ has the same sign as }a-t_0b.\]
\end{prop}

\begin{cor}\label{moveL}
Given $n\geq 2$, $k\geq 0$ there is $\ep>0$ such that \[ \Gamma_{S_{n,k},\beta}(A_{1,n+1}^{(k)})>\sqrt{\frac{S_{n,k}}{2\beta}} \mbox{ for }L_{n,k}<\beta<L_{n,k}+\ep,\] and \[ \mu_{S_{n,k},\beta}\left((2,2;2,1^{\times 5})\right)>\sqrt{\frac{S_{n,k}}{2\beta}}\mbox{ for }L_{n,k}-\ep<\beta <L_{n,k}\mbox{ if }k\geq 1 \] while \[ \mu_{S_{n,k},\beta }(G_n)>\sqrt{\frac{S_{n,k}}{2\beta }}\mbox{ for }L_{n,k}-\ep<\beta <L_{n,k}\mbox{ if }k=0. \]
Thus for any choice of $n,k$ we have $C_{\beta}(S_{n,k})>\sqrt{\frac{S_{n,k}}{2\beta}}$ for $0<|\beta-L_{n,k}|<\ep$. 
\end{cor}

\begin{proof}As functions of $\beta $, $\Gamma_{S_{n,k},\beta}(A_{1,n+1}^{(k)})$ and the various $\mu_{S_{n,k},\beta }(E)$ have the form $\beta \mapsto \frac{c}{a+\beta b}$ where $a,b$ are first two entries in the expression of $A_{1,n+1}^{(k)}$ or $E$ as $(a,b;\vec{m})$.  So by Propositions \ref{an+1} and \ref{movesign} the first statement follows from the statement that $a_{1,n+1,k}>L_{n,k}b_{1,n+1,k}$, which holds by Lemma \ref{aLb} (and the fact that $L_{n+1,k}>L_{n,k}$).  Similarly the second statement follows from Propositions \ref{2221} and \ref{movesign} and the fact that $2-2L_{n,k}<0$.   Finally the third statement follows from Propositions \ref{gnmatch} and \ref{movesign} and the calculation \[ (2n^2-n-1)^2-L_{n,0}^{2}(2n-1)^2=(2n^2-n-1)^2-(n^2-1)(2n-1)^2=-2n+2<0.\]
\end{proof}

Combining Corollary \ref{moveL} with Proposition \ref{finaway} and continuity considerations, we see that if $\beta$ is sufficiently close to but not equal to $L_{n,k}$ then $C_{\beta}$ is given on a neighborhood $U_{\beta}$ of $S_{n,k}$ as the maximum of a finite collection of obstruction functions $\mu_{\cdot,\beta}(E)$.  In particular, for any given such $\beta$, only finitely many of the $A_{i,n}^{(k)}$ can influence $C_{\beta}$ in this neighborhood.  A bit more strongly, since $\frac{c_{i,n,k}}{d_{i,n,k}}\to S_{n,k}$, for all but finitely many $i$ it will hold that $\frac{c_{i,n,k}}{d_{i,n,k}}\in U_{\beta}$, and so we will have $C_{\beta}\left(\frac{c_{i,n,k}}{d_{i,n,k}}\right)>\mu_{\frac{c_{i,n,k}}{d_{i,n,k}},\beta}(A_{i,n}^{(k)})$.  But just as in the proof of Proposition \ref{Gammadef} this latter inequality implies that in fact $C_{\beta}(\alpha)>\Gamma_{\alpha,\beta}(A_{i,n}^{(k)})$ for \emph{all} $\alpha$.  Thus for a fixed $\beta$ with $0<|\beta-L_{n,k}|<\ep$, only finitely many of the $\Gamma_{\cdot,\beta}(A_{i,n}^{(k)})$ ever coincide with $C_{\beta}$.  Similar remarks apply to the classes $\hat{A}_{i,n}^{(k)}$ discussed in the next section.

\subsection{Additional obstructions} \label{ahatsect}

We will see now that the $A_{i,n}=A_{i,n}^{(0)}$ are not the only classes that contribute to the infinite staircase for $L_{n,0}=\sqrt{n^2-1}$.  For each $n\geq 2$ define a sequence of integer vectors $\vec{w}_{i,n}=(\hat{a}_{i,n},\hat{b}_{i,n},\hat{c}_{i,n},\hat{d}_{i,n})$ by: \begin{align*} \vec{w}_{-1,n}&= (n+1,-1,-1,2n+1) \\ \vec{w}_{0,n} &= (n-1,1,2n-1,1) \\ \vec{w}_{i+2,n} &= (4n^2-2)\vec{w}_{i+1,n}-\vec{w}_{i,n}-(0,4n,4n+4,4n-4).\end{align*}   
Since it is clear from these recurrences that $\hat{a}_{i,n},\hat{b}_{i,n},\hat{c}_{i,n},\hat{d}_{i,n}$ are all nonnegative for $i\geq 0$ we can then define a class \[ \hat{A}_{i,n}=\left(\hat{a}_{i,n},\hat{b}_{i,n},\mathcal{W}(\hat{c}_{i,n},\hat{d}_{i,n})\right).\]  
In terms of the $a_{i,n},b_{i,n},c_{i,n},d_{i,n}$ from (\ref{abcd}) one finds that:
\begin{align}\label{abcdhat} \hat{a}_{i,n} &= a_{2i+1,n}-b_{2i+1,n}\\ \nonumber \hat{b}_{i,n} &= b_{2i+1,n}-\frac{1}{n^2-1}(a_{2i+1,n}-n) \\ \nonumber \hat{c}_{i,n} &= b_{2i+2,n}-\frac{1}{n-1}(a_{2i+1,n}-1) \\ \nonumber \hat{d}_{i,n} &= -b_{2i,n}+\frac{1}{n+1}(a_{2i+1,n}+1). \end{align}  Using (\ref{abcd}) one then finds that \begin{align}\nonumber \frac{a_{2i+1,n}-1}{n-1}\hat{d}_{i,n}-b_{2i,n}\hat{c}_{i,n} &= \frac{a_{2i+1,n}^{2}-1}{n^2-1}-b_{2i,n}b_{2i+2,n} \\ \nonumber &= \frac{1}{4(n^2-1)}\left((\omega_{n}^{4i+2}-2+\omega_{n}^{-4i-2})-(\omega_{n}^{2i}-\omega_{n}^{-2i})(\omega_{n}^{2i+2}-\omega_{n}^{-2i-2})\right)
\\ &= \frac{(\omega_{n}-\omega_{n}^{-1})^{2}}{4(n^2-1)}=1 \label{gcdhat} \end{align} which in particular implies that $\gcd(\hat{c}_{i,n},\hat{d}_{i,n})=1$.  Also, using that (by a straightforward induction argument) \[ b_{2i+2,n}-b_{2i,n}=2nb_{2i+1,n}-2b_{2i,n}=2a_{2i+1,n},\] one sees that \begin{align*} \hat{c}_{i,n}+\hat{d}_{i,n} &= b_{2i+2,n}-b_{2i,n}+\frac{1}{n^2-1}(2n-2a_{2i+1,n})=2(\hat{a}_{i,n}+\hat{b}_{i,n}).\end{align*}  Together with the fact that $\gcd(\hat{c}_{i,n},\hat{d}_{i,n})=1$ this implies that $\hat{A}_{i,n}$ has Chern number $1$ by \cite[Lemma 1.2.6]{MS}.  Moreover a routine computation shows that $2\hat{a}_{i,n}\hat{b}_{i,n}-\hat{c}_{i,n}\hat{d}_{i,n}=-1$, \emph{i.e.} that $\hat{A}_{i,n}$ has self-intersection $-1$.

This suffices to show that $\hat{A}_{i,n}$ is quasi-perfect and hence, as seen in Proposition \ref{Gammadef}, that $C_{\beta}(\alpha)\geq \Gamma_{\alpha,\beta}(\hat{A}_{i,n})$.  We expect that the $\hat{A}_{i,n}$ are all perfect, but we will neither prove nor use this.

We record the following identities, each of which can be proven by a straightforward but (in some cases) tedious calculation based on (\ref{abcd}) and (\ref{abcdhat}):

\begin{equation} \label{cdhatcdi} \hat{c}_{i,n}d_{i,n}-\hat{d}_{i,n}c_{i,n}=2(a_{i+1,n}-b_{i+1,n}); \end{equation} 
\begin{equation} \label{cdhatcdplus} \hat{c}_{i,n}d_{i+1,n}-\hat{d}_{i,n}c_{i+1,n} = -2(a_{i,n}-b_{i,n});\end{equation}
\begin{equation} \label{AhataboveA} \hat{c}_{i,n}(a_{i,n}+\sqrt{n^2-1}b_{i,n})-c_{i,n}(\hat{a}_{i,n}+\sqrt{n^2-1}\hat{b}_{i,n})=\omega_{n}^{-i-1};\end{equation}
\begin{equation} \label{AhataboveAplus} \hat{d}_{i,n}(a_{i+1,n}+\sqrt{n^2-1}b_{i+1,n})-d_{i+1,n}(\hat{a}_{i,n}+\sqrt{n^2-1}\hat{b}_{i,n})=\omega_{n}^{-i};\end{equation}
\begin{equation} \label{lastjunction} 2\hat{c}_{i,n}d_{i+1,n}\sqrt{n^2-1}-(\hat{a}_{i,n}+\sqrt{n^2-1}\hat{b}_{i,n})(a_{i+1,n}+\sqrt{n^2-1}b_{i+1,n})=\frac{\omega_{n}^{-i-1}}{\sqrt{n^2-1}}\left(n-(1+\sqrt{n^2-1})\omega^{-2i-1}_{n}\right).\end{equation}

(In each of our applications of these identities the signs of the right hand sides, not their exact values, will be what is relevant.)
Since $a_{i,n}>b_{i,n}$ for all $i\geq 0,n\geq 2$, the identities (\ref{cdhatcdi}) and (\ref{cdhatcdplus}) show that \begin{equation}\label{ordercd} \frac{c_{i,n}}{d_{i,n}}<\frac{\hat{c}_{i,n}}{\hat{d}_{i,n}} < \frac{c_{i+1,n}}{d_{i+1,n}} .\end{equation}  Now $\Gamma_{\frac{c_{i,n}}{d_{i,n}},\sqrt{n^2-1}}(A_{i,n})=\frac{c_{i,n}}{a_{i,n}+\sqrt{n^2-1}b_{i,n}}$ provides a lower bound for the value of $C_{\sqrt{n^2-1}}(\alpha)$ at any $\alpha\geq \frac{c_{i,n}}{d_{i,n}}$, and in particular at $\alpha=\frac{\hat{c}_{i,n}}{\hat{d}_{i,n}}$.  However, (\ref{AhataboveA}) shows that this lower bound for  $C_{\sqrt{n^2-1}}(\hat{c}_{i,n}/\hat{d}_{i,n})$ coming from $A_{i,n}$ is smaller than the lower bound $\Gamma_{\frac{\hat{c}_{i,n}}{\hat{d}_{i,n}},\sqrt{n^2-1}}(\hat{A}_{i,n})$ coming from our new class $\hat{A}_{i,n}$.  

 Using (\ref{iiplus1}) and the fact that $a_{i,n}+\sqrt{n^2-1}b_{i,n}=\omega_{n}^{i}$, we have \begin{equation}\label{midstep} C_{\sqrt{n^2-1}}(\alpha)\geq \max\left\{c_{i,n}\omega_{n}^{-i},\Gamma_{\alpha,\sqrt{n^2-1}}(\hat{A}_{i,n}),d_{i+1,n}\omega_{n}^{-i-1}\alpha\right\} \end{equation} for $\alpha\in [c_{i,n}/d_{i,n},c_{i+1,n}/d_{i+1,n}]$.

\begin{prop}\label{hatbeatsvol}
Denote the right hand side of (\ref{midstep}) by $\hat{B}_{i,n}(\alpha)$.  Then for all $\alpha\in [\frac{c_{i,n}}{d_{i,n}},\frac{c_{i+1,n}}{d_{i+1,n}}]$ we have $\hat{B}_{i,n}(\alpha)>\sqrt{\frac{\alpha}{2\sqrt{n^2-1}}}$.  Thus $C_{\sqrt{n^2-1}}$ is strictly greater than the volume bound throughout $[\frac{c_{i,n}}{d_{i,n}},\frac{c_{i+1,n}}{d_{i+1,n}}]$.
\end{prop}

\begin{proof} We claim that: 
\begin{itemize} \item[(i)] \[ (c_{i,n}\omega_{n}^{-i})^2>\frac{\hat{c}_{i,n}}{2\sqrt{n^2-1}\hat{d}_{i,n}},\] and 
\item[(ii)] \[ \left(\frac{\hat{c}_{i,n}}{\hat{a}_{i,n}+\sqrt{n^2-1}\hat{b}_{i,n}}\right)^2 > \frac{\hat{c}_{i,n}\omega_{n}^{i+1}}{2\sqrt{n^2-1}d_{i+1,n}(\hat{a}_{i,n}+\sqrt{n^2-1}\hat{b}_{i,n})}.\]
\end{itemize}

To prove (i), first of all note that a routine computation shows that $c_{i,n}
^{2}=\frac{a_{2i+1,n}-1}{n-1}$, and so (\ref{gcdhat}) shows that \[ c_{i,n}^{2
}\hat{d}_{i,n}=b_{2i,n}\hat{c}_{i,n}+1.\]  Thus (i) is equivalent to the 
statement that \[ 2\sqrt{n^2-1}\omega_{n}^{-2i}(b_{2i,n}\hat{c}_{i,n}+1) > 
\hat{c}_{i,n}.\]  Since $2\sqrt{n^2-1}b_{2i,n}\omega_{n}^{-2i}=1-\omega_{n}^{-
4i}$, this in turn is equivalent to the statement that $-\hat{c}_{i,n}\omega_{
n}^{-4i}+2\sqrt{n^2-1}\omega_{n}^{-2i}>0$, \emph{i.e.} that \begin{equation}
\label{chatsize} \hat{c}_{i,n}<2\sqrt{n^2-1}\omega_{n}^{2i}.\end{equation}  
We find \begin{align*}
\hat{c}_{i,n} &= \frac{1}{2\sqrt{n^2-1}}\left(\omega_{n}^{2i+2}-\omega_{n}^{-2
i-2}\right)-\frac{1}{2(n-1)}\left(\omega_{n}^{2i+1}-2+\omega_{n}^{-2i-1}\right)
\\ &= \frac{\omega_{n}^{2i}}{2\sqrt{n^2-1}}\left(\omega_{n}^{2}-\sqrt{\frac{n+
1}{n-1}}\omega_{n}\right)+\frac{1}{n-1}-\frac{\omega_{n}^{-2i}}{2\sqrt{n^2-1}}
\left(\omega_{n}^{-2}+\sqrt{\frac{n+1}{n-1}}\omega_{n}^{-1}\right) 
\\ &<  \frac{\omega_{n}^{2i}}{2\sqrt{n^2-1}}\left(\omega_{n}^{2}-\sqrt{\frac{n
+1}{n-1}}\omega_{n}+2\omega_{n}^{-2i}\sqrt{\frac{n+1}{n-1}}\right).\end{align*}

Now \begin{align*}
\omega_{n}^{2}&-\sqrt{\frac{n+1}{n-1}}\omega_{n}+2\omega_{n}^{-2i}\sqrt{\frac{n+1}{n-1}} \leq \omega_{n}^{2}-\frac{\sqrt{n+1}(\omega_{n}-2)}{\sqrt{n-1}}
\\ &= 2n^2-1+2n\sqrt{n^2-1}-\sqrt{\frac{n+1}{n-1}}(n-2+\sqrt{n^2-1}) 
\\ &= (2n^2-1)-(n+1)+2n\sqrt{n^2-1}-(n-2)\sqrt{\frac{n+1}{n-1}}
\\ &\leq (2n^2-4)+2n\sqrt{n^2-1}<4(n^2-1)
\end{align*} since $n\geq 2$. Thus \[ \hat{c}_{i,n}<\frac{\omega_{n}^{2i}}{2\sqrt{n^2-1}}\cdot 4(n^2-1)=2\sqrt{n^2-1}\omega_{n}^{2i},\] proving (\ref{chatsize}) and hence proving claim (i) at the start of the proof.

As for claim (ii), that claim is equivalent to the statement that \[ \frac{\hat{c}_{i,n}}{\hat{a}_{i,n}+\sqrt{n^2-1}\hat{b}_{i,n}}>\frac{a_{i+1,n}+\sqrt{n^2-1}b_{i+1,n}}{2\sqrt{n^2-1}d_{i+1,n}}.\] But this latter inequality follows immediately from (\ref{lastjunction}).

We now deduce the proposition from claims (i) and (ii).  Since by definition $\hat{B}_{i,n}(\alpha)\geq c_{i,n}\omega_{n}^{-i}$ for all $i$, claim (i) shows that $\hat{B}_{i,n}(\alpha)\geq \sqrt{\frac{\alpha}{2\sqrt{n^2-1}}}$ for all $\alpha\leq \frac{\hat{c}_{i,n}}{\hat{d}_{i,n}}$. Next let \[ \alpha_0=\frac{\hat{c}_{i,n}\omega_{n}^{i+1}}{d_{i+1,n}(\hat{a}_{i,n}+\sqrt{n^2-1}\hat{b}_{i,n})}.\]  Then (\ref{AhataboveAplus}) implies that $\alpha_0> \frac{\hat{c}_{i,n}}{\hat{d}_{i,n}}$, and claim (ii) shows that $\frac{\hat{c}_{i,n}}{\hat{a}_{i,n}+\sqrt{n^2-1}\hat{b}_{i,n}}>\sqrt{\frac{\alpha_0}{2\sqrt{n^2-1}}}$.  For all  $\alpha\in [\frac{\hat{c}_{i,n}}{\hat{d}_{i,n}},\alpha_0]$ we then have \[ \hat{B}_{i,n}(\alpha)\geq \frac{\hat{c}_{i,n}}{\hat{a}_{i,n}+\sqrt{n^2-1}\hat{b}_{i,n}}>\sqrt{\frac{\alpha_0}{2\sqrt{n^2-1}}}\geq \sqrt{\frac{\alpha}{2\sqrt{n^2-1}}}.\]  Finally, $\alpha_0$ was chosen to have the property that $\frac{\hat{c}_{i,n}}{\hat{a}_{i,n}+\sqrt{n^2-1}\hat{b}_{i,n}}=d_{i+1,n}\omega^{-i-1}_{n}\alpha_0$, so we have \[ d_{i+1,n}\omega_{n}^{-i-1}\alpha >\sqrt{\frac{\alpha}{2\sqrt{n^2-1}}} \mbox{ for }\alpha=\alpha_0, \mbox{and hence also for all }\alpha>\alpha_0.\]  But by definition $\hat{B}_{i,n}(\alpha)\geq   d_{i+1,n}\omega_{n}^{-i-1}\alpha $ for all $\alpha$.  So we have shown that $\hat{B}_{i,n}(\alpha)$ is strictly greater than the volume bound $\sqrt{\frac{\alpha}{2\sqrt{n^2-1}}}$ for all $\alpha$ in each of the three intervals $[\frac{c_{i,n}}{d_{i,n}},\frac{\hat{c}_{i,n}}{\hat{d}_{i,n}}],[\frac{\hat{c}_{i,n}}{\hat{d}_{i,n}},\alpha_0],$ and $[\alpha_0,\frac{c_{i+1,n}}{d_{i+1,n}}],$ completing the proof.
\end{proof}

Applying Brahmagupta moves, one obtains quasi-perfect classes $\hat{A}_{i,n}^{(k)}$ for all $i,k\geq 0$ and $n\geq 2$.
Proposition \ref{hatbeatsvol} shows that, for $k=0$, $\sup_i\max\{\Gamma_{\alpha,L_{n,k}}(A_{i,n}^{(k)}),\Gamma_{\alpha,L_{n,k}}(\hat{A}_{i,n}^{(k)})\}$ exceeds the volume bound for all $\alpha\in\cup_i\left[\frac{c_{i,n,k}}{d_{i,n,k}},\frac{c_{i+1,n,k}}{d_{i+1,n,k}}\right]$. We suspect that the same inequality holds for all $k$, and computer calculations following the same strategy as those  described at the end of 
Section \ref{undervolsect} confirm that it holds whenver $n,k\leq 100$.

\subsection{Connecting the staircases}\label{connect}

The Frenkel-M\"uller classes that were featured in Section \ref{fmsect} fit in to our collections of classes $A_{i,n}^{(k)}$ and $\hat{A}_{i,n}^{(k)}$. Specifically we have \begin{align*} A_{0,n}^{(k)} &= \left(\frac{1+1}{2},\frac{1-1}{2};\mathcal{W}(1+0,1-0)\right)^{(k)}
\\ &= \left(\frac{H_{2k}+1}{2},\frac{H_{2k}-1}{2};\mathcal{W}(H_{2k}+P_{2k},H_{2k}-P_{2k})\right)=FM_{2k-1} \end{align*} (independently of $n$), and \begin{align}\nonumber \hat{A}_{0,2}^{(k)}&=\left(\frac{2+0}{2},\frac{2-0}{2};\mathcal{W}(2+1,2-1)\right)^{(k)}
\\ &= \left(H_{2k}+P_{2k},H_{2k}+P_{2k};\mathcal{W}(3H_{2k}+4P_{2k},H_{2k})\right)=FM_{2k}. \label{evenfm}\end{align}  These are not the only ways of expressing the $FM_m$ as $A_{i,n}^{(k)}$ or $\hat{A}_{i,n}^{(k)}$; for instance because $\hat{A}_{0,3}=A_{1,2}=(2,1;\mathcal{W}(5,1))=FM_1=A_{0,n}^{(1)}$ we can write \begin{equation}\label{oddfm} FM_{2k+1}=A_{0,n}^{(k+1)}=A_{1,2}^{(k)}=\hat{A}_{0,3}^{(k)}.\end{equation}


Theorem \ref{fmsup} shows that for $\alpha\leq 3+2\sqrt{2}$ we  have $C_{\beta}(\alpha)=\sup\{\Gamma_{FM_n}(\alpha,\beta)|n\in\N\cup\{-1\}\}$ (and if $\beta>1$ then  Proposition \ref{supreduce} reduces this supremum to a maximum over a finite set).  Meanwhile for the specific values $\beta=L_{n,k}$, (\ref{stairnhd}) shows that \[ C_{L_{n,k}}(\alpha)=\Gamma_{\alpha,L_{n,k}}(A_{i,n}^{(k)}) \mbox{ for all }\alpha\mbox{ in a neighborhood of }\frac{c_{i,n,k}}{d_{i,n,k}}.\]  We have seen that, at least for $k=0$ and likely for all $k$, the equality $C_{\beta}(\alpha)=\sup_j\Gamma_{\alpha,L_{n,k}}(A_{j,n}^{(k)})$ does not persist throughout the interval $\left[\frac{c_{i,n,k}}{d_{i,n,k}},\frac{c_{i+1,n,k}}{d_{i+1,n,k}}\right]$; indeed Corollary \ref{undervolcor} and Proposition \ref{hatbeatsvol} show that (again at least for $k=0$) there is a subinterval of this interval on which $\hat{A}_{i,n}^{(k)}$ gives a stronger lower bound than do any of the $A_{j,n}^{(k)}$.  However we conjecture that this is all that needs to be taken into account to fully describe our infinite staircase:

\begin{conj}\label{fillconj}
Let $n\geq 2$ and $k\geq 0$.  Then for all $\alpha\in [\frac{c_{0,n,k}}{d_{0,n,k}},S_{n,k}]$ we have \[ C_{L_{n,k}}(\alpha)=\sup\left\{\Gamma_{\alpha,L_{n,k}}(A)\left|A\in\cup_{i=0}^{\infty}\{A_{i,n}^{(k)},\hat{A}_{i,n}^{(k)}\}\right.\right\}.\]
\end{conj}

Let us compare the behavior of $C_{L_{n,k}}$ on $[\frac{c_{0,n,k}}{d_{0,n,k}},S_{n,k}]$  predicted by this conjecture to the behavior given by Theorem \ref{fmsup} on $[1,3+2\sqrt{2}]$.   First, notice that since (\ref{oddfm}) gives \[ A_{0,n}^{(k)}=FM_{2k-1}=\left(\frac{H_{2k}+1}{2},\frac{H_{2k}-1}{2};\mathcal{W}(P_{2k+1},P_{2k-1})\right),\]  the left endpoint $\frac{c_{0,n,k}}{d_{0,n,k}}$ of the interval in Conjecture \ref{fillconj} is equal to $\frac{P_{2k+1}}{P_{2k-1}}$, which is less than $3+2\sqrt{2}$ by Proposition \ref{orderratios}. If $n\geq 4$, then we can conclude that the first step in our infinite staircase for $C_{L_{n,k}}$ coincides with the final step in (what remains of) the Frenkel-M\"uller staircase, since Proposition \ref{Lsize} shows that in this case $L_{n,k}\in (b_{2k},b_{2k-1})$, which by Proposition \ref{supreduce} implies that the last step remaining in the Frenkel-M\"uller staircase is the one determined by $FM_{2k-1}$.  For the case $n=3$, since $A_{0,3}^{(k)}=FM_{2k-1}$ and $\hat{A}_{0,3}^{(k)}=FM_{2k+1}$ by (\ref{oddfm}), referring again to Propositions \ref{Lsize} and \ref{supreduce} we see that the first two steps in the staircase described by 
Conjecture \ref{fillconj} are $\Gamma_{\cdot,L_{3,k}}(FM_{2k-1})$ and $\Gamma_{\cdot,L_{3,k}}(FM_{2k+1})$, which  coincide with the last two steps of the Frenkel-M\"uller staircase.  Finally in the case $n=2$ we see that $A_{0,2}^{(k)}=FM_{2k-1}$, $\hat{A}_{0,2}^{(k)}=FM_{2k}$, and $A_{1,2}^{(k)}=FM_{2k+1}$, and so the first three steps in the staircase from Conjecture \ref{fillconj} coincide with the last three steps of the Frenkel-M\"uller staircase.  In all cases Theorem \ref{fmsup} therefore implies that the formula in Conjecture \ref{fillconj} holds for all $\alpha\in [\frac{c_{0,n,k}}{d_{0,n,k}},3+2\sqrt{2}]$.

\end{document}